\documentclass[12pt]{amsart}
\usepackage{amssymb}
\usepackage[shortlabels]{enumitem}
\usepackage[all]{xy}
\usepackage{xcolor}
\usepackage[margin=1in]{geometry}
\usepackage{graphicx}
\usepackage{stmaryrd}
\usepackage{hyperref}
\usepackage{mathrsfs}
\usepackage{booktabs}
\usepackage{eucal}
\usepackage[normalem]{ulem}
\hypersetup{bookmarksdepth=2}
\hypersetup{colorlinks=true}
\hypersetup{linkcolor=blue}
\hypersetup{citecolor=blue}
\hypersetup{urlcolor=blue}
\usepackage{tikz}

\numberwithin{equation}{section}
\newtheorem{theorem}[equation]{Theorem}
\newtheorem{proposition}[equation]{Proposition}
\newtheorem{lemma}[equation]{Lemma}
\newtheorem{corollary}[equation]{Corollary}

\newtheorem{maintheorem}{Theorem}
\newtheorem{maincorollary}[maintheorem]{Corollary}

\theoremstyle{definition}
\newtheorem{remark}[equation]{Remark}
\newtheorem{example}[equation]{Example}
\newtheorem{definition}[equation]{Definition}

\newcommand{\bA}{\mathbf{A}}
\newcommand{\fA}{\mathfrak{A}}
\newcommand{\fB}{\mathfrak{B}}

\newcommand{\cC}{\mathcal{C}}
\newcommand{\fC}{\mathfrak{C}}
\newcommand{\bE}{\mathbf{E}}
\newcommand{\cE}{\mathcal{E}}
\newcommand{\sE}{\mathscr{E}}
\newcommand{\bF}{\mathbf{F}}
\newcommand{\cF}{\mathcal{F}}
\newcommand{\fF}{\mathfrak{F}}
\newcommand{\bG}{\mathbf{G}}

\newcommand{\bN}{\mathbf{N}}

\newcommand{\cP}{\mathcal{P}}
\newcommand{\fP}{\mathfrak{P}}

\newcommand{\bQ}{\mathbf{Q}}

\newcommand{\fQ}{\mathfrak{Q}}

\newcommand{\bS}{\mathbf{S}}

\newcommand{\fS}{\mathfrak{S}}

\newcommand{\bZ}{\mathbf{Z}}

\newcommand{\fa}{\mathfrak{a}}

\renewcommand{\AA}{\mathbb{A}}
\newcommand{\BB}{\mathbb{B}}
\newcommand{\CC}{\mathbb{C}}

\newcommand{\arxiv}[1]{\href{http://arxiv.org/abs/#1}{{\tiny\tt arXiv:#1}}}
\newcommand{\DOI}[1]{\href{http://doi.org/#1}{\color{purple}{\tiny\tt DOI:#1}}}
\newcommand{\defn}[1]{\emph{#1}}

\let\ul\underline
\let\ol\overline
\renewcommand{\phi}{\varphi}
\DeclareMathOperator{\Aut}{Aut}
\DeclareMathOperator{\Max}{Max}
\DeclareMathOperator{\Min}{Min}
\DeclareMathOperator{\im}{im}
\DeclareMathOperator{\Sym}{Sym}
\DeclareMathOperator{\uPerm}{\ul{Perm}}
\DeclareMathOperator{\Iso}{Iso}
\DeclareMathOperator{\Quot}{Quot}

\newcommand{\Aff}{\mathbf{Aff}}
\newcommand{\bone}{\mathbf{1}}

\newcommand{\bzero}{\mathbf{0}}
\renewcommand{\emptyset}{\varnothing}
\let\lbb\llbracket
\let\rbb\rrbracket

\title{Measures on partial orders}

\author{Andrew Snowden}
\address{Department of Mathematics, University of Michigan, Ann Arbor, MI, USA}
\email{\href{mailto:asnowden@umich.edu}{asnowden@umich.edu}}
\urladdr{\url{http://www-personal.umich.edu/~asnowden/}}
\thanks{The author was supported by NSF grant DMS-2301871.}

\date{September 1, 2026}

\begin{document}

\begin{abstract}
We determine the measures (in the sense of Harman--Snowden) on the Fra\"iss\'e class of partially ordered sets: the space of measures is a union of a plane, eight lines, and 15 isolated points. This is the first case where the space is not equidimensional, and the first primitive case in which it has dimension at least two.
\vskip.5\baselineskip\noindent
ChatGPT was used to obtain many arguments. The writing was done entirely by the author.
\end{abstract}

\maketitle
\tableofcontents

\section{Introduction}

In joint work with Harman \cite{repst}, we introduced the notion of measure on a Fra\"iss\'e class. We showed that each measure leads to an additive tensor category, and we used this idea to construct interesting new pre-Tannakian categories, such as the Delannoy category \cite{line} and the arboreal categories \cite{arboreal}. Constructing and classifying measures on Fra\"iss\'e classes is therefore an important problem. Currently, there are very few general results, and only a modest list of solutions for specific classes. We solve the problem for a new case: the class of partially ordered sets. The answer turns out to be more complicated than those in previously studied cases, and it exhibits some new phenomena.

\subsection{Background}

Let $\fF$ be a Fra\"iss\'e class, i.e., a class of finite relational structures satisfying the Fra\"iss\'e axioms (\S \ref{ss:fraisse}). A \defn{measure} on $\fF$ valued in a field $k$ is a rule $\nu$ that assigns to each embedding $\alpha \colon P \to Q$ in $\fF$ a value $\nu(\alpha)$ in $k$ such that three axioms hold:
\begin{enumerate}
\item $\nu(\alpha)=1$ if $\alpha$ is an isomorphism.
\item $\nu(\beta \circ \alpha) = \nu(\beta) \cdot \nu(\alpha)$, when defined.
\item Given embeddings $\alpha \colon P \to Q$ and $\beta \colon P \to P'$, let $Q'_1, \ldots, Q'_r$ be the various amalgamations of $Q$ and $P'$ over $P$, and let $\alpha'_i \colon P' \to Q'_i$ be the given embedding. Then $\nu(\alpha) = \nu(\alpha'_1) + \cdots + \nu(\alpha'_r)$.
\end{enumerate}
Amalgamations are reviewed in \S \ref{ss:fraisse}, and the definition is discussed in more detail in \S \ref{ss:meas}. The measures on $\fF$ are the $k$-points of a scheme (typically a variety), which we call the space of measures. Some typical examples of these spaces are given in Table~\ref{t:examples}.

Suppose $\nu$ is a measure on $\fF$. A \defn{marked structure} is a pair $(P,x)$ where $P$ is a structure in $\fF$ and $x$ is an element of $P$. We let $P^x=P \setminus \{x\}$, with the induced structure, and put
\begin{displaymath}
\nu(P,x) = \nu(P^x \to P).
\end{displaymath}
Since every embedding in $\fF$ factors into a sequence of one-point extensions, axiom (b) implies that $\nu$ is determined by the values $\nu(P,x)$. Moreover, it is known how to characterize when a rule defined on marked structures comes from a measure (\S \ref{ss:onepoint}). This is an important perspective that will be used throughout the paper.
\begin{table}
\caption{Some examples of spaces of measures.}
\label{t:examples}
\centering
\begin{tabular}{lll}
\toprule
$\fF$ & space of measures & reference \\
\midrule
sets & line & \cite[\S 14]{repst} \\
equivalence relations & plane & --- \\
trees & two lines meeting at a point & \cite{arboreal} \\
planar boron trees & empty & \cite[\S 2.7]{Nekrasov} \\
total orders & four points & \cite[\S 16]{repst} \\
permutations & 37 points & \cite{homoperm} \\
\bottomrule
\end{tabular}
\end{table}

\subsection{Main results}

The purpose of this paper is to determine the measures on the class $\fP$ of finite partially ordered sets. We now state our main result in this direction. Let $J_{p,q}$ be the poset with underlying set $\{a_1, \ldots, a_p, b_1, \ldots, b_q, x\}$ and order generated by the relations $a_i < x$ and $x < b_j$ for all $i$ and $j$; thus $\{a_1, \ldots, a_p\}$ and $\{b_1, \ldots, b_q\}$ are antichains. We mark $J_{p,q}$ using the point $x$, which we sometimes denote by $\ast$.

\begin{maintheorem} \label{mainthm}
Let $\nu$ be a measure for $\fP$. Define $\nu_{p,q}=\nu(J_{p,q}, \ast)$. Then $\nu$ is completely determined by the $3 \times 3$ matrix $N=(\nu_{p,q})_{0 \le p,q \le 2}$. The matrices that arise in this way are exactly the ones listed in Table~\ref{t:matrix}, for any $s$ and $t$, and their transposes.
\end{maintheorem}

We actually give a formula for $\nu$ in terms of the matrix $N$ (see \S \ref{ss:quadthm}). By the theorem, we can identify the collection of all measures with a subset of the space $\bA^{3 \times 3}$ of $3 \times 3$ matrices; it is in fact a closed subvariety. It is described geometrically as follows\footnote{Corollary~B, we only describe the reduced subscheme. We do not prove that the scheme is reduced.}.

\begin{maincorollary}
The space of measures for $\fP$ is the union of one plane, eight lines, and 15 isolated points. The plane and three of the lines intersect at a common point, and there are no other intersections.
\end{maincorollary}

\begin{proof}
The matrix $\AA$ gives the plane, as $s$ and $t$ vary. The matrices $\BB_i$, for $1 \le i \le 5$, give five lines, and the transposes of $\BB_2$, $\BB_3$, and $\BB_4$ give three additional lines; $\BB_1$ and $\BB_5$ are symmetric. The intersection point lies on $\AA$, $\BB_1$, $\BB_2$, and $\BB_2^{\top}$ when $s=t=0$. The matrices $\CC_i$ for $1 \le i \le 8$ give eight isolated points, and the matrices $\CC_i^{\top}$ for $1 \le i \le 7$ give seven additional isolated points. The matrix $\CC_8$ is symmetric.
\end{proof}

We find the above results noteworthy because the space of measures is more complicated than the spaces arising in previously studied cases: indeed, this is the first known case where the space is not equidimensional and the first primitive case where it has dimension at least two. (It is easy to give examples of arbitrarily high dimension using Fra\"iss\'e classes that correspond to products or wreath products of oligomorphic groups; this is what we mean by imprimitive. For instance, the class of equivalence relations appearing in Table~\ref{t:examples} corresponds to the wreath product $\fS \wr \fS$ of the infinite symmetric group with itself. The class $\fP$ is primitive.)

\begin{table}
\caption{The matrices used to describe measures; $s$ and $t$ are parameters, and $+$ and $-$ stand for $+1$ and $-1$.}
\label{t:matrix}
\newcommand{\smatrix}[1]{\scalebox{0.7}{$\begin{pmatrix} #1 \end{pmatrix}$}}
\centering
\begin{tabular}{cccccccc}
\toprule \addlinespace[6pt]
& $\AA(s,t)$ & $\BB_1(t)$ & $\BB_2(t)$ & $\BB_3(t)$ & $\BB_4(t)$ & $\BB_5(t)$ \\[3pt]
& \smatrix{t & 0 & 0 \\ 0 & s & s \\ 0 & s & s} &
\smatrix{t & t & t \\ t & t & t \\ t & t & t} &
\smatrix{0 & t & t \\ 0 & t & t \\ 0 & t & t} &
\smatrix{t & 0 & 0 \\ 0 & - & 0 \\ 0 & - & 0} &
\smatrix{t & - & 0 \\ - & - & 0 \\ 0 & - & 0} &
\smatrix{t & - & 0 \\ - & - & 0 \\ 0 & 0 & +} \\[15pt]
$\CC_1$ & $\CC_2$ & $\CC_3$ & $\CC_4$ & $\CC_5$ & $\CC_6$ & $\CC_7$ & $\CC_8$ \\[3pt]
\smatrix{0 & - & 0 \\ 0 & - & 0 \\ 0 & - & 0} &
\smatrix{- & - & 0 \\ - & - & 0 \\ - & - & 0} &
\smatrix{0 & 0 & 0 \\ - & - & 0 \\ - & - & 0} &
\smatrix{0 & 0 & + \\ - & - & 0 \\ 0 & - & 0} &
\smatrix{0 & 0 & + \\ - & - & 0 \\ - & - & 0} &
\smatrix{+ & 0 & + \\ 0 & - & 0 \\ 0 & - & 0} &
\smatrix{0 & - & 0 \\ 0 & - & 0 \\ + & 0 & +} &
\smatrix{+ & 0 & + \\ 0 & - & 0 \\ + & 0 & +} \\[15pt]
\bottomrule
\end{tabular}
\end{table}

\subsection{Other results} \label{ss:further}

We discuss a few other results that we obtain.

\textit{(a) Weak measures.} The measure axioms impose both quadratic equations (in axiom (b)) and linear equations (in axiom (c)). It is often useful to study the linear equations first, as they are easier to solve and still offer strong constraints. We introduce the notion of \defn{weak measure} to carry out this strategy systematically (\S \ref{ss:weak}): this is a rule that assigns to each marked poset a value in $k$ such that an appropriate version of axiom (c) holds\footnote{In fact, a weak measure has one additional parameter.}. The collection of weak measures forms a $k$-vector space. A measure is a weak measure that satisfies additional non-linear equations.

We completely determine the weak measures for $\fP$ (Theorem~\ref{thm:weak}). There is one, denoted $\tau$, related to ``twins'' (\S \ref{ss:twins}) and a family, denoted $\lambda_{p,q}$, related to ``cuts'' (\S \ref{ss:cuts}). These weak measures form a (topological) basis for the space of weak measures. We also give explicit formulas for certain linear combinations of the $\lambda_{p,q}$ (\S \ref{ss:explicit} and \S \ref{ss:mobius}). This detailed description of weak measures is essential in our analysis of measures.

\textit{(b) The monoid of marked posets.} If $(P,x)$ is a marked poset and $Q$ is another poset, then we can ``substitute'' $Q$ into $x$. This is a poset with underlying set $P^x \sqcup Q$. The order restricts to the given ones on $P^x$ and $Q$, and for $a \in P^x$ and $b \in Q$ we have $a \le b$ if and only if $a \le x$, and similarly for $b \le a$. We call this operation \defn{composition}, and denote it by $(P,x) \circ Q$. If $Q$ has a marked point then so does the composition. This gives the set $\cP$ of isomorphism classes of marked posets the structure of a monoid (\S \ref{ss:poset-comp}).

It turns out that the composition operation respects the defining equations of measures, and so the monoid $\cP$ acts on the space of measures. We use this observation in the following way. To prove Theorem~\ref{mainthm}, we must show that each of the matrices in Table~\ref{t:matrix} comes from a measure. We do this directly for the matrices $\BB_1(t)$, $\CC_3$, $\CC_4$, and $\CC_8$ (\S \ref{s:some-meas}). We then use the monoid action to obtain the remaining matrices from these four (\S \ref{ss:remaining}). This is far easier than constructing the remaining measures by hand.

The space of measures is naturally a closed subvariety of the space $\bA^{3 \times 3}$ of matrices. The action of $\cP$ on the space of measures is induced from an affine linear action of $\cP$ on $\bA^{3 \times 3}$. We describe the Zariski closure of the image of $\cP$ in the space affine linear transformations (\S \ref{ss:monoid-closure}). It is five-dimensional and has 30 irreducible components (21 of which are isolated points).

\textit{(c) Support.} Given a measure $\nu$ on a Fra\"iss\'e class $\fF$, the collection $\fF_{\nu}$ of structures $P$ for which $\nu(P) = \nu(\emptyset \to P)$ is non-zero forms a Fra\"iss\'e subclass, called the \defn{support} of $\nu$ (\S \ref{ss:support}). We determine the supports of all the measures on $\fP$ (Theorem~\ref{thm:support}). An important point here is that Schmerl \cite{Schmerl} classified the Fra\"iss\'e subclasses of $\fP$, and there are not too many of them (see Theorem~\ref{thm:schmerl}). We note that there is no measure whose support is all of $\fP$, that is, there is no regular measure (\S \ref{ss:regular}). However, in characteristic~0, every proper Fra\"iss\'e class of $\fP$ does occur as the support of some measure (Corollary~\ref{cor:support}).

\textit{(d) Tensor categories.} Given a measure $\nu$ on $\fP$, the main construction of \cite{repst} (as reformulated in \cite{arboreal}) gives an additive tensor category $\uPerm(\fP, \nu)$. An important question is whether this category can be embedded into a pre-Tannakian category. There is a basic necessary condition: if $\uPerm(\fP, \nu)$ admits a tensor functor to a pre-Tannakian category, then nilpotent endomorphisms in $\uPerm(\fP, \nu)$ must have trace~0. We show that this condition does indeed hold (Corollary~\ref{cor:nilp}), by making use of the semi-simplification (which is intimately tied to support). Thus, as far as we know, every category $\uPerm(\fP, \nu)$ could embed into a pre-Tannakian category. This suggests a potentially rich source of new and interesting pre-Tannakian categories.

\textit{(e) The associated oligomorphic group.} Let $\Omega$ be the Fra\"iss\'e limit of the class $\fP$, i.e., the universal homogeneous partial order, and let $G$ be its automorphism group; the pair $(G, \Omega)$ is an oligomorphic permutation group. In general, measures on a Fra\"iss\'e class and measures on the associated oligomorphic group are closely related, but do not always perfectly align. We show that if $k$ is a field of characteristic~0 then measures for $G$ valued in $k$ correspond bijectively to measures for $\fP$ valued in $k$ (Proposition~\ref{prop:oligo-meas}). The key input is a result about the action of $G$ on $\Omega$ due to Jech \cite[Lemma~7.7]{Jech}. We thus have a complete classification of measures for $G$ in characteristic~0.

\textit{(f) Knop-like measures.} The category $\fP^+$ of finite posets with monotone maps is coregular (Proposition~\ref{prop:coreg}), and so there is a notion of degree function in the sense of Knop \cite{Knop}. Prior results \cite{regcat} show that each degree function has an associated measure on $\fP$; we call the measures arising in this way \defn{Knop-like}. We show that the Knop-like measures are $\BB_5(t)$, $\CC_7$, $\CC_7^{\top}$, and $\CC_8$ (\S \ref{ss:knoplike}). Geometrically, the space of Knop-like measures is a line and three isolated points. 

\subsection{AI usage}

I used ChatGPT Pro 5.6 extensively on this project. In fact, it produced most of the mathematical arguments. I did all of the writing myself, though I did have ChatGPT proofread a draft and provide comments.

Here is an account of how the proof of Theorem~\ref{mainthm} was developed.
\begin{enumerate}[(1)]
\item I first asked ChatGPT to determine the measures on $\fP$. It thought for about 100 minutes, and returned the correct classification. However, it did not have a rigorous argument. It had simply generated all the defining equations for posets of size at most seven and solved them. It noticed that going to larger posets did not change the result, so (reasonably) guessed that this finite approximation was correct.
\item In step (1), ChatGPT had observed that there is a universal formula $\nu(P,x) = b+\sum c_{p,q} \nu(J_{p,q},\ast)$ where the coefficients depend only on $(P,x)$ and not on $\nu$. I asked it to describe the coefficients explicitly, and it came up with formulas involving twins and cuts; in our notation, these formulas are $b=-\tau(P,x)$ and $c_{p,q}=\lambda_{p,q}(P,x)$, where $\tau$ is defined in \S \ref{ss:twins} and $\lambda_{p,q}$ is defined in \S \ref{ss:linear}. These formulas and the notions of twins and cuts turned out to be extremely important for the project.
\item Step (2) shows that the classes $e$ and $j_{p,q} = \lbb J_{p,q}, \ast \rbb$ span what we call $\Theta^1(\fP)$. I next wanted to know that they form a basis. Essentially, this means that the above formula for $\nu(P,x)$ is unique. This is equivalent to $\tau$ and $\lambda_{p,q}$ being weak measures. ChatGPT furnished proofs, which became Lemmas~\ref{lem:weak-2} and~\ref{lem:weak-3}.
\item At this point, we had a complete and rigorous understanding of the linear equations defining measures. I next asked ChatGPT to rigorously prove that the quadratic equations for posets of size at most seven generate the ideal of all quadratic equations. It was unable to do this. I think it was working with the equations symbolically and attempting to use Gr\"obner methods, but that approach is quite difficult.
\item Due to the failure in (4), I set a more modest goal. I asked ChatGPT to show that the matrix $\BB_1(t)$ comes from a measure. From steps (2) and (3), we know that this matrix corresponds to a weak measure $\nu$, and we have an explicit formula for $\nu(P,x)$. For $\nu$ to be a measure, we need its defect to vanish; this is essentially measure axiom (b). From the explicit formula for $\nu$, this amounts to a very concrete statement about posets. ChatGPT was able to prove this; the argument is given in \S \ref{ss:meas1}.
\item I next asked ChatGPT to do the same for the other matrices. It realized that it was enough to handle $\CC_3$, $\CC_4$, and $\CC_8$, and then the other cases could be deduced formally. I recognized that this reduction step could be better phrased in terms of the monoid action discussed in \S \ref{s:monoid}. It then gave proofs for these three matrices; these arguments are given in \S \ref{s:some-meas}.
\item At this point, we had shown that all the matrices in Table~\ref{t:matrix} come from measures. We also had an explicit list of equations (Table~\ref{t:eqs}) showing that if $\nu$ is a measure then the associated matrix $(\nu_{p,q})_{0 \le p,q \le 2}$ must be among these matrices (up to transpose). The one missing piece was the statement that this matrix completely determines $\nu$. From step (2), we know $\nu$ is determined by the $\nu_{p,q}$ with $p,q \in \bN$. ChatGPT had observed empirically that these values stabilize: $\nu_{p,q} = \nu_{\ol{p}, \ol{q}}$, where $\ol{n}=\min(n,2)$. I asked it to prove this rigorously, and it came up with the clever construction of the poset $E$ (\S \ref{ss:jstab}).
\end{enumerate}
After proving Theorem~\ref{mainthm}, I used ChatGPT to help prove the results discussed in \S \ref{ss:further}.

\subsection{Notation}

We list the most important notation:
\begin{description}[align=right,labelwidth=2.5cm,leftmargin=!]
\item[ $k$ ] The coefficient field
\item[ $\fP$ ] The class of finite posets
\item[ $\parallel$ ] The incomparability relation in a poset
\item[ $\lessdot$ ] The cover relation in a poset
\item[ $P_{<x}$ ] The set of elements below $x$ in a poset $P$
\item[ $P_{\parallel x}$ ] The set of elements incomparable to $x$ in a poset $P$ (excluding $x$)
\item[ $P^x$ ] The set $P \setminus \{x\}$
\item[ $\mu$ ] The M\"obius function of a poset
\item[ $\tau$ ] A weak measure related to twins (\S \ref{ss:twins})
\item[ $\lambda_{p,q}$ ] A weak measure related to cuts (\S \ref{ss:linear})
\item[ $J_{p,q}$ ] A poset with $p+q+1$ elements and marked point $\ast$
\item[ $j_{p,q}$ ] The class $\lbb J_{p,q}, \ast \rbb$ in $\Theta^1(\fP)$, or its image in $\Theta(\fP)$
\end{description}

\subsection*{Acknowledgments}

We thank Johannes Flake for helpful discussions.

\section{Measures on Fra\"iss\'e classes} \label{s:fraisse}

In this section, we review the definitions of Fra\"iss\'e classes and measures, and introduce the notion of weak measure\footnote{As far as I know, this idea does not occur in the literature, though it is familiar to experts.}.

\subsection{Fra\"iss\'e classes} \label{ss:fraisse}

A \defn{signature} $\sigma$ is a pair $(I, n)$, where $I$ is an index set and $n(i)$ is a positive integer for each $i \in I$. Fix a signature $\sigma$. A \defn{structure} (with signature $\sigma$) is a set $P$ together with a subset $R^P_i$ of $P^{n(i)}$ for each $i \in I$. If $Q$ is a structure and $P$ is a subset of $Q$, then $P$ carries the \defn{induced substructure} via $R^P_i=R^Q_i \cap P^{n(i)}$. If $P$ and $Q$ are structures, an \defn{embedding} is an injection $\alpha \colon P \to Q$ that maps $P$ isomorphically onto $\alpha(P)$, equipped with the induced substructure.

Let $\alpha \colon P \to Q$ and $\beta \colon P \to P'$ be embeddings of structures. An \defn{amalgamation} of $Q$ and $P'$ over $P$ is a triple $(Q', \alpha', \beta')$ where $Q'$ is a structure, $\alpha' \colon P' \to Q'$ is an embedding, and $\beta' \colon Q \to Q'$ is an embedding such that $\beta' \circ \alpha = \alpha' \circ \beta$ and $Q'=\im(\alpha') \cup \im(\beta')$. There is an evident notion of isomorphism for amalgamations.

\begin{definition} \label{defn:fraisse}
A \defn{Fra\"iss\'e class} is a class $\fF$ of finite structures, all with the same signature, satisfying the following conditions:
\begin{enumerate}
\item $\fF$ is non-empty.
\item $\fF$ is hereditary: if $\alpha \colon P \to Q$ is an embedding of structures and $Q$ belongs to $\fF$ then $P$ belongs to $\fF$.
\item $\fF$ satisfies the amalgamation property: if $\alpha \colon P \to Q$ and $\beta \colon P \to P'$ are embeddings in $\fF$, there is at least one amalgamation $(Q', \alpha', \beta')$ with $Q'$ in $\fF$.
\item $\fF$ has finitely many structures of a given cardinality, up to isomorphism.
\end{enumerate}
We note that (b) implies $\fF$ is closed under isomorphism, and (d) is often relaxed to simply asking that $\fF$ has countably many isomorphism classes.
\end{definition}

Suppose $\fF$ is a Fra\"iss\'e class. Then there exists a countable homogeneous structure $\Omega$ whose age is $\fF$. Here \defn{homogeneous} means that whenever $\alpha$ and $\beta$ are two embeddings of the same finite structure into $\Omega$, there is an automorphism $\gamma$ of $\Omega$ such that $\alpha = \gamma \circ \beta$. The \defn{age} of $\Omega$ is the class of finite structures that embed into $\Omega$. The structure $\Omega$ is unique up to isomorphism and is called the \defn{Fra\"iss\'e limit} of $\fF$. The automorphism group $G$ of $\Omega$ is oligomorphic; this means that $G$ has finitely many orbits on $\Omega^n$ for all $n \ge 0$. Fra\"iss\'e limits and oligomorphic groups will only play an ancillary role in this paper.

Fra\"iss\'e classes abound; examples include sets, total orders, graphs, trees (as in \cite{arboreal}), tournaments. Of course, in all cases we really mean the finite structures of the given type, and it is important to specify exactly what the signature is. Another example the class of partially ordered sets, which is our primary focus. See \cite{Cameron} or \cite{Macpherson} for additional background.

\subsection{Measures} \label{ss:meas}

Fix a Fra\"iss\'e class $\fF$ and a commutative ring $k$. For the remainder of \S \ref{s:fraisse}, all structures will belong to $\fF$. The following definition, which originally appeared in \cite[Definition~6.2]{repst}, is the central concept studied in this paper.

\begin{definition} \label{defn:meas}
A \defn{measure} for $\fF$ valued in $k$ is a rule $\nu$ that assigns to each embedding $\alpha \colon P \to Q$ in $\fF$ a value $\nu(\alpha)$ in $k$ such that the following conditions hold:
\begin{enumerate}
\item $\nu(\alpha)=1$ if $\alpha$ is an isomorphism.
\item $\nu(\beta \circ \alpha) = \nu(\beta) \cdot \nu(\alpha)$, when the composition is defined.
\item Let $\alpha \colon P \to Q$ and $\beta \colon P \to P'$ be embeddings in $\fF$, and let $(Q'_i, \alpha'_i, \beta'_i)$, for $1 \le i \le r$, be the various amalgamations in $\fF$ (up to isomorphism). Then $\nu(\alpha) = \sum_{i=1}^r \nu(\alpha'_i)$.
\end{enumerate}
We note that there are finitely many amalgamations by Definition~\ref{defn:fraisse}(d). For a measure $\nu$ and a structure $P$, we put $\nu(P) = \nu(\emptyset \to P)$.
\end{definition}

There is a universal measure valued in a ring $\Theta(\fF)$. To define $\Theta(\fF)$, we take the polynomial ring in symbols indexed by embeddings in $\fF$, and then quotient by the relations encoding the measure axioms; see \cite[\S 6.4]{repst}. We write $[\alpha]$ for the class of the embedding $\alpha$ in $\Theta(\fF)$. For a structure $P$, we write $[P]$ for the class $[\emptyset \to P]$. ``Universal'' means that giving a $k$-valued measure $\nu$ is equivalent to giving a ring homomorphism $\phi \colon \Theta(\fF) \to k$; precisely, $\nu$ and $\phi$ correspond if $\nu(\alpha)=\phi([\alpha])$ for all $\alpha$.

\subsection{Marked structures} \label{ss:onepoint}

A \defn{marked structure} is a pair $(P,x)$ where $P$ is a structure in $\fF$ and $x \in P$. Let $P^x = P \setminus \{x\}$ with the induced substructure. Given a measure $\nu$, define
\begin{displaymath}
\nu(P,x) = \nu(P^x \to P).
\end{displaymath}
Similarly, write $[P,x]$ for the class $[P^x \to P]$ in $\Theta(\fF)$. A \defn{one-point extension} is an embedding $P \to Q$ where $\vert Q \vert = \vert P \vert + 1$. Any one-point extension is isomorphic to $P^x \to P$ for some marked structure $(P,x)$. Since any embedding can be factored into a sequence of one-point extensions, it follows from Definition~\ref{defn:meas}(b) that $\nu$ is determined by its values $\nu(P,x)$ on marked structures. We now explain how to view measures entirely from this perspective.

A \defn{pre-amalgamation}\footnote{Pre-amalgamations are essentially equivalent to the L-data used in \cite{arboreal}, but slightly neater.} is a triple $(P,x,y)$ where $P$ is a finite set and $x$ and $y$ are distinct elements of $P$, such that $P^x$ and $P^y$ carry structures in $\fF$ that induce the same structure on $P^{x,y}=P \setminus \{x,y\}$. Let $(P,x,y)$ be a pre-amalgamation. We analyze the amalgamations of $P^x$ and $P^y$ over $P^{x,y}$ in $\fF$. In an amalgamation, the points $x$ and $y$ can remain distinct or become identified; we refer to these as \defn{proper} and \defn{improper} amalgamations, respectively. Every proper amalgamation is uniquely isomorphic to a structure with underlying set $P$ such that the induced structures on $P^x$ and $P^y$ are the given ones; we always use these representatives of the isomorphism classes. We now turn to improper amalgmations. There is a unique bijection $\gamma \colon P^x \to P^y$ that is the identity on $P^{x,y}$ and takes $y$ to $x$. If $\gamma$ is an isomorphism of structures then $P^y$ itself (equipped with the identity map $P^y \to P^y$ and $\gamma \colon P^x \to P^y$) is the unique improper amalgamation. If $\gamma$ is not an isomorphism then no improper amalgamation exists. We define $\epsilon(P,x,y)$ to be~1 when an improper amalgamation exists and~0 otherwise.

Now, suppose $\nu$ is a measure valued in a ring $k$. The rule $(P,x) \mapsto \nu(P,x)$ has the following properties:
\begin{enumerate}
\item If $(P,x)$ is isomorphic to $(P',x')$ then $\nu(P,x)=\nu(P',x')$.
\item Given a structure $P$ in $\fF$ and distinct points $x$ and $y$, we have
\begin{displaymath}
\nu(P,x) \nu(P^x, y) = \nu(P, y) \nu(P^y, x).
\end{displaymath}
\item Given a pre-amalgamation $(P,x,y)$ with proper amalgamations $Q_1, \ldots, Q_r$, we have
\begin{displaymath}
\nu(P^y, x) = \epsilon(P,x,y) + \sum_{i=1}^r \nu(Q_i, x).
\end{displaymath}
\end{enumerate}
Condition (b) comes from Definition~\ref{defn:meas}(b) by factoring the embedding $P^{x,y} \to P$ in two ways. Condition (c) comes from Definition~\ref{defn:meas}(c). If an improper amalgamation exists, then one of the embeddings $\alpha'_i$ in Definition~\ref{defn:meas}(c) will be an isomorphism, and then $\nu(\alpha'_i)=1=\epsilon(P,x,y)$.

The following result shows that these conditions are sufficient to define a measure.

\begin{proposition} \label{prop:meas-marked}
Let $\nu'$ be a rule that assigns to each marked structure $(P,x)$ a value $\nu'(P,x)$ in the ring $k$ such that (a), (b), and (c) above hold. Then there exists a unique measure $\nu$ valued in $k$ such that $\nu'(P,x) = \nu(P,x)$ for all $(P,x)$.
\end{proposition}

\begin{proof}
This is simply a rephrasing of \cite[Proposition~2.7]{arboreal}.
\end{proof}

In what follows, we identify measures with rules defined on marked structures satisfying (a), (b), and (c).

\subsection{Weak measures} \label{ss:weak}

The measure axioms have both linear and quadratic equations. It is often advantageous to begin with the linear equations and temporarily ignore the quadratic ones. We now introduce a concept that formalizes this approach.

\begin{definition} \label{defn:weak}
A \defn{weak measure} on $\fF$ valued in an abelian group $M$ is a rule $\nu$ that assigns a value $\nu(P,x) \in M$ to each marked structure $(P,x)$, and a value $\nu(e) \in M$ to the formal symbol $e$, such that the following conditions hold:
\begin{enumerate}
\item If $(P,x)$ is isomorphic to $(P',x')$ then $\nu(P,x) = \nu(P',x')$.
\item Given a pre-amalgamation $(P,x,y)$ with proper amalgamations $Q_1, \ldots, Q_r$, we have
\begin{displaymath}
\nu(P^y,x) = \epsilon(P,x,y) \nu(e) + \sum_{i=1}^r \nu(Q_i, x).
\end{displaymath}
\end{enumerate}
\end{definition}

We make a few comments on the definition. The weak measure axioms are linear, which is why weak measures are allowed to be valued in an abelian group; in contrast, the measure axioms involve both linear and quadratic equations, and thus only make sense when the values belong to a ring. The symbol $e$ is introduced to homogenize the linear relations. Homogeneity ensures that the weak measures valued in a fixed abelian group $M$ themselves form an abelian group, under pointwise addition. A measure will satisfy $\nu(e)=1$, where~1 is the unit of the coefficient ring.

Let $\nu$ be a weak measure valued in a commutative ring $k$. For a structure $P$ and distinct points $x$ and $y$ in $P$, define the \defn{defect}
\begin{displaymath}
\delta_{\nu}(P,x,y) = \nu(P,x) \nu(P^x, y) - \nu(P, y) \nu(P^y, x).
\end{displaymath}
It follows from Proposition~\ref{prop:meas-marked} that $\nu$ is a measure\footnote{Technically, we should say ``the rule $(P,x) \mapsto \nu(P,x)$ comes from a measure.''} if and only if $\nu(e)=1$ and $\delta_{\nu}$ vanishes identically. Note that the defect is antisymmetric in $x$ and $y$, that is, $\delta_{\nu}(P,y,x) = - \delta_{\nu}(P,x,y)$. The following proposition shows that defects satisfy a version of the amalgamation equation. See Proposition~\ref{prop:cover-defect} for an illustration of how this can be useful.

\begin{proposition} \label{prop:defect-sum}
Let $\nu$ be as above. Let $(P,x,y)$ be a pre-amalgamation, and let $Q_1, \ldots, Q_r$ be the various proper amalgamations. Then
\begin{displaymath}
\sum_{i=1}^r \delta_{\nu}(Q_i, x, y) = 0.
\end{displaymath}
\end{proposition}

\begin{proof}
Let $\epsilon=\epsilon(P,x,y)$. Note that $Q_i^x=P^x$ is the same structure for all $i$, and similarly $Q_i^y=P^y$. We thus have
\begin{align*}
\sum_{i=1}^r \delta_{\nu}(Q_i, x, y)
&= \nu(P^x,y) \sum_{i=1}^r \nu(Q_i,x) - \nu(P^y,x) \sum_{i=1}^r \nu(Q_i,y) \\
&= \nu(P^x,y) (\nu(P^y,x)-\epsilon \nu(e)) - \nu(P^y,x)(\nu(P^x,y)-\epsilon \nu(e)) \\
&= \epsilon \nu(e) (\nu(P^y,x) - \nu(P^x,y)) = 0.
\end{align*}
In the second step we used Definition~\ref{defn:weak}(b), and in the final step we used the fact that $(P^x,y)$ and $(P^y,x)$ are isomorphic if $\epsilon=1$.
\end{proof}

Let $P$ be a structure and let $x$ and $y$ be distinct elements of $P$. We say that $x$ and $y$ are \defn{separated} if $P$ is the unique amalgamation of $P^x$ and $P^y$ over $P^{x,y}$; in particular, this means there is no improper amalgamation. Separated elements enjoy some useful properties:

\begin{proposition} \label{prop:sep}
Let $P$ be a structure, let $x \ne y$ be separated elements of $P$, and let $\nu$ be a weak measure on $\fF$ valued in an abelian group.
\begin{enumerate}
\item We have $\nu(P, x) = \nu(P^y, x)$.
\item If $\nu$ is valued in a ring then $\delta_{\nu}(P,x,y)=0$.
\item We have $\lbb P, x \rbb = \lbb P^y, x \rbb$ in $\Theta^1(\fF)$ and $[P,x]=[P^y,x]$ in $\Theta(\fF)$.
\end{enumerate}
\end{proposition}

\begin{proof}
(a) follows from applying Definition~\ref{defn:weak}(b) to the pre-amalgamation $(P,x,y)$, while (b) follows from applying Proposition~\ref{prop:defect-sum} to the same pre-amalgamation. The two equalities in (c) follow by applying (a) to the universal weak measure and the universal measure, respectively.
\end{proof}

Similar to measures, there is a universal weak measure valued in an abelian group $\Theta^1(\fF)$. To define $\Theta^1(\fF)$, take the free abelian group on symbols indexed by marked structures, and one additional symbol $e$, and quotient by the relations defining weak measures. We write $\lbb P,x \rbb$ for the class of the marked structure $(P,x)$ in $\Theta^1(\fF)$. Let $\fa$ be the ideal in the symmetric algebra $\Sym{\Theta^1(\fF)}$ generated by $e-1$ and the elements
\begin{displaymath}
\lbb P,x \rbb \cdot \lbb P^x, y \rbb - \lbb P,y \rbb \cdot \lbb P^y, x \rbb
\end{displaymath}
whenever $P$ is a structure and $x$ and $y$ are distinct elements of $P$. It follows from the above discussion that $\Theta(\fF)$ is naturally isomorphic to the quotient of $\Sym{\Theta^1(\fF)}$ by $\fa$.

\subsection{Support} \label{ss:support}

Let $\nu$ be a measure on $\fF$ valued in a field $k$. We define the \defn{support} of $\nu$ to be the collection $\fF_{\nu}$ of all structures $P$ such that $\nu(P)$ is non-zero. This is itself a Fra\"iss\'e class, and $\nu$ restricts to a measure $\ol{\nu}$ on $\fF_{\nu}$ \cite[Proposition~2.18]{arboreal}. We say that $\nu$ is \defn{regular} if $\fF_{\nu}=\fF$, i.e., $\nu(P)$ is non-zero for all structures $P$. By Definition~\ref{defn:meas}(b), $\nu$ is regular if and only if $\nu(\alpha)$ is non-zero for all embeddings $\alpha$. The restricted measure $\ol{\nu}$ is always regular.

\section{Partial orders} \label{s:poset}

In \S \ref{s:poset}, we introduce notation and terminology related to partially ordered sets. We then examine amalgamations of posets and introduce the important ideas of twins and cuts.

\subsection{Basic definitions} \label{ss:poset}

All posets are finite unless otherwise stated. We will freely switch between reflexive partial orders and strict partial orders. Let $P$ be a poset and let $x$ and $y$ be elements of $P$.
\begin{itemize}
\item We write $P^x$ for $P \setminus \{x\}$, and $P^{x,y}$ for $P \setminus \{x,y\}$.
\item We write $x \parallel y$ if $x$ and $y$ are distinct and incomparable.
\item We let $P_{<x}$, $P_{>x}$, and $P_{\parallel x}$ denote the sets of elements $a \in P^x$ such that $a<x$, $a>x$, and $a \parallel x$, respectively. These form a partition: $P^x = P_{<x} \sqcup P_{>x} \sqcup P_{\parallel x}$.
\item We say $x$ is \defn{minimal} if there is no $a$ with $a<x$, i.e., $P_{<x}$ is empty. We say $x$ is \defn{least} if it is the unique minimal element, i.e., $P_{>x}=P^x$. We define \defn{maximal} and \defn{greatest} similarly.
\item We let $\Min(P)$ and $\Max(P)$ denote the sets of minimal and maximal elements. Note that $x$ is least if and only if $\Min(P) = \{x\}$.
\item We say that $y$ is a \defn{cover} of $x$, denoted $x \lessdot y$, if $x<y$ and there is no $a$ with $x<a<y$.
\item We say that $x$ is \defn{isolated} if it is incomparable to all other elements, i.e., $P^x = P_{\parallel x}$.
\item We say that $P$ is an \defn{antichain} if $x \parallel y$ for all $x \ne y$ in $P$.
\item Let $A$ and $B$ be subsets of $P$. We write $A<B$ to mean $a<b$ for all $a \in A$ and $b \in B$. We also write $x<B$ to mean $x<b$ for all $b \in B$, and similarly for $A<x$.
\item We define the \defn{transpose} of $P$ to be the same underlying set equipped with the opposite order.
\end{itemize}
We let $\fP$ be the class of all finite posets. This is well known to be a Fra\"iss\'e class. Except for the amalgamation property, the axioms are clear; we will give a self-contained proof of the amalgamation property in \S \ref{ss:amalg}.

\subsection{Amalgamations} \label{ss:amalg}

Let $(P,x,y)$ be a pre-amalgamation of posets. Recall that this means $P^x$ and $P^y$ have partial orders, and these partial orders have the same restriction to $P^{x,y}$. To extend the partially defined order $<$ on $P$ to a binary relation $\prec$ on $P$, one must specify the truth values of
\begin{displaymath}
x \prec y, \quad y \prec x, \quad x \prec x, \quad y \prec y.
\end{displaymath}
For $\prec$ to be a partial order, at most one of the first two can be true, and neither of the final two can be true. This leaves three extensions, which we denote $<_-$, $<_0$, and $<_+$: for $<_-$, we have $x <_- y$; for $<_+$, we have $y <_+ x$; and for $<_0$, both comparisons are false. Of course, these relations need not be partial orders, as the transitivity can fail. We let $P_-$, $P_0$, and $P_+$ denote the resulting partially ordered sets, when transitivity does hold. The proper amalgamations for $(P,x,y)$ are those $P_-$, $P_0$, and $P_+$ which exist. We now examine existence of these posets in more detail.

\begin{proposition} \label{prop:poset-amalg}
Let $(P,x,y)$ be a pre-amalgamation.
\begin{enumerate}
\item $P_-$ exists if and only if $P^y_{<x} \subset P^x_{<y}$ and $P^x_{>y} \subset P^y_{>x}$.
\item $P_+$ exists if and only if $P^x_{<y} \subset P^y_{<x}$ and $P^y_{>x} \subset P^x_{>y}$.
\item $P_0$ exists if and only if $P^y_{<x} \cap P^x_{>y} = \emptyset$ and $P^y_{>x} \cap P^x_{<y} = \emptyset$.
\end{enumerate}
\end{proposition}

\begin{proof}
(a) Write $<$ for $<_-$. Suppose $P_-$ exists, i.e., $<$ is transitive. We have $x<y$, and so if $a<x$ then $a<y$; similarly, if $y<b$ then $x<b$. This shows that the inclusions in (a) hold.

Conversely, suppose the inclusions in (a) hold. We show $<$ is transitive. Thus suppose we have $a<b$ and $b<c$ with $a,b,c \in P$; we show $a<c$. Since $<$ is transitive on $P^x$ and $P^y$, we have $a<c$ if at most one of $x$ or $y$ occurs in the tuple $(a,b,c)$. Thus suppose that both occur. Since $y \nless x$, the tuple $(a,b,c)$ cannot have the form $(y,x,\ast)$ or $(\ast,y,x)$. We also cannot have $(x,\ast,x)$ since the middle point must be $y$, and we have already excluded this possibility; similarly for $(y,\ast,y)$. This leaves four cases:
\begin{itemize}
\item $(x, \ast, y)$. There is nothing to prove since $a<c$ holds.
\item $(y, \ast, x)$. This case cannot occur: $b$ belongs to $P^x_{>y} \cap P^y_{<x}$, which is contained in $P^x_{>y} \cap P^x_{<y}$ by assumption, which is empty.
\item $(x, y, \ast)$. Since $b<c$, we cannot have $c=x$. Thus $c$ belongs to $P^x_{>y}$. The latter set is contained in $P^y_{>x}$, and so $a<c$, as required.
\item $(\ast, x, y)$. This is similar to the previous case.
\end{itemize}
We thus see that $<$ is transitive, and so $P_-$ exists.

(b) This is just the transpose of (a).

(c) Write $<$ for $<_0$. Suppose $P_0$ exists, i.e., $<$ is transitive. If $x<a$ and $a<y$ then we would have $x<y$, which is a contradiction. Thus $P^y_{>x} \cap P^x_{<y}$ must be empty. The other condition in (c) is similar.

Now suppose the two conditions in (c) hold. We show that $<$ is transitive. Thus suppose we have $a<b$ and $b<c$. We claim that $x$ and $y$ cannot both occur in the tuple $(a,b,c)$; this will prove $a<c$ since $<$ is transitive on $P^x$ and $P^y$. Since $x$ and $y$ are incomparable, they cannot occur consecutively in $(a,b,c)$. We cannot have $(x,\ast,y)$ by the assumption in (c), and, similarly, we cannot have $(y,\ast,x)$. Of course, we also cannot have $(x,\ast,x)$ since the middle point would be $y$, which is not allowed; similarly, we cannot have $(y,\ast,y)$. This establishes the claim, and completes the proof.
\end{proof}

\begin{proposition} \label{prop:poset-amalg-2}
Let $(P,x,y)$ be a pre-amalgamation. Then at least one amalgamation exists. More precisely, suppose $P_0$ fails to exist. Then
\begin{enumerate}
\item If $P^x_{<y} \cap P^y_{>x}$ is non-empty then $P_-$ exists and $P_+$ does not exist.
\item If $P^y_{<x} \cap P^x_{>y}$ is non-empty then $P_+$ exists and $P_-$ does not exist.
\end{enumerate}
\end{proposition}

\begin{proof}
The two statements are exchanged by transpose, so we prove (a). Let $z$ be an element of $P^x_{<y} \cap P^y_{>x}$, meaning we have $x<z$ and $z<y$. If the partially defined order $<$ extends to a partial order $\prec$ on all of $P$ then we have $x \prec y$, and so $\prec$ must be $<_-$; this means that only $P_-$ can exist. If $a \in P^y_{<x}$ then $a<x$ in $P^y$, and so $a<z$ also holds since $<$ is transitive on $P^y$. Since $a<z$ holds in $P^x$ and $<$ is transitive on $P^x$, we have $a<y$, meaning $a \in P^x_{<y}$. Thus $P^y_{<x} \subset P^x_{<y}$. A similar argument shows $P^x_{>y} \subset P^y_{>x}$. Hence $P_-$ exists by Proposition~\ref{prop:poset-amalg}(a).
\end{proof}

\begin{remark}
Let $\alpha \colon P \to Q$ and $\beta \colon P \to P'$ be embeddings of posets. The amalgamation property asserts that there is at least one amalgamation of $Q$ and $P'$ over $P$ in $\fP$. Proposition~\ref{prop:poset-amalg-2} proves this when $\alpha$ and $\beta$ are one-point extensions. The general case can now easily be deduced by factoring $\alpha$ and $\beta$ into one-point extensions.
\end{remark}

Recall that distinct elements $x$ and $y$ of a poset $P$ are \defn{separated} if $P$ is the unique amalgamation of the pre-amalgamation $(P,x,y)$; thus $P$ must be the unique proper amalgamation, and we must also have $\epsilon(P,x,y)=0$. In fact, this has a concrete order-theoretic interpretation:

\begin{proposition} \label{prop:poset-sep}
Let $P$ be a poset and let $x$ and $y$ be distinct elements of $P$.
\begin{enumerate}
\item If $x<y$ then $x$ and $y$ are separated if and only if there exists $a \in P$ with $x<a<y$.
\item If $x \parallel y$ then $x$ and $y$ are separated if and only if both of the sets below are non-empty
\begin{displaymath}
(P_{<x} \setminus P_{<y}) \cup (P_{>y} \setminus P_{>x}), \qquad
(P_{<y} \setminus P_{<x}) \cup (P_{>x} \setminus P_{>y}).
\end{displaymath}
\end{enumerate}
\end{proposition}

\begin{proof}
(a) First observe that $P$ is $P_-$ for the pre-amalgamation $(P,x,y)$. If there is some $a \in P$ such that $x<a<y$ then $P_-$ is clearly the unique amalgamation (in particular, the improper amalgamation does not exist), and so $x$ and $y$ are separated. Conversely, suppose $x$ and $y$ are separated. Since $P_0$ fails to exist, Proposition~\ref{prop:poset-amalg}(c) shows that one of the sets $P^x_{<y} \cap P^y_{>x}$ or $P^y_{<x} \cap P^x_{>y}$ is non-empty. Since $P_-$ does exist, Proposition~\ref{prop:poset-amalg-2} shows that the first set is non-empty, which furnishes the required element $a$.

(b) Now observe that $P$ is $P_0$ for the pre-amalgamation $(P,x,y)$. By Proposition~\ref{prop:poset-amalg}, $P_-$ exists if and only if the first set in the statement of (b) is empty, and $P_+$ exists if and only if the second set is empty. Suppose both sets are non-empty. Then $P_-$ and $P_+$ do not exist; moreover, the non-emptiness of one of these sets shows that the improper amalgamation does not exist ($x$ and $y$ cannot be interchanged). Thus $x$ and $y$ are separated. Conversely, if $x$ and $y$ are separated then $P_-$ and $P_+$ do not exist, and so the two sets are non-empty.
\end{proof}

\subsection{Twins} \label{ss:twins}

Let $P$ be a poset, and let $x \ne y$ be elements of $P$. We say that $x$ and $y$ are \defn{upper twins} if $P_{>x}=P_{>y}$, and \defn{lower twins} if $P_{<x}=P_{<y}$. In either case, $x \parallel y$. We say that $x$ and $y$ are \defn{twins} if they are both upper and lower twins. Equivalently, $x$ and $y$ are twins if the transposition $P \to P$ switching $x$ and $y$ is a poset automorphism. For a marked poset $(P,x)$, we define $\tau(P,x)$ to be the number of twins of $x$ in $P$; we emphasize that $x$ does not count as a twin of itself. We will see (\S \ref{s:linear}) that this is a weak measure (with $\tau(e)=-1$).

We observe the following alternative characterization of lower twins.

\begin{proposition} \label{prop:lowertwin}
Let $P$ be a poset and let $x \ne y$ be elements of $P$. Then $y$ is a lower twin of $x$ if and only if $y \in \Min(P_{\parallel x})$ and $P_{<x}<y$.
\end{proposition}

\begin{proof}
Suppose $y$ is a lower twin of $x$. Then $y \in P_{\parallel x}$. It is a minimal element of this set, for if $z \in P_{\parallel x}$ satisfied $z<y$ then we would have $z \in P_{<y}$ and $z \notin P_{<x}$, a contradiction. We also have $P_{<x}<y$, as this amounts to $P_{<x} \subset P_{<y}$, and in fact these two sets are equal.

Now suppose $y$ is a minimal element of $P_{\parallel x}$ and $P_{<x} \subset P_{<y}$. We must show $P_{<y} \subset P_{<x}$. Thus suppose $a<y$. Since $x \parallel y$, we cannot have $x \le a$. We also cannot have $a \parallel x$, for then $y$ would not be minimal in $P_{\parallel x}$. Thus $a<x$, as required.
\end{proof}

\subsection{Cuts} \label{ss:cuts}

Let $(P,x)$ be a marked poset. A \defn{cut} is a pair $(L,R)$ consisting of subsets $L$ and $R$ of $P_{\parallel x}$ such that
\begin{displaymath}
L \cup P_{<x} < R \cup P_{>x}.
\end{displaymath}
We will see (\S \ref{s:linear}) that cuts can be used to define a whole family of weak measures.

We now describe a construction that explains why the notion of cut is natural. Let $(L,R)$ be a cut for $(P,x)$. Define a new binary relation $\prec$ on $P$ by starting with $<$, adding the relations $a \prec x$ and $x \prec b$ for $a \in L$ and $b \in R$, and taking the transitive closure.

\begin{proposition} \label{prop:cut-order}
The relation $\prec$ is a partial order, and it restricts to $<$ on $P^x$.
\end{proposition}

\begin{proof}
We first show that $\prec$ restricts to $<$ on $P^x$. Thus suppose we have $a,b \in P^x$ with $a \prec b$. Then there is a chain $(c_0, \ldots, c_n)$ in $P$ where $c_0=a$ and $c_n=b$ and for each $0 \le i < n$ we have $c_i<c_{i+1}$, or $c_i \in L$ and $c_{i+1}=x$, or $c_i=x$ and $c_{i+1}\in R$. If $c_i \ne x$ for some $0<i<n$ then we have $a<c_i$ and $c_i<b$ by induction on the length of the chain, and so $a<b$ as desired. Since $x$ cannot appear consecutively in the chain, we reduce to the case of a chain of the form $(a,x,b)$. Necessarily, $a \in L \cup P_{<x}$ and $b \in R \cup P_{>x}$. By definition of cut, we thus have $a<b$, as required.

It remains to show that $\prec$ is a partial order. Transitivity is given. We now show that $\prec$ is irreflexive. Since $\prec$ restricts to $<$ on $P^x$, we have $a \not\prec a$ for $a \in P^x$. Suppose now that $x \prec x$. Then there is a chain $(c_0, \ldots, c_n)$ as above with $c_0=c_n=x$. Necessarily, $n \ge 2$, and $c_1$ and $c_{n-1}$ are not equal to $x$. We must also have $c_1 \in R \cup P_{>x}$ and $c_{n-1} \in L \cup P_{<x}$, and so $c_{n-1}<c_1$ by the cut condition; in particular, $n \ge 3$. The chain thus implies $c_1 \prec c_{n-1}$, and so $c_1<c_{n-1}$ by the previous paragraph, which is a contradiction.  Hence $x \not\prec x$, as required.
\end{proof}

There is a partial converse to the proposition. Suppose $\prec$ is a partial order on $P$ that restricts to $<$ on $P^x$ and is stronger than $<$ (i.e., $a<b$ implies $a \prec b$). Let $L$ be the set of elements $a \in P$ such that $a \prec x$ and $a$ is incomparable to $x$ in the order $<$. Analogously define $R$. Then $(L,R)$ is a cut. This cut has the property that $L$ is downwards closed ($b \in L$ and $a \le b$ implies $a \in L)$ and $R$ is upwards closed. Thus one does not necessarily obtain all cuts in this way.

\section{Weak measures} \label{s:linear}

In \S \ref{s:linear}, we classify the weak measures on $\fP$. We also give explicit formulas for some particularly important weak measures.

\subsection{The main theorem} \label{ss:linear}

Let $(P,x)$ be a marked poset. Given a cut $(L,R)$ for $(P,x)$ (see \S \ref{ss:cuts}), we put
\begin{displaymath}
p(L) = \vert \Max(L \cup P_{<x}) \vert, \qquad q(R) = \vert \Min(R \cup P_{>x}) \vert.
\end{displaymath}
Define
\begin{displaymath}
\lambda(P,x) = \sum (-1)^{\vert L \vert+\vert R \vert} u^{p(L)} v^{q(R)} \in \bZ[u,v]
\end{displaymath}
where the sum is over cuts $(L,R)$, and put $\lambda(e)=0$. We will see (\S \ref{ss:mobius}) that this somewhat mysterious looking formula has a natural M\"obius interpretation. Let $\lambda_{p,q}(P,x)$ be the coefficient of $u^pv^q$ in $\lambda(P,x)$, and also put $\lambda_{p,q}(e)=0$. Recall that $\tau(P,x)$ is the number of twins of $x$ (\S \ref{ss:twins}); also put $\tau(e)=-1$.

Recall that $J_{p,q}$ is the poset with elements $\{a_1, \ldots, a_p, x, b_1, \ldots, b_q\}$ with the order generated by $a_i<x$ and $x<b_j$; we often write $\ast$ for $x$, and use it as the marked point. We let $j_{p,q}$ be the class $\lbb J_{p,q}, \ast \rbb$ in $\Theta^1(\fP)$; we often still write $j_{p,q}$ for its image in $\Theta(\fP)$.

The following theorem is the main result of \S \ref{s:linear}.

\begin{theorem} \label{thm:weak}
We have the following:
\begin{enumerate}
\item $\Theta^1(\fP)$ is a free abelian group with basis $e$ and $j_{p,q}$ for $p,q \in \bN$.
\item Each $\lambda_{p,q}$ and $\tau$ is a $\bZ$-valued weak measure on $\fP$.
\item We have $\lambda_{p,q}(j_{r,s}) = \delta_{p,r} \delta_{q,s}$ and $\tau(j_{r,s})=0$ for all $p$, $q$, $r$, and $s$.
\end{enumerate}
\end{theorem}

The proof is given in \S \ref{ss:weakpf}. We now discuss a few consequences of the theorem. Observe that for any marked poset $(P,x)$, we have $\lambda_{p,q}(P,x)=0$ if $p$ or $q$ exceeds $\vert P \vert$; in particular, $\lambda_{p,q}(P,x)$ is non-zero for only finitely many $(p,q)$. Hence any infinite linear combination of the $\lambda_{p,q}$ is well-defined.

\begin{corollary} \label{cor:weak1}
Let $\nu$ be a weak measure on $\fP$ valued in an abelian group. Put $\nu_e = \nu(e)$ and $\nu_{p,q} = \nu(j_{p,q})$. Then
\begin{displaymath}
\nu = -\nu_e \tau + \sum_{p,q} \nu_{p,q} \lambda_{p,q}.
\end{displaymath}
\end{corollary}

\begin{proof}
Let $\hat{\nu}$ be the right side of the equation. By Theorem~\ref{thm:weak}(b) and the preceding discussion, $\hat{\nu}$ is a well-defined weak measure on $\fP$. By Theorem~\ref{thm:weak}(c), $\nu$ and $\hat{\nu}$ take the same values on $e$ and the $j_{p,q}$. By Theorem~\ref{thm:weak}(a), it follows that $\nu=\hat{\nu}$.
\end{proof}

Due to the above corollary, the quantities $\nu_e$ and $\nu_{p,q}$ will play an important role whenever we work with weak measures. We call these the \defn{parameters} of $\nu$. Note that if $\nu$ is a measure then $\nu_e=1$; in this case, we simply refer to the $\nu_{p,q}$ as the parameters. From Theorem~\ref{thm:weak}(c), we see that the expression in Corollary~\ref{cor:weak1} is the unique expression for $\nu$ as a linear combination of $\tau$ and the $\lambda_{p,q}$. We therefore find:

\begin{corollary}
Every weak measure can be expressed uniquely as a (possibly infinite) linear combination of $\tau$ and the $\lambda_{p,q}$.
\end{corollary}

Applying Corollary~\ref{cor:weak1} to the universal weak measure, we find:

\begin{corollary} \label{cor:weak3}
For a marked poset $(P,x)$, we have the following identity in $\Theta^1(\fP)$:
\begin{displaymath}
\lbb P, x \rbb = -\tau(P,x) e + \sum_{p,q \ge 0} \lambda_{p,q}(P,x) j_{p,q}
\end{displaymath}
\end{corollary}

Finally, we record a simple corollary about measures.

\begin{corollary} \label{cor:linear}
The classes $j_{p,q}$ generate $\Theta(\fP)$ as a ring.
\end{corollary}

\begin{proof}
This follows since $\Sym{\Theta^1(\fP)} \to \Theta(\fP)$ is surjective and maps $e$ to~1.
\end{proof}

\subsection{Proof of Theorem~\ref*{thm:weak}} \label{ss:weakpf}

We now prove the theorem. The plan is as follows: in the following lemmas, we show that $e$ and the $j_{p,q}$ generate $\Theta^1(\fP)$, and then show that $\lambda_{p,q}$ and $\tau$ are weak measures. From here, completing the proof is a simple matter.

\begin{lemma} \label{lem:weak-1}
$\Theta^1(\fP)$ is generated by $e$ and the $j_{p,q}$ for $p,q \in \bN$.
\end{lemma}

\begin{proof}
Let $(P,x)$ be a marked poset. We show that $\lbb P,x \rbb$ can be generated by the stated classes by induction on the pair $(\vert P \vert, \vert P_{\parallel x} \vert)$, with the lexicographic order. First suppose there is an element $y \in P^x$ that is incomparable to $x$. Let $Q$ be the pre-amalgamation $(P,x,y)$. Note that $P$ is the amalgamation $Q_0$. In $\Theta^1(\fP)$, we have the linear relation
\begin{displaymath}
\lbb P^y, x \rbb = \epsilon \cdot e + \lbb P, x \rbb + A + B,
\end{displaymath}
where $\epsilon=\epsilon(Q,x,y)$, $A$ is $\lbb Q_-, x \rbb$ if $Q_-$ exists and~0 otherwise, and $B$ is defined similarly with $Q_+$. Rewrite the above equation as
\begin{displaymath}
\lbb P, x \rbb = \lbb P^y, x \rbb - \epsilon \cdot e - A - B.
\end{displaymath}
The marked posets appearing on the right are smaller than $(P,x)$: indeed, $P^y$ has smaller cardinality, while $(Q_{\pm}, x)$ has fewer elements that are incomparable with $x$. The inductive hypothesis thus shows that $\lbb P, x \rbb$ is generated by the desired elements.

Now, let $(P,x)$ be a marked poset with $P_{\parallel x} = \emptyset$. If $y \in P^x$ is not immediately above or below $x$ then $y$ is separated from $x$ (Proposition~\ref{prop:poset-sep}), and so $\lbb P, x \rbb = \lbb P^y, x \rbb$ (Proposition~\ref{prop:sep}). Thus, by the inductive hypothesis, $\lbb P,x \rbb$ is generated by the desired elements. We thus reduce to the case where every element of $P^x$ is either immediately above or below $x$. The marked poset $(P,x)$ is isomorphic to some $(J_{p,q}, \ast)$, and so the result follows.
\end{proof}

\begin{lemma} \label{lem:weak-2}
The rule $\lambda$ defines a weak measure valued in $\bZ[u,v]$. That is, if $(P,x,y)$ is a pre-amalgamation then
\begin{displaymath}
\lambda(P^y,x) = \sum_Q \lambda(Q,x),
\end{displaymath}
where the sum is over those $P_+$, $P_0$, $P_-$ that exist.
\end{lemma}

\begin{proof}
Let $a_*=\lambda(P_*, x)$, where $*$ is $+$, $-$, or 0, with the convention that $a_*=0$ if $P_*$ does not exist. Thus we must show $\lambda(P^y, x) = a_0 + a_+ + a_-$. We break the proof into two cases.

\textit{Case 1: $P_0$ exists.} We have
\begin{displaymath}
(P_0)_{<x} = (P^y)_{<x}, \qquad
(P_0)_{>x} = (P^y)_{>x}, \qquad
(P_0)_{\parallel x} = (P^y)_{\parallel x} \cup \{y\}.
\end{displaymath}
We say that a cut $(L,R)$ for $(P_0,x)$ has type~1 if $y$ does not belong to $L \cup R$, type~2 if $y \in R$, and type~3 if $y \in L$; every cut has exactly one of the three types. Let $\cC_i$ be the set of cuts of type~$i$. Write $a_0=a_0^1+a_0^2+a_0^3$, where $a_0^i$ consists of the contributions from cuts of type $i$. Note that the type~1 cuts for $(P_0,x)$ correspond bijectively to cuts for $(P^y,x)$, and this correspondence preserves the $p$ and $q$ quantities. Thus $a_0^1=\lambda(P^y,x)$. We show that $a_0^2$ and $a_-$ cancel, and similarly $a_0^3$ and $a_+$ cancel. We divide the analysis into two subcases.

\textit{Case 1a: $P_-$ exists.} Suppose $L$ and $R$ are subsets of $(P_-)_{\parallel x}$, and hence do not contain $y$. Then
\begin{displaymath}
L \cup (P_-)_{<x} = L \cup (P_0)_{<x}, \qquad
R \cup (P_-)_{>x} = R \cup \{y\} \cup (P_0)_{>x}.
\end{displaymath}
Moreover, since $x$ does not appear in these sets, it follows that comparisons between any two elements are the same whether computed in $P_-$ or in $P_0$. It follows that the conditions
\begin{displaymath}
L \cup (P_-)_{<x} < R \cup (P_-)_{>x}, \qquad
L \cup (P_-)_{<x} < R \cup \{y\} \cup (P_0)_{>x}, \qquad
\end{displaymath}
are equivalent, where in the first we work in $P_-$ and in the second in $P_0$. Thus, letting $\cC_-$ be the set of cuts for $(P_-,x)$, we have a bijection
\begin{displaymath}
\cC_- \to \cC_2, \qquad (L,R) \mapsto (L, R \cup \{y\}).
\end{displaymath}
Moreover, we have
\begin{displaymath}
\Max(L \cup (P_-)_{<x}) = \Max(L \cup (P_0)_{<x}), \qquad
\Min(R \cup (P_-)_{>x}) = \Min(R \cup \{y\} \cup (P_0)_{>x}).
\end{displaymath}
Here the left sides are computed in $P_-$ and the right sides in $P_0$. Thus the bijection preserves the $p$ and $q$ invariants. Since $(L,R)$ and $(L, R \cup \{y\})$ contribute to $a_-$ and $a_0^2$ with opposite signs, we find $a_0^2+a_-=0$.

\textit{Case 1b: $P_-$ does not exist.} We claim that $a_0^2=0$. If $P^y_{<x} \not\subset P^x_{<y}$ then there is some $z \in P^{x,y}$ such that $z<x$ but $z \nless y$. If $(L,R)$ is a cut then the condition $L \cup (P_0)_{<x} < R \cup (P_0)_{>x}$ implies in particular that $z<b$ for all $b \in R$, and so we cannot have $y \in R$. Thus there are no type~2 cuts and so $a_0^2=0$, which proves the claim in this case.

Now suppose $P^y_{<x} \subset P^x_{<y}$. By Proposition~\ref{prop:poset-amalg}(a), it follows that $P^x_{>y} \not\subset P^y_{>x}$, and so there is some $z \in P^{x,y}$ with $z>y$ and $z \ngtr x$. Since $P_0$ exists, we have $P^x_{>y} \cap P^y_{<x}=\emptyset$ by Proposition~\ref{prop:poset-amalg}(c), and so $z \nless x$. Hence $z \parallel x$. Write $\cC_2=\cC^1_2 \sqcup \cC_2^2$, where $(L,R)$ belongs to $\cC_2^1$ if $z \notin R$, and to $\cC_2^2$ if $z \in R$. We then have a bijection
\begin{displaymath}
\cC_2^1 \to \cC_2^2, \qquad (L,R) \mapsto (L, R \cup \{z\})
\end{displaymath}
preserving the $(p,q)$ invariants, as we now explain. Since we are only working with type~2 cuts, we have $y \in R$. Since $y<z$, the condition $L \cup (P_0)_{<x} < R \cup (P_0)_{>x}$ is not disturbed by adding $z$ to $R$ or removing $z$ from it. Moreover, for the same reason, the set $\Min(R \cup P_{>x})$ is unaffected by adding $z$ to $R$ or removing $z$ from it. Finally, a type~2 cut cannot have $z \in L$ since $z>y$ and $y \in R$. This establishes the bijection. Since $(L,R)$ and $(L,R \cup \{z\})$ contribute with opposite signs to $a^2_0$, we find $a^2_0=0$, which completes the proof of the claim.

\textit{Completion of Case~1.} We have shown that $a^2_0 + a_-=0$ whether $P_-$ exists or not. Similarly, $a^3_0 + a_+=0$ in all cases. We thus find
\begin{displaymath}
a_0 + a_+ + a_- = a_0^1 = \lambda(P^y, x),
\end{displaymath}
as required.

\textit{Case 2: $P_0$ does not exist.} By Proposition~\ref{prop:poset-amalg}, at least one of $(P^y)_{<x} \cap (P^x)_{>y}$ and $(P^y)_{>x} \cap (P^x)_{<y}$ is non-empty. They cannot both be non-empty, by Proposition~\ref{prop:poset-amalg-2} (or by direct inspection).

Suppose $(P^y)_{>x} \cap (P^x)_{<y}$ is non-empty, and fix an element $z$ of this set. By Proposition~\ref{prop:poset-amalg-2}, $P_-$ is the unique amalgamation. Let $L$ and $R$ be subsets of $(P^y)_{\parallel x}=(P_-)_{\parallel x}$. Then
\begin{displaymath}
L \cup (P_-)_{<x} = L \cup P^y_{<x}, \qquad
R \cup (P_-)_{>x} = R \cup \{y\} \cup P^y_{>x} .
\end{displaymath}
Clearly, if $(L,R)$ is a cut for $(P_-,x)$ then it is one for $(P^y, x)$ as well. Conversely, if $(L,R)$ is a cut for $(P^y,x)$ then for any $a \in L \cup (P_-)_{<x}$ we have $a<z$ (since $z$ belongs to $P^y_{>x}$), and therefore $a<y$ (since $z<y$), which shows that $(L,R)$ is a cut for $(P_-,x)$. Moreover, we have
\begin{displaymath}
\Max(L \cup (P_-)_{<x}) = \Max(L \cup P^y_{<x}), \qquad
\Min(R \cup (P_-)_{>x}) = \Min(R \cup P^y_{>x}),
\end{displaymath}
where the left sides are computed in $P_-$ and the right sides in $P^y$. Indeed, the key point is that $z$ belongs to $(P_-)_{>x}$ and $z<y$, and so $y$ is not a minimal element of $R \cup (P_-)_{>x}$. We thus see that the cuts for $(P_-,x)$ and $(P^y,x)$ are exactly the same, and have the same $p$ and $q$ invariants. Therefore, $\lambda(P^y,x)=\lambda(P_-,x)$, as required.

The case where $P^y_{<x} \cap P^x_{>y}$ is non-empty follows from a similar argument (apply the above argument in the opposite poset).
\end{proof}

It follows from the above lemma that each $\lambda_{p,q}$ is a weak measure.

\begin{lemma} \label{lem:weak-3}
The rule $\tau$ defines a weak measure. That is, if $(P,x,y)$ is a pre-amalgamation then
\begin{displaymath}
\tau(P^y,x) = -\epsilon + \sum_Q \tau(Q,x),
\end{displaymath}
where $\epsilon=\epsilon(P,x,y)$ and $Q$ runs over the proper amalgamations.
\end{lemma}

\begin{proof}
We first examine twins that are not $y$. Let $S$ be the set of pairs $(Q,z)$ where $Q$ is a proper amalgamation and $z \in Q^{x,y}$ is a twin of $x$ in $Q$, and let $T$ be the set of elements in $P^{x,y}$ that are twins of $x$ in $P^y$. If $(Q,z) \in S$ then $Q_{<z}=Q_{<x}$ and $Q_{>z}=Q_{>x}$, and so $P^y_{<z}=P^y_{<x}$ and $P^y_{>z}=P^y_{>x}$, meaning $z$ is a twin of $x$ in $P^y$. We thus have a map
\begin{displaymath}
S \to T, \qquad (Q,z) \mapsto z.
\end{displaymath}
We claim that this map is a bijection. Let $z \in T$ be given. We must show that there is a unique proper amalgamation $Q$ such that $z$ is a twin of $x$ in $Q$. Uniqueness is clear: since $x$ and $z$ are twins, we have
\begin{displaymath}
x<y \iff z<y, \qquad x>y \iff z>y,
\end{displaymath}
and so the comparison of $x$ and $y$ in $Q$ is determined.

We now prove existence. Let $\pi \colon P \to P^x$ be the map that is the identity on $P^x$ and satisfies $\pi(x)=z$. Let $\prec$ be the pullback of $<$ under $\pi$, i.e., $a \prec b$ is defined to be $\pi(a)<\pi(b)$. It is clear that $\prec$ is a partial order on $P$, that $x$ and $z$ are twins in $(P, \prec)$, and that $\prec$ agrees with $<$ on $P^x$. Moreover, $\prec$ agrees with $<$ on $P^y$. Indeed, the two orders agree on $P^{x,y}$, and so it suffices to check comparisons involving $x$. For $a \in P^{x,y}$, we have
\begin{displaymath}
a<x \iff a<z \iff a \prec z \iff a \prec x,
\end{displaymath}
where in the first step we use that $x$ and $z$ are twins in $(P^y,<)$, in the second step we use that $\prec$ and $<$ agree on $P^x$, and in the third step we use that $x$ and $z$ are twins in $(P,\prec)$; the equivalence also holds when $a=x$ since $x<x$ and $x \prec x$ are both false. We have $a>x \iff a \succ x$ for $a \in P^y$ by the same reasoning. We thus see that $(P,\prec)$ is a proper amalgamation $Q$, and $(Q,z) \in S$, as required.

We now examine when $y$ is a twin. Let $S'$ be the set of proper amalgamations $Q$ such that $y$ is a twin of $x$ in $Q$. Since twins are incomparable, it follows that $S'$ can only contain $P_0$. It is clear that $P_0$ exists and belongs to $S'$ if and only if $\epsilon=1$.

We have
\begin{displaymath}
\sum_Q \tau(Q,x) = \vert S \vert + \vert S' \vert = \vert T \vert + \epsilon = \epsilon + \tau(P^y,x),
\end{displaymath}
which completes the proof.
\end{proof}

We now complete the proof of the theorem.

\begin{proof}[Proof of Theorem~\ref{thm:weak}]
Lemmas~\ref{lem:weak-2} and~\ref{lem:weak-3} cover part (b). Clearly, $\ast$ has no twin in $J_{r,s}$, and so $\tau(j_{r,s})=0$. Every element of $J_{r,s} \setminus \{\ast\}$ is comparable to $\ast$, and so the only cut is $(\emptyset, \emptyset)$. This has $p(\emptyset)=r$ and $q(\emptyset)=s$. We thus find $\lambda(j_{r,s})=u^rv^s$, and so $\lambda_{p,q}(j_{r,s})=\delta_{p,r} \delta_{q,s}$. This proves (c). Lemma~\ref{lem:weak-1} shows that $e$ and the $j_{p,q}$ generate $\Theta^1(\fP)$. Consider a linear relation $ae+\sum b_{p,q} j_{p,q} = 0$ where $a$ and $b_{p,q}$ belong to $\bZ$, and all but finitely many $b_{p,q}$ vanish. Applying $\tau$ gives $a=0$, and applying $\lambda_{p,q}$ gives $b_{p,q}=0$. Thus $e$ and the $j_{p,q}$ are $\bZ$-linearly independent. This proves (a).
\end{proof}

\subsection{Explicit calculations} \label{ss:explicit}

Put $\lambda_{p,\ast} = \sum_{q \ge 0} \lambda_{p,q}$, and analogously define $\lambda_{\ast,q}$. Also, let $\lambda_{\ast,\ast} = \sum_{p,q \ge 0} \lambda_{p,q}$. The $\lambda_{p,q}$ with $p,q \in \{0,1,\ast\}$ are particularly relevant to measures since the parameters of measures stabilize (Theorem~\ref{thm:quad}(a)). With the exception of $\lambda_{\ast,\ast}$, they admit reasonably direct descriptions, as we now explain. We say that elements $x<y$ in a poset $P$ are a \defn{series pair} if $y$ is the unique cover of $x$ and $x$ is the unique element covered by $y$; in other words, $P_{<y}=P_{\le x}$ and $P_{>x}=P_{\ge y}$. In this case, we say $x$ \defn{begins} the series pair and $y$ \defn{ends} it.

\begin{proposition} \label{prop:lambda-calc}
Let $(P,x)$ be a marked poset. Then:
\begin{enumerate}
\item $\lambda_{0,0}(P,x)$ is~1 if $x$ is an isolated point of $P$, and~0 otherwise.
\item $\lambda_{\ast,0}(P,x)$ is~1 if $x$ is the greatest element of $P$, and~0 otherwise.
\item If $x$ is not a maximal element of $P$ then $\lambda_{1,0}(P,x)=0$. Assume $x$ is maximal. Let $\epsilon$ be~1 if $x$ covers a unique element, and~0 otherwise, and let $n$ be the number of lower twins to $x$ (\S \ref{ss:twins}). Then $\lambda_{1,0}(P,x)=\epsilon-n$.
\item Let $\epsilon$ be~1 if $x$ begins a series pair and~0 otherwise. Then $\lambda_{\ast,1}(P,x)=\epsilon-\tau(P,x)$.
\end{enumerate}
See Table~\ref{t:weak} for a summary.
\end{proposition}

\begin{table}
\caption{Qualitative meaning of weak measures; see Proposition~\ref{prop:lambda-calc} for details.}
\label{t:weak}
\centering
\begin{tabular}{cl}
\toprule
weak measure & meaning \\
\midrule
$\lambda_{0,0}$ & detects isolated marked point \\
$\lambda_{\ast,0}$ & detects greatest marked point \\
$\lambda_{1,0}$ & unique lower cover minus lower twin count \\
$\lambda_{\ast,1}$ & begins a series pair minus twin count \\
\bottomrule
\end{tabular}
\end{table}

\begin{proof}
Recall that for a cut $(L,R)$ of $(P,x)$, we let $p(L)=\vert \Max(L \cup P_{<x}) \vert$ and $q(R)=\vert \Min(R \cup P_{>x}) \vert$. The map $\lambda_{p,q} \colon \Theta^1(\fP) \to \bZ$ kills $e$ and takes a marked poset $(P,x)$ to $\sum (-1)^{\vert L \vert+\vert R \vert}$, where the sum is over cuts $(L,R)$ with $p(L)=p$ and $q(R)=q$. A cut $(L,R)$ has $q(R)=0$ if and only if $R$ and $P_{>x}$ are empty. The condition $P_{>x}=\emptyset$ is equivalent to $x$ being maximal in $P$. If $x$ is maximal then any pair $(L, \emptyset)$ with $L \subset P_{\parallel x}$ is a cut. There is a dual description of cuts with $p(L)=0$.

(a) From the above discussion, we see that a cut $(L,R)$ has $p(L)=q(R)=0$ if and only if $L$ and $R$ are empty and $x$ is isolated. The result thus follows.

(b) Again, $\lambda_{\ast,0}(P,x)=0$ if $x$ is not a maximal element of $P$, so assume $x$ is maximal. Then $\lambda_{\ast,0} = \sum (-1)^{\vert L \vert}$, where the sum is over all subsets $L$ of $P_{\parallel x}$. This vanishes if $P_{\parallel x}$ is non-empty, and is~1 otherwise. The result thus follows; indeed, $x$ is the greatest element if and only if it is maximal and $P_{\parallel x}$ is empty.

(c) Once again, $\lambda_{1,0}(P,x)=0$ unless $x$ is maximal, so we assume that $x$ is maximal. We have $\lambda_{1,0}(P,x) = \sum (-1)^{\vert L \vert}$, where the sum is over subsets $L$ of $P_{\parallel x}$ with $p(L)=1$. We have $p(L)=1$ if and only if $L \cup P_{<x}$ has a greatest element. We thus have
\begin{displaymath}
\lambda_{1,0}(P,x) = \sum_{a \in P_{<x} \cup P_{\parallel x}} \sum_{\substack{L \subset P_{\parallel x}, \\ \Max(L \cup P_{<x})=\{a\}}} (-1)^{\vert L \vert}.
\end{displaymath}
Suppose $a \in P_{<x}$. Since $L<a<x$ and $L \subset P_{\parallel x}$, we see that $L=\emptyset$. The condition $\Max(P_{<x})=\{a\}$ means that $a$ is the greatest element of $P_{<x}$, which exactly means that $a$ is the unique element covered by $x$. We thus see that the total contribution of the terms with $a \in P_{<x}$ is $\epsilon$.

Now consider the case $a \in P_{\parallel x}$. We must have $P_{<x}<a$, or there are no terms in the inner sum; thus assume this. The inner sum varies over sets $L$ of the form $\{a\} \cup L_0$, where $L_0$ is an arbitrary subset of $P_{\parallel x} \cap P_{<a}$. If this set is empty, meaning $a$ is a minimal element of $P_{\parallel x}$, then the inner sum is $-1$; otherwise, it vanishes. We thus see that each minimal element $a$ of $P_{\parallel x}$ satisfying $P_{<x}<a$ contributes $-1$, and the other elements $a$ of $P_{\parallel x}$ contribute~0. Such elements $a$ are exactly the lower twins of $x$ by Proposition~\ref{prop:lowertwin}. The result thus follows.

(d) We have $\lambda_{\ast,1} = \sum_{p \ge 0} \lambda_{p,1}$. For a cut $(L,R)$, we have $q(R)=1$ if and only if $R \cup P_{>x}$ has a least element. As in (c), we organize the computation according to this least element. Thus write
\begin{displaymath}
\lambda_{\ast,1}(P,x) = \sum_{a \in P_{>x} \cup P_{\parallel x}} \sum_{\substack{(L,R), \\ \Min(R \cup P_{>x})=\{a\}}} (-1)^{\vert L \vert+\vert R \vert},
\end{displaymath}
where the inner sum is over cuts $(L,R)$ as specified.

Suppose $a \in P_{>x}$. Consider a cut $(L,R)$ in the inner sum. Since $a$ is the least element of $R \cup P_{>x}$ and $x<a$ and $R \subset P_{\parallel x}$, it follows that $R$ is empty and $a$ is the least element of $P_{>x}$. We thus see that the inner sum is empty unless $a$ is the least element of $P_{>x}$, i.e., $a$ is the unique cover of $x$; thus assume this. The inner sum is then over subsets $L$ of $P_{<a} \cap P_{\parallel x}$. If $P_{<a} \cap P_{\parallel x}$ is empty then the inner sum is~1. In this case, if $b<a$ then necessarily $b \le x$, which means $x$ is the unique element covered by $a$; thus $x<a$ is a series pair. If $P_{<a} \cap P_{\parallel x}$ is non-empty then the inner sum is zero. In this case, $a$ covers an element other than $x$, and so $x<a$ is not a series pair. We thus see that the total contribution from the case $a \in P_{>x}$ is $\epsilon$.

Now suppose $a \in P_{\parallel x}$. For the inner sum to be non-zero, we must have $a<P_{>x}$ (otherwise $a$ cannot be the least element of $R \cup P_{>x}$) and $P_{<x}<a$ (otherwise there is no cut with $a \in R$); thus assume these hold. The inner sum is taken over pairs $(L,\{a\} \cup R_0)$ where $L \subset P_{\parallel x} \cap P_{<a}$ and $R_0 \subset P_{\parallel x} \cap P_{>a}$ are arbitrary subsets. If $P_{\parallel x} \cap P_{<a}$ and $P_{\parallel x} \cap P_{>a}$ are empty then the inner sum is $-1$, as $\vert R \vert=1$; otherwise the inner sum vanishes. The condition that both $P_{\parallel x} \cap P_{<a}$ and $P_{\parallel x} \cap P_{>a}$ are empty exactly means that $a$ is incomparable to every element of $P_{\parallel x}$; combined with the condition $P_{<x}<a<P_{>x}$, this exactly means that $a$ is a twin of $x$. We thus see that the total contribution from the case $a \in P_{\parallel x}$ is $-\tau(P,x)$.
\end{proof}

Of course, there are similar formulas for $\lambda_{0,\ast}$, $\lambda_{0,1}$, and $\lambda_{1,\ast}$ obtained by transposition.

\begin{remark}
The formula for $\lambda_{1,1}$ is a little more complicated. As we do not actually use it, we simply provide the statement. Let $(P,x)$ be a marked poset. Define the following:
\begin{itemize}
\item $\delta$ is~1 if $x$ covers a unique element, and~0 otherwise.
\item $\epsilon$ is~1 if $x$ has a unique cover, and~0 otherwise.
\item $A = \{ a \in \Min(P_{\parallel x}) \mid P_{<x} < a < P_{>x} \}$.
\item $B = \{ b \in \Max(P_{\parallel x}) \mid P_{<x} < b < P_{>x} \}$.
\item $n$ is the number of pairs $(a,b) \in A \times B$ with $a<b$.
\end{itemize}
Then $\lambda_{1,1}(P,x) = \delta \epsilon - \epsilon \vert A \vert - \delta \vert B \vert + n$.
\end{remark}

\subsection{Connection to M\"obius functions} \label{ss:mobius}

We now show how the weak measures $\tau$ and $\lambda_{p,q}$ are related to the M\"obius function on a certain poset. This will be an important perspective when we prove Theorem~\ref{thm:some-meas}(a).

Let $P$ be a poset and let $n \in\bN$. An \defn{$n$-extension} of $P$ is a triple $(Q, \alpha, x_1, \ldots, x_n)$ where $Q$ is a poset, $\alpha \colon P \to Q$ is an embedding of posets, and the $x_i$ are elements of $Q$ such that $Q=\im(\alpha) \cup \{x_1, \ldots, x_n\}$. For notational simplicity, we identify $P$ with its image in $Q$ and suppress $\alpha$ from the notation. Suppose $Q=(Q, x_1, \ldots, x_n)$ and $T=(T, y_1, \ldots, y_n)$ are two $n$-extensions. The \defn{canonical function} $\phi \colon Q \to T$, when it exists, is the unique function that is the identity on $P$ and maps $x_i$ to $y_i$. We define $Q \le T$ if the canonical function $\phi \colon Q \to T$ exists and is strictly monotone, meaning $a<b$ implies $\phi(a)<\phi(b)$ for all $a,b \in Q$. Let $\cE_n(P)$ denote the (finite) set of isomorphism classes of $n$-extensions of $P$. This set is partially ordered by the relation $\le$ just defined.

For the present discussion, the $n=1$ case is most relevant. We say that a 1-extension $(Q, x)$ is \defn{proper} if $x \notin P$, and \defn{improper} otherwise. We let $\cE_1^{\circ}(P)$ denote the set of proper extensions. Note that improper extensions are maximal elements of the poset $\cE_1(P)$.

Let $(P,x)$ be a marked poset and let $(L,R)$ be a cut. In \S \ref{ss:cuts}, we defined a modified order $\prec$ associated to the cut. We denote the poset $(P, \prec)$ by $P[L,R]$, and mark it by $x$.

\begin{proposition} \label{prop:mu-cuts}
Let $(P,x)$ be a marked poset, and let $T$ be a proper 1-extension of $P^x$ such that $P \le T$ holds in $\cE_1(P^x)$. Then
\begin{displaymath}
\mu(P,T) = \sum_{P[L,R]=T} (-1)^{\vert L \vert + \vert R \vert}
\end{displaymath}
where $(L,R)$ varies over cuts of $(P,x)$ satisfying the stated condition.
\end{proposition}

\begin{proof}
Let $L_T$ be the set of elements $a \in P_{\parallel x}$ such that $a <_T x$, and similarly define $R_T$. Then $(L_T, R_T)$ is a cut. In fact, for any subsets $L \subset L_T$ and $R \subset R_T$, the pair $(L,R)$ is a cut. Indeed, if $a \in L \cup P_{<x}$ and $b \in R \cup P_{>x}$ then $a <_T x <_T b$, and so $a <_T b$, which implies $a<b$, as required. For a cut $(L,R)$, we have $P[L,R] \le T$ if and only if $L \subset L_T$ and $R \subset R_T$. Now, put
\begin{displaymath}
\hat{\mu}(P,T) = \sum_{P[L,R]=T} (-1)^{\vert L \vert + \vert R \vert}.
\end{displaymath}
We have
\begin{displaymath}
\sum_{P \le S \le T} \hat{\mu}(P,S) = \sum_{L \subset L_T, R \subset R_T} (-1)^{\vert L \vert + \vert R \vert} = \delta_{T,P}.
\end{displaymath}
Note that the sum is~1 if $L_T$ and $R_T$ are empty, i.e., $T=P$, and~0 otherwise; this explains the final equality above. Since this identity characterizes the M\"obius function, we have $\mu=\hat{\mu}$, and so the result follows.
\end{proof}

For a marked poset $(T,y)$, put $p(T)=\vert \Max(T_{<y}) \vert$ and $q(T)=\vert \Min(T_{>y}) \vert$.

\begin{proposition}
For a marked poset $(P,x)$, we have
\begin{displaymath}
\lambda(P,x) = \sum_{\substack{T \in \cE^{\circ}_1(P^x), \\ P \le T}} \mu(P,T) u^{p(T)} v^{q(T)}.
\end{displaymath}
\end{proposition}

\begin{proof}
We have
\begin{align*}
\sum_{P \le T} \mu(P,T) u^{p(T)} v^{q(T)}
&= \sum_{P \le T} \sum_{P[L,R]=T} (-1)^{\vert L \vert + \vert R \vert} u^{p(T)} v^{q(T)} \\
&= \sum_{(L,R)} (-1)^{\vert L \vert + \vert R \vert} u^{p(P[L,R])} v^{q(P[L,R])},
\end{align*}
where $T \in \cE_1^{\circ}(P^x)$, the inner sum on the right varies over cuts $(L,R)$ with the stated property, and the final sum is over all cuts. In the first step, we used Proposition~\ref{prop:mu-cuts}. Since $p(P[L,R]) = p(L)$ and $q(P[L,R])=q(R)$, the result follows.
\end{proof}

\begin{corollary} \label{cor:euler-char}
We have
\begin{displaymath}
\lambda_{\ast,\ast}(P,x) = \sum_{\substack{T \in \cE^{\circ}_1(P^x), \\ T \ge P}} \mu(P,T).
\end{displaymath}
\end{corollary}

\begin{remark}
Let $(P,x)$ be a marked poset and put $\cE=\cE^{\circ}_1(P^x)$. Corollary~\ref{cor:euler-char} can be phrased as follows:
\begin{displaymath}
\lambda_{\ast,\ast}(P,x) = -\tilde{\chi}(\cE_{>(P,x)})
\end{displaymath}
where $\tilde{\chi}$ is the reduced Euler characteristic, assuming $\cE_{>(P,x)}$ is non-empty. Thus $\lambda_{\ast,\ast}$ is a natural invariant attached to $(P,x)$.
\end{remark}

We now turn our attention to twins.

\begin{proposition} \label{prop:twin-mobius}
Let $(P,x)$ be a proper 1-extension of $Q$ and let $T=(Q,y)$ be an improper 1-extension of $Q$. Then $\mu(P,T)$ is~$-1$ if $x$ and $y$ are twins in $P$, and~0 otherwise.
\end{proposition}

\begin{proof}
If $P \not\le T$ then one easily sees that $x$ and $y$ are not twins, and $\mu(P,T)=0$ by convention. Suppose now that $P \le T$. By general properties of the M\"obius function, we have
\begin{displaymath}
\mu(P,T) = -\sum_{P \le S < T} \mu(P,S).
\end{displaymath}
Since improper 1-extensions are maximal elements of $\cE_1(Q)$, the extensions $S$ appearing above must be proper 1-extensions. Appealing to Proposition~\ref{prop:mu-cuts}, we have
\begin{displaymath}
\mu(P,T) = -\sum_{P \le S < T} \sum_{P[L,R]=S} (-1)^{\vert L \vert + \vert R \vert},
\end{displaymath}
where the inner sum is over cuts $(L,R)$ of $(P,x)$ satisfying the stated condition.

Now, suppose $(L,R)$ is an arbitrary cut and put $S=P[L,R]$. The canonical function $\phi \colon S \to T$ exists, and is the unique function that is the identity on $Q$ and maps $x$ to $y$. It is strictly monotone if and only if all elements of $P_{<x} \cup L$ are below $y$, and all elements of $P_{>x} \cup R$ are above $y$. We thus see that $S \le T$ if and only if these conditions hold. We therefore have
\begin{equation} \label{eq:twin-mobius}
\mu(P,T) = \sum_{(L,R)} (-1)^{\vert L \vert + \vert R \vert+1},
\end{equation}
where the sum is over all cuts $(L,R)$ for $(P,x)$ with $P_{<x} \cup L$ below $y$ and $P_{>x} \cup R$ above $y$.

Suppose that one of the containments $P_{<x} \subset P_{<y}$ or $P_{>x} \subset P_{>y}$ fails to hold. Then there are no cuts satisfying the above conditions, and so the sum in \eqref{eq:twin-mobius} is empty, i.e., $\mu(P,T)=0$. We also see that $x$ and $y$ are not twins. Thus the proposition holds.

Now suppose $P_{<x} \subset P_{<y}$ and $P_{>x} \subset P_{>y}$. The cuts appearing in \eqref{eq:twin-mobius} are exactly those pairs $(L,R)$ where $L$ is an arbitrary subset of $P_{<y} \cap P_{\parallel x}$ and $R$ is an arbitrary subset of $P_{>y} \cap P_{\parallel x}$; indeed, given such $L$ and $R$, we have
\begin{displaymath}
P_{<x} \cup L < y < P_{>x} \cup R,
\end{displaymath}
and so $(L,R)$ is a cut. The sum in \eqref{eq:twin-mobius} is therefore $-1$ precisely when both $P_{<y} \cap P_{\parallel x}$ and $P_{>y} \cap P_{\parallel x}$ are empty, which is equivalent to $x$ and $y$ being twins; otherwise it is~0.
\end{proof}

\section{Upper bounds on measures} \label{s:upper}

In \S \ref{s:upper}, we give an upper bound on the space of measures on $\fP$: we show that any measure is determined by its associated $3 \times 3$ matrix, and that this matrix (or its transpose) must be among those in Table~\ref{t:matrix}.

\subsection{The main theorem} \label{ss:quadthm}

The following is the main result of \S \ref{s:upper}:

\begin{theorem} \label{thm:quad}
Let $\nu$ be a measure for $\fP$ valued in a field $k$, and let $\nu_{p,q}=\nu(j_{p,q})$.
\begin{enumerate}
\item We have $\nu_{p,q}=\nu_{\ol{p}, \ol{q}}$, where $\ol{n}=\min(n,2)$. Consequently, $\nu$ is completely determined by the $3 \times 3$ matrix $N = (\nu_{p,q})_{0 \le p,q \le 2}$.
\item The matrix $N$ belongs to the 2-parameter family $\AA$, one of the eight 1-parameter families $\BB_i$ or $\BB_i^{\top}$ ($1 \le i \le 5$), or is one of the 15 matrices $\CC_i$ or $\CC_i^{\top}$ ($1 \le i \le 8$).
\end{enumerate}
\end{theorem}

This theorem allows us to parametrize measures by $3 \times 3$ matrices. Precisely, a measure $\nu$ corresponds to the matrix $N = N(\nu)$ appearing above. We can recover $\nu$ from $N$ using the formula from Corollary~\ref{cor:weak1}:
\begin{equation} \label{eq:meas}
\nu = -\tau + \sum_{p,q \ge 0} \nu_{\ol{p}, \ol{q}} \lambda_{p,q}.
\end{equation}
If we start with an arbitrary matrix $N$ and define $\nu = \nu(N)$ by the above formula, then $\nu$ is a weak measure, but may or may not be an actual measure. Theorem~\ref{thm:quad}(b) shows that for $\nu$ to be a measure a necessary condition is that $N$ (or its transpose) must appear in Table~\ref{t:matrix}. We will eventually show this is sufficient as well.

The proof of Theorem~\ref{thm:quad} will take the entirety of \S \ref{s:upper}. We first prove Theorem~\ref{thm:quad}(a), then obtain a list of equations the entries of $N$ must satisfy, and finally solve these equations.

\subsection{Stabilization of the $j$ classes} \label{ss:jstab}

For $m,n \in \bN$, let $\delta_{m,n}=j_{m,n+3}-j_{m,n+2}$, regarded as an element of $\Theta(\fP)$. We show:

\begin{proposition} \label{prop:jstab}
The elements $\delta_{m,n}$ are nilpotent for all $m,n \in \bN$.
\end{proposition}

Using the transpose involution, it follows that the analogous elements $j_{m+3,n}-j_{m+2,n}$ are also nilpotent. We have seen that the classes $j_{p,q}$ generate $\Theta(\fP)$ as a ring (Corollary~\ref{cor:linear}). We can now refine this result:

\begin{corollary} \label{cor:quad}
The classes $j_{p,q}$ with $0 \le p,q \le 2$ generate $\Theta(\fP)$ modulo its nilradical.
\end{corollary}

Proposition~\ref{prop:jstab} implies Theorem~\ref{thm:quad}(a). Indeed, the measure $\nu$ in Theorem~\ref{thm:quad} corresponds to a ring homomorphism $\Theta(\fP) \to k$, and this must kill nilpotents since $k$ is a field.

In the remainder of \S \ref{ss:jstab}, we prove Proposition~\ref{prop:jstab}. For $m,n \in \bN$, let $E=E_{m,n}$ denote the poset depicted by the following Hasse diagram
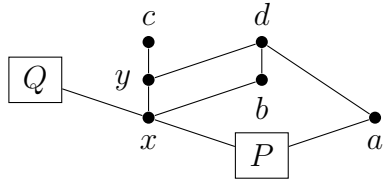
\begin{center}
\begin{tikzpicture}[
    x=1.5cm,
    y=1cm,
    dot/.style={
        circle,
        fill,
        inner sep=1.6pt
    },
    boxed/.style={
        draw,
        rectangle,
        minimum width=7mm,
        minimum height=6mm,
        inner sep=1pt
    }
]

\node[boxed] (P) at (0,0) {$P$};
\node[dot,label=below:$x$] (x) at (-1,0.5) {};
\node[dot,label=below:$a$] (a) at (1,0.5) {};
\node[boxed] (Q) at (-2,1) {$Q$};
\node[dot,label=left:$y$] (y) at (-1,1) {};
\node[dot,label=below:$b$] (b) at (0,1) {};
\node[dot,label=above:$c$] (c) at (-1,1.5) {};
\node[dot,label=above:$d$] (d) at (0,1.5) {};

\draw
    (P) -- (x)
    (P) -- (a)
    (x) -- (y)
    (x) -- (b)
    (x) -- (Q)
    (y) -- (c)
    (y) -- (d)
    (a) -- (d)
    (b) -- (d);
\end{tikzpicture}
\end{center}
where $P=\{p_1, \ldots, p_m\}$ and $Q=\{q_1, \ldots, q_n\}$ are antichains. Thus the $p_i$ are the minimal elements (if $m>0$), while $c$, $d$, and the $q_i$ are the maximal elements. As we will see, vanishing of the defect for $(E,x,y)$ will provide the key relation to show nilpotence of $\delta_{m,n}$.

\begin{lemma} \label{lem:quad-1}
We have $[E, x] = -\delta_{m,n}$.
\end{lemma}

\begin{proof}
We have
\begin{displaymath}
E_{<x}=P, \qquad E_{>x} = Q \cup \{b,c,d,y\}, \qquad E_{\parallel x}=\{a\}.
\end{displaymath}
A cut for $(E,x)$ consists of $L,R \subset \{a\}$ such that $L \cup E_{<x} < R \cup E_{>x}$. Since $a \nless b$, we cannot have $a \in L$, and so $L=\emptyset$ is the only option. Both $R=\emptyset$ and $R=\{a\}$ give cuts. We have
\begin{displaymath}
\vert \Max(E_{<x}) \vert = m, \quad
\vert \Min(E_{>x}) \vert = n+2, \quad
\vert \Min(\{a\} \cup E_{>x}) \vert = n+3.
\end{displaymath}
We thus find
\begin{displaymath}
\lambda(E, x) = u^m v^{n+2} - u^m v^{n+3}.
\end{displaymath}
Clearly, $x$ does not have a twin in $E$. The result thus follows from Corollary~\ref{cor:weak3}.
\end{proof}

\begin{lemma} \label{lem:quad-2}
We have $[E, y]=0$.
\end{lemma}

\begin{proof}
We have
\begin{displaymath}
E_{<y}=P \cup \{x\}, \qquad E_{>y} = \{c,d\}, \qquad E_{\parallel y}=Q \cup \{a,b\}.
\end{displaymath}
Consider a cut $(L,R)$ for $(E,y)$. Since every element of $L$ must be $<c$, as $c \in E_{>y}$, we see that $L=\emptyset$. Since every element of $R$ must be $>x$, we see that $a$ cannot belong to $R$. These are the only constraints. The cuts are thus $(\emptyset, R)$, where $R$ is an arbitrary subset of $Q \cup \{b\}$. Let $R$ be a subset of $Q \cup \{b\}$, and write $\vert R \vert=r+\delta$, where $r=\vert R \cap Q \vert$ and $\delta$ is~1 if $b \in R$ and~0 otherwise. We have
\begin{displaymath}
\vert \Max(E_{<y}) \vert = 1, \qquad
\vert \Min(R \cup E_{>y}) \vert = 2+r.
\end{displaymath}
We thus find
\begin{displaymath}
\lambda(E,y) = \sum_{r=0}^n \sum_{\delta=0}^1 \binom{n}{r} (-1)^{r+\delta} u v^{2+r} = 0,
\end{displaymath}
as the $\delta=0$ terms cancel the $\delta=1$ terms. Clearly, $y$ has no twin in $E$. The result thus follows from Corollary~\ref{cor:weak3}.
\end{proof}

\begin{lemma} \label{lem:quad-3}
We have $[E^x, y] = \sum_{r=0}^n (-1)^r \binom{n}{r} \delta_{m,r}$.
\end{lemma}

\begin{proof}
We have
\begin{displaymath}
E^x_{<y}=P, \qquad E^x_{>y} = \{c,d\}, \qquad E^x_{\parallel y}=Q \cup \{a,b\}.
\end{displaymath}
Consider a cut $(L,R)$ for $(E^x, y)$. As in Lemma~\ref{lem:quad-2}, we must have $L=\emptyset$. However, there is now no restriction on $R$. We thus see that the cuts are $(\emptyset, R)$, where $R$ is any subset of $Q \cup \{a,b\}$. Write $\vert R \vert=r+\epsilon+\delta$, where $r=\vert R \cap Q \vert$, and $\epsilon$ and $\delta$ detect if $b$ and $a$ belong to $R$. We have
\begin{displaymath}
\vert \Max(E^x_{<y}) \vert = m, \qquad
\vert \Min(R \cup E^x_{>y}) \vert = 1+r+\max(1,\epsilon+\delta).
\end{displaymath}
In the second expression, the~1 comes from $c$, the $r$ from $R \cap Q$, and the final term from the identity
\begin{displaymath}
\vert \Min(\{d\} \cup (R \cap \{a,b\})) \vert = \max(1,\epsilon+\delta).
\end{displaymath}
We thus have
\begin{displaymath}
\lambda(E^x,y) = \sum_{r=0}^n \sum_{0 \le \delta, \epsilon \le 1} \binom{n}{r} (-1)^{r+\delta+\epsilon} u^m v^{1+r+\max(1,\epsilon+\delta)}.
\end{displaymath}
Now, we have
\begin{displaymath}
\sum_{0 \le \delta, \epsilon \le 1} (-1)^{\delta+\epsilon} v^{\max(1,\epsilon+\delta)} = v^2 - v,
\end{displaymath}
and so
\begin{displaymath}
\lambda(E^x,y) = \sum_{r=0}^n \binom{n}{r} (-1)^r u^m (v^{r+3} - v^{r+2}).
\end{displaymath}
As $y$ does not have a twin in $E^x$, the result follows from Corollary~\ref{cor:weak3}.
\end{proof}

\begin{proof}[Proof of Proposition~\ref{prop:jstab}]
In $\Theta(\fP)$, we have the identity
\begin{displaymath}
[E,x][E^x,y] = [E,y][E^y,x].
\end{displaymath}
Applying Lemmas~\ref{lem:quad-1},~\ref{lem:quad-2}, and~\ref{lem:quad-3}, we find
\begin{displaymath}
\delta_{m,n} \cdot \big( \sum_{r=0}^n (-1)^r \binom{n}{r} \delta_{m,r} \big) = 0
\end{displaymath}
for all $m,n \in \bN$. Fixing $m$ and proceeding by induction on $n$, we can assume that $\delta_{m,0}, \ldots, \delta_{m,n-1}$ are nilpotent. The above identity then shows that $\delta_{m,n}^2$ is a sum of nilpotents, and so $\delta_{m,n}$ is nilpotent. This completes the proof.
\end{proof}

\begin{remark}
The proof of Proposition~\ref{prop:jstab} relies on the cleverly constructed poset $E$. ChatGPT found this for $m=n=0$ by a brute force search, and then observed that one can add the $P$ and $Q$ vertices.
\end{remark}

\subsection{Some explicit equations} \label{ss:some-eq}

Let $\nu$ be a measure for $\fP$ valued in a field $k$, and let $\nu_{p,q} = \nu(j_{p,q})$. As we have seen, $\nu$ is completely determined by the parameters $\nu_{p,q}$ with $0 \le p,q \le 2$. These nine values cannot be chosen freely: they are subject to some constraints. Indeed, if $P$ is a poset and $x$ and $y$ are distinct points, the defect $\delta_{\nu}(P,x,y)$ must vanish. This amounts to the equation
\begin{displaymath}
\nu(P, x) \nu(P^x, y) = \nu(P, y) \nu(P^y, x).
\end{displaymath}
We can express the four factors above as affine linear combinations of the values $\nu_{p,q}$ for $0 \le p,q \le 2$ using \eqref{eq:meas}. In this way, the equation becomes a (possibly inhomogeneous) degree two polynomial equation in these quantities.

In Table~\ref{t:eqs}, we have listed 23 specific equations obtained in this way. We explain the notation. The $i$th poset is denoted $P_i$, and its underlying set is taken to be $[n]=\{1, \ldots, n\}$, where $n = \vert P_i \vert$. The column labeled $P_i$ specifies the poset structure by giving a generating set for the relations; each pair $ij$ in the brackets indicates a relation $i<j$. The underlined numbers indicate the two marked points ($x$ and $y$); if there is only one underlined point then one of the marked points is an isolated point (it does not matter which one). For instance, $P_{17}$ has underlying set $[5]$ and generating relations $2<5$ and $3<4$ and $3<5$, and the marked points are $x=1$ and $y=2$. The polynomial $h_i$ records the relation, that is, $h_i=0$ is the equation $\delta_{\nu}(P_i,x,y)=0$. To make the equations easier to read, we rename the variables as follows:
\begin{displaymath}
\begin{pmatrix}
x & b & q \\
a & y & t \\
p & s & z
\end{pmatrix}
=
\begin{pmatrix}
\nu_{0,0} & \nu_{0,1} & \nu_{0,2} \\
\nu_{1,0} & \nu_{1,1} & \nu_{1,2} \\
\nu_{2,0} & \nu_{2,1} & \nu_{2,2}
\end{pmatrix}.
\end{displaymath}
So, for instance, $h_1=0$ says that any measure satisfies $\nu_{0,0}=0$ or $\nu_{1,0}=\nu_{0,1}$. We solve the system of equations in \S \ref{ss:soln}.

We walk through a few examples to explain how to obtain the stated value of $h_i$.

\begin{example}
We explain the calculation of $h_1$. Let $P=[2]$ with order $1<2$. Then $h_1=0$ is the relation
\begin{displaymath}
\nu(P, 1) \nu(P^1, 2) = \nu(P, 2) \nu(P^2, 1).
\end{displaymath}
We have isomorphisms of marked posets
\begin{displaymath}
(P,1) \cong (J_{0,1}, \ast), \qquad
(P,2) \cong (J_{1,0}, \ast), \qquad
(P^1, 2) \cong (P^2, 1) \cong (J_{0,0}, \ast),
\end{displaymath}
and so
\begin{displaymath}
\nu(P,1) = \nu_{0,1} = b, \qquad
\nu(P,2) = \nu_{1,0} = a, \qquad
\nu(P^1,2) = \nu(P^2, 1) = \nu_{0,0} = x.
\end{displaymath}
Thus $h_1=x(a-b)$, as in the Table~\ref{t:eqs}.
\end{example}

\begin{example}
We explain the calculation of $h_7$. Let $P=[4]$ with order $2<3<4$; the element~1 is isolated. The marked points are~2 and~3, and so $h_7=0$ is the relation
\begin{displaymath}
\nu(P, 2) \nu(P^2, 3) = \nu(P, 3) \nu(P^3, 2).
\end{displaymath}
We have
\begin{align*}
\nu(P, 2) &= b-q & \nu(P^2, 3) &= b-q \\
\nu(P, 3) &= y & \nu(P^3, 2) &= b-q,
\end{align*}
which yields the formula for $h_7$ in the table. We explain the calculation of $\nu(P,2)$. We have $P_{\parallel 2} = \{1\}$. A cut $(L,R)$ for $(P,2)$ must satisfy $L < \{3,4\} \cup R$, and so we cannot have $1 \in L$. We thus find that there are two cuts, namely $(\emptyset, \emptyset)$ and $(\emptyset, \{1\})$. The $(p,q)$ invariants for these two cuts are $(0,1)$ and $(0,2)$. Since~2 does not have a twin in $P$, we find
\begin{displaymath}
[ P, 2 ] = j_{0,1} - j_{0,2}
\end{displaymath}
in $\Theta(\fP)$. Thus
\begin{displaymath}
\nu(P,2) = \nu_{0,1} - \nu_{0,2} = b-q,
\end{displaymath}
as claimed. The other calculations are similar.
\end{example}

\begin{example} \label{ex:eq-3}
We say a little about $h_{20}$. Let $P=[6]$ with order generated by
\begin{displaymath}
2<4, \quad 2<5, \quad 3<4, \quad 3<5.
\end{displaymath}
The points~1 and~6 are isolated. The marked points are~1 and~2. Thus $h_{20}=0$ is the relation
\begin{displaymath}
\nu(P, 1) \nu(P^1, 2) = \nu(P, 2) \nu(P^2, 1).
\end{displaymath}
We have
\begin{align*}
\nu(P, 1) &= - 1 + x - 3b + 2q - 3a + 4y - 2t + 2p - 2s + z \\
\nu(P^1, 2) &= - 1 - b + q - t \\
\nu(P, 2) &= - 1 - b + q - t \\
\nu(P^2, 1) &= -1 + x - 3b + 2q - 2a + 2y - t + p,
\end{align*}
which, after some simplification, yields the formula for $h_{20}$ in the table. We explain the formula for $\nu(P^1, 2)$. For $(P^1, 2)$, there are five cuts:
\begin{center}
\begin{tabular}{llll}
\toprule
$L$ & $R$ & $p$ & $q$ \\
\midrule
$\emptyset$ & $\emptyset$ & 0 & 2 \\
$\emptyset$ & $\{3\}$ & 0 & 1 \\
$\emptyset$ & $\{6\}$ & 0 & 3 \\
$\emptyset$ & $\{3,6\}$ & 0 & 2 \\
$\{3\}$ & $\emptyset$ & 1 & 2 \\
\bottomrule
\end{tabular}
\end{center}
Furthermore, 2 has one twin in $P^1$, namely, 3. We thus find
\begin{displaymath}
[ P^1, 2 ] = -1 + j_{0,2} - j_{0,1} - j_{0,3} + j_{0,2} - j_{1,2},
\end{displaymath}
in $\Theta(\fP)$, and so
\begin{displaymath}
\nu(P^1, 2) = -1 + q - b - t,
\end{displaymath}
as claimed. The other cases are similar, but more laborious: e.g., for $(P,1)$ there are~70 cuts.
\end{example}

\begin{remark}
We had ChatGPT create an interactive webpage \cite{webpage} where one can enter a marked poset and it will list all of the cuts, as well as the corresponding affine linear expression in the nine variables used in the table. Using this tool, one can, for example, easily obtain the list of 70 cuts mentioned in Example~\ref{ex:eq-3}.
\end{remark}

\begin{remark}
The equations in Table~\ref{t:eqs} were obtained by brute force: we (or really ChatGPT) generated all equations for $\vert P \vert \le 7$ and discarded redundant equations. This is a somewhat involved computational task, so one might worry that errors could have been made in the process. Fortunately, this is not important. For our proof, all that matters is that each equation appearing in the table is valid, and the table contains certificates for this (namely, the twice marked posets $P_i$). We have verified (independently of ChatGPT) that each $h_i$ has been correctly obtained from $P_i$.
\end{remark}

\begin{remark}
The collection $\{h_1, \ldots, h_{23}\}$ is stable under transpose: each $h_i$ is either its own transpose or adjacent to its transpose.
\end{remark}

\begin{table}
\caption{Some quadratic equations. See \S \ref{ss:some-eq} for an explanation.}
\label{t:eqs}
\renewcommand{\sc}[1]{\scalebox{0.7}{#1}}
\centering
\begin{tabular}{llll}
\toprule
\sc{$i$} & \sc{$\vert P_i \vert$} & \sc{$P_i$} & \sc{$h_i$} \\
\midrule
 \sc{1} & \sc{2} & \sc{$[\ul{1}\ul{2}]$} & \sc{$x(a-b)$} \\
 \sc{2} & \sc{3} & \sc{$[\ul{2}\ul{3}]$} & \sc{$(-a+b+p-q)(a+b-x+1)$} \\
 \sc{3} & \sc{3} & \sc{$[\ul{2}3]$} & \sc{$aq+bq-by-b-x q+q$} \\
 \sc{4} & \sc{3} & \sc{$[2\ul{3}]$} & \sc{$ap-ay-a+bp-x p+p$} \\
 \sc{5} & \sc{3} & \sc{$[\ul{1}\ul{2}; 23]$} & \sc{$b(y-b)$} \\
 \sc{6} & \sc{3} & \sc{$[12;\ul{2}\ul{3}]$} & \sc{$a(y-a)$} \\
 \sc{7} & \sc{4} & \sc{$[\ul{2}\ul{3}; 34]$} & \sc{$(-b+q)(-b+q+y)$} \\
 \sc{8} & \sc{4} & \sc{$[23;\ul{3}\ul{4}]$} & \sc{$(-a+p)(-a+p+y)$} \\
 \sc{9} & \sc{4} & \sc{$[\ul{3}4]$} & \sc{$(y+1)(q-b)$} \\
\sc{10} & \sc{4} & \sc{$[3\ul{4}]$} & \sc{$(y+1)(p-a)$} \\
\sc{11} & \sc{4} & \sc{$[23; 2\ul{4}]$} & \sc{$(y+1)(-b+q+y-t)$} \\
\sc{12} & \sc{4} & \sc{$[\ul{2}4; 34]$} & \sc{$(y+1)(-a+p+y-s)$} \\
\sc{13} & \sc{4} & \sc{$[\ul{2}3; 24]$} & \sc{$q(-a-2b+x+q+2y-t)$} \\
\sc{14} & \sc{4} & \sc{$[2\ul{4}; 34]$} & \sc{$p(-2a-b+x+p+2y-s)$} \\
\sc{15} & \sc{5} & \sc{$[15; 12; \ul{3}\ul{4}]$} & \sc{$(-a+b+p-q)(-a-2b+x+q+2y-t)$} \\
\sc{16} & \sc{5} & \sc{$[15; 25; \ul{3}\ul{4}]$} & \sc{$(-a+b+p-q)(-2a-b+x+p+2y-s)$} \\
\sc{17} & \sc{5} & \sc{$[\ul{2}5; 34; 35]$} & \sc{$(-b+q+y-t)(-a+p+y-s)$} \\
\sc{18} & \sc{5} & \sc{$[23; 24; 2\ul{5}]$} & \sc{$(-b+q+y-t)(-a+p-2y+t-2)$} \\
\sc{19} & \sc{5} & \sc{$[\ul{2}5; 35; 45]$} & \sc{$(-a+p+y-s)(-b+q-2y+s-2)$} \\
\sc{20} & \sc{5} & \sc{$[\ul{2}4; 25; 34; 35]$} & \sc{$(-b+q-t-1)(-a+p+2y-t-2s+z)$} \\
\sc{21} & \sc{5} & \sc{$[2\ul{4}; 25; 34; 35]$} & \sc{$(-a+p-s-1)(-b+q+2y-2t-s+z)$} \\
\sc{22} & \sc{6} & \sc{$[15; 16; 25; 26; \ul{3}\ul{4}]$} & \sc{$(-a+b+p-q)(-2a-2b+x+p+q+4y-2t-2s+z)$} \\
\sc{23} & \sc{7} & \sc{$[1\ul{2}; 2\ul{4}; 25; 35; 46; 47; 56; 57]$} & \sc{$(y-t-s)(y-t-s+z)$} \\
\bottomrule
\end{tabular}
\end{table}

\subsection{Solutions to the equations} \label{ss:soln}

We now solve the system of equations $h_1=\cdots=h_{23}=0$. We show that any solution is one of the matrices appearing in Table~\ref{t:matrix}, or the transpose of such a matrix. This will complete the proof of Theorem~\ref{thm:quad}.

We begin with some general observations. The equations $h_9, \ldots, h_{12}$ essentially make $y=-1$ the ``exceptional locus.'' Off of this locus, the equations collapse considerably. On the exceptional locus, equations $h_5, \ldots, h_8$ give two possibilities for each of $a$, $b$, $p-a$, and $q-b$. Our analysis breaks into cases along these dividing lines.

\textit{Case 1: $y \ne -1$.} From $h_9$ and $h_{10}$ we find $p=a$ and $q=b$. From $h_{11}$ and $h_{12}$, we $s=t=y$. From equation $h_{20}$, we find $z=y$. Our matrix thus has the form.
\begin{displaymath}
\begin{pmatrix}
x & b & b \\
a & y & y \\
a & y & y
\end{pmatrix}
\end{displaymath}
Now, $h_5$ implies $b=0$ or $b=y$, and $h_6$ implies $a=0$ or $a=y$. Also, if $a \ne b$ then $h_1$ implies $x=0$. If $a=b=0$ then we have $\AA$. If $a=0$ and $b\ne 0$ then $b=y$ and $x=0$, and we have $\BB_2$. If $a \ne 0$ and $b=0$ then $a=y$ and $x=0$, and we have $\BB_2^{\top}$. Finally, if $a$ and $b$ are both non-zero then $a=b=y$. Equation $h_{13}$ gives $a+b=x+y$, and so $x=y$, and we have $\BB_1$.

\textit{Case 2: $y = -1$.} From $h_5$ we see $b \in \{0, -1\}$ and from $h_6$ we see $a \in \{0, -1\}$. From $h_7$, we see $q=b+\epsilon$ with $\epsilon \in \{0, 1\}$ and from $h_8$ we see $p=a+\delta$ with $\delta \in \{0, 1\}$. We now break into subcases.

\textit{Case 2a: $\epsilon \ne \delta$.} Note $-a+b+p-q=\delta-\epsilon$ is non-zero, and so $h_2$, $h_{15}$, $h_{16}$, and $h_{22}$ give
\begin{displaymath}
x = a+b+1, \quad
t = \epsilon-1, \quad
s = \delta-1, \quad
z = \epsilon+\delta-1 = 0.
\end{displaymath}
Our matrix is now
\begin{displaymath}
\begin{pmatrix}
a+b+1 & b & b+\epsilon \\
a & -1 &  \epsilon-1 \\
a+\delta & \delta-1 & 0
\end{pmatrix}
\end{displaymath}
There are eight possible values for $(a, b, \delta, \epsilon)$ (since $\delta \ne \epsilon$), all of which yield solutions of the desired form, as follows:
\renewcommand{\sc}[1]{\scalebox{0.8}{#1}}
\begin{center}
\begin{tabular}{ccccl}
$a$ & $b$ & $\delta$ & $\epsilon$ & sol \\
\midrule
 0  &  0  &  0  & $+$ & \sc{$\CC_6^{\top}$} \\
 0  &  0  & $+$ &  0  & \sc{$\CC_6$} \\
 0  & $-$ &  0  & $+$ & \sc{$\CC_5^{\top}$} \\
 0  & $-$ & $+$ &  0  & \sc{$\CC_1$} \\
\end{tabular}
\qquad
\begin{tabular}{ccccl}
$a$ & $b$ & $\delta$ & $\epsilon$ & sol \\
\midrule
$-$ &  0  &  0  & $+$ & \sc{$\CC_1^{\top}$} \\
$-$ &  0  & $+$ &  0  & \sc{$\CC_5$} \\
$-$ & $-$ &   0 & $+$ & \sc{$\CC_2^{\top}$} \\
$-$ & $-$ & $+$ &  0  & \sc{$\CC_2$}
\end{tabular}
\end{center}

\textit{Case 2b: $\delta=\epsilon$ and $t \ne \epsilon-1$.} The first factors in $h_{17}$, $h_{18}$, and $h_{20}$ are non-zero, and so these equations give us:
\begin{displaymath}
t = -\epsilon, \quad
s = \epsilon-1, \quad
z = 0
\end{displaymath}
Since $y=-1$, we have $h_3=q(a+b-x+1)$, and so $h_3+h_{13}=q(\epsilon-t-1)$. Since $t \ne \epsilon-1$, it follows that $q=0$ and so $\epsilon=-b$. Our matrix thus has the form
\begin{displaymath}
\begin{pmatrix}
x & b & 0 \\
a & -1 & b \\
a-b & -b-1 & 0
\end{pmatrix}
\end{displaymath}
As usual, if $a \ne b$ then $h_1$ gives $x=0$. Every choice of $(a,b)$ gives a solution on our list: $(0,0)$ is $\BB_3$, $(0,-1)$ is $\CC_4^{\top}$, $(-1, 0)$ is $\CC_3$, and $(-1, -1)$ is $\BB_4^{\top}$.

\textit{Case 2c: $\delta=\epsilon$ and $s \ne \epsilon-1$.} This is analogous to Case~2b. Equations $h_{17}$, $h_{19}$, and $h_{21}$ give $t = \epsilon-1$, $s = -\epsilon$, and $z = 0$, while $h_4+h_{14}$ gives $p=0$, and so $\epsilon=-a$. (Note now that it is the second factor of $h_{17}$ which is non-zero.) Our matrix is:
\begin{displaymath}
\begin{pmatrix}
x & b & b-a \\
a & -1 & -a-1 \\
0 & a & 0
\end{pmatrix}
\end{displaymath}
Every choice of $(a,b)$ gives a solution on our list: $(0,0)$ is $\BB_3^{\top}$, $(0,-1)$ is $\CC_3^{\top}$, $(-1,0)$ is $\CC_4$, and $(-1,-1)$ is $\BB_4$.

\textit{Case 2d: $\delta=\epsilon=1$ and $s=t=\epsilon-1$.} From $h_{23}$ we find $z=1$. Our matrix is thus
\begin{displaymath}
\begin{pmatrix}
x & b & b+1 \\
a & -1 & 0 \\
a+1 & 0 & 1
\end{pmatrix}
\end{displaymath}
If $(a,b)=(-1,-1)$ then we have $\BB_5$. Otherwise, $p \ne 0$ or $q \ne 0$, and so $h_3$ or $h_4$ gives $x=a+b+1$. The remaining possible values for $(a,b)$ give solutions on our list: $(0,0)$ is $\CC_8$, $(0,-1)$ is $\CC_7$, and $(-1,0)$ is $\CC_7^{\top}$.

\textit{Case 2e: $\delta=\epsilon=0$ and $s=t=\epsilon-1$.} From $h_{23}$ we find $z=-1$. Our matrix is thus
\begin{displaymath}
\begin{pmatrix}
x & b & b \\
a & -1 & -1 \\
a & -1 & -1
\end{pmatrix}
\end{displaymath}
If $(a,b)=(0,0)$ then we have $\AA$. Otherwise, $p\ne 0$ or $q\ne 0$, and so $h_3$ or $h_4$ gives $x=a+b+1$. The remaining possible values for $(a,b)$ gives solutions on our list: $(-1, -1)$ is $\BB_1$, $(0,-1)$ is $\BB_2$, and $(-1,0)$ is $\BB_2^{\top}$.

We have thus shown that any solution to the equations is one of the matrices in Table~\ref{t:matrix}, up to transpose. This completes the proof of Theorem~\ref{thm:quad}.

\section{Existence of certain measures} \label{s:some-meas}

In \S \ref{s:some-meas}, we construct some measures for $\fP$. Precisely, we construct the one-parameter family of measures with matrix $\BB_1(t)$, and the three measures with matrices $\CC_3$, $\CC_4$, and $\CC_8$. We will construct the remaining measures from these four in \S \ref{s:monoid} using the monoid action.

\subsection{Statement of results}

Recall that we have weak measures $\lambda_{p,q}$ for $p,q \in \{0,1,\ast\}$ (\S \ref{ss:explicit}), as well as the weak measure $\tau$ (\S \ref{ss:linear}).

\begin{theorem} \label{thm:some-meas}
The following weak measures are measures:
\begin{enumerate}
\item $-\tau+t \lambda_{\ast,\ast}$, for any $t \in k$.
\item $-\tau+\lambda_{\ast,\ast}-\lambda_{1,\ast}-\lambda_{\ast,1}$.
\item $-\tau+\lambda_{0,0}+\lambda_{0,1}-\lambda_{\ast,0}-\lambda_{\ast,1}$
\item $-\tau+\lambda_{0,\ast}-\lambda_{0,0}-\lambda_{1,0}-\lambda_{\ast,1}$
\end{enumerate}
\end{theorem}

The matrices associated to the measures in the theorem are $\BB_1(t)$, $\CC_8$, $\CC_3$, and $\CC_4$. Thus these matrices are actually realized by measures.

The proof of Theorem~\ref{thm:some-meas} will take the entirety of \S \ref{s:some-meas}. We make a few remarks on the structure of the proofs. The measures in (a) and (b) admit global formulas. By a global formula, we mean a formula for the value of the measure on an embedding of posets that does not rely on factoring the embedding into a sequence of one-point extensions. This makes the proofs for (a) and (b) somewhat natural. We do not have global formulas in cases (c) and (d). However, the $\lambda_{p,q}$ appearing in (c) and (d) are exactly the ones that admit a simple combinatorial interpretation (Proposition~\ref{prop:lambda-calc}). For this reason, the proofs in these cases are elementary. These is a significant amount of case work, but Proposition~\ref{prop:cover-defect} helps to reduce it. Finally, we remind the reader that to show a weak measure is a measure, we must show that its defect vanishes (\S \ref{ss:weak}).

\subsection{The first measure} \label{ss:meas1}

Fix $t \in k$, and put
\begin{displaymath}
\nu = -\tau+t \lambda_{\ast,\ast}.
\end{displaymath}
We show that $\nu$ is a measure.  Recall the poset $\cE_n(P)$ of $n$-extensions of $P$ defined in \S \ref{ss:mobius}. Let $\mu$ be the M\"obius function of the poset $\cE_n(P)$. For $Q \in \cE_n(P)$, put $\omega_P(Q) = t^{\vert Q \vert-\vert P \vert}$, and define
\begin{displaymath}
\theta_P(Q) = \sum_{T \ge Q} \mu(Q, T) \omega_P(T).
\end{displaymath}
The following is the key result:

\begin{proposition} \label{prop:meas1}
Let $P \subset Q$ be an embedding of posets, write $Q=P \sqcup \{x_1, \ldots, x_n\}$, and thus regard $Q$ as an $n$-extension of $P$. Put $Q_i = P \sqcup \{x_1, \ldots, x_i\}$. Then
\begin{displaymath}
\theta_P(Q) = \nu(Q_1, x_1) \cdots \nu(Q_n, x_n).
\end{displaymath}
\end{proposition}

The value of $\theta_P(Q)$ is easily seen to be independent of the order of the points $x_1, \ldots, x_n$. Thus, when $n=2$, the proposition shows that the defect for $\nu$ vanishes identically, and so $\nu$ is a measure. Once we know that $\nu$ is indeed a measure, the proposition shows that $\nu(P \subset Q)$ is equal to $\theta_P(Q)$. This is the global formula for $\nu$.

We now start the proof of the proposition. We begin with the $n=1$ case.

\begin{lemma} \label{lem:meas1-3}
Proposition~\ref{prop:meas1} holds if $n=1$.
\end{lemma}

\begin{proof}
Suppose $n=1$, so that $Q=P \sqcup \{x\}$. Let $\cF \subset \cE_1(P)$ be the set of 1-extensions of $P$ above $Q$, and write $\cF=\cF_1 \sqcup \cF_2$, where $\cF_1$ are the improper 1-extensions and $\cF_2$ are the proper 1-extensions. Then
\begin{displaymath}
\theta_P(Q) = \sum_{T \in \cF_1} \mu(Q,T) + t \sum_{T \in \cF_2} \mu(Q, T)
\end{displaymath}
The sum over $\cF_1$ is $-\tau(P,x)$ by Proposition~\ref{prop:twin-mobius}, and the sum over $\cF_2$ is $\lambda_{\ast,\ast}(P,x)$ by Corollary~\ref{cor:euler-char}. The result follows.
\end{proof}

Let $Q$ be an $n$-extension of $P$ with marked points $x_1, \ldots, x_n$. Let $Q^{\flat}$ be the $(n-1)$-extension of $P$ with underlying set $P \cup \{x_1, \ldots, x_{n-1}\}$ and marked points $x_1, \ldots, x_{n-1}$.

\begin{lemma} \label{lem:sharp}
Let $Q$ be an $n$-extension of $P$, and put $\cE=\cE_n(P)_{\ge Q}$ and $\cE^{\flat} = \cE_{n-1}(P)_{\ge Q^{\flat}}$. Then the monotone function $\cE \to \cE^{\flat}$ given by $T \mapsto T^{\flat}$ admits a left adjoint, denoted $S \mapsto S^{\sharp}$. Precisely, this means that if $T \in \cE$ and $S \in \cE^{\flat}$ then $S^{\sharp} \le T$ if and only if $S \le T^{\flat}$. Moreover:
\begin{enumerate}
\item We have $(S^{\sharp})^{\flat}=S$ for any $S \in \cE^{\flat}$.
\item If $T \in \cE_n(P)$ then $T \ge Q$ if and only if $T^{\flat} \ge Q^{\flat}$ and $T \ge (T^{\flat})^{\sharp}$.
\end{enumerate}
\end{lemma}

\begin{proof}
Let $x_1, \ldots, x_n$ be the marked points in $Q$. First suppose $x_n \in Q^{\flat}$. Let $(S, y_1, \ldots, y_{n-1})$ be an element of $\cE^{\flat}$ with canonical function $\phi \colon Q^{\flat} \to S$. We define $S^{\sharp}$ to be the poset $S$ with marked points $(y_1, \ldots, y_{n-1}, \phi(x_n))$. It is a simple matter to verify the claims of the lemma, which we leave to the reader.

We now treat the case where $x_n \notin Q^{\flat}$. Let $(S, y_1, \ldots, y_{n-1})$ be an element of $\cE^{\flat}$ with canonical function $\phi \colon Q^{\flat} \to S$. We define $S^{\sharp}$ to be the set $S \sqcup \{y_n\}$, where $y_n$ is a formal symbol. The order on $S^{\sharp}$ is defined as follows. On $S$, it is the given order. For $a \in S$, we have $a<y_n$ if and only if $a \le \phi(b)$ for some $b \in Q_{<x_n}$; similarly, $a>y_n$ if and only if $a \ge \phi(c)$ for some $c \in Q_{>x_n}$. One readily verifies that this is indeed a partial order. The canonical function $\phi^{\sharp} \colon Q \to S^{\sharp}$ exists, and is the unique extension of $\phi$ satisfying $\phi^{\sharp}(x_n)=y_n$. This is clearly strictly monotone, and so $Q \le S^{\sharp}$. Hence $S \mapsto S^{\sharp}$ is a well-defined function $\cE^{\flat} \to \cE$. It is clear that (a) is satisfied.

We now verify the adjointness property. Maintain the above notation, and let $(T, z_1, \ldots, z_n)$ be an element of $\cE$ with canonical function $\theta \colon Q \to T$. If $S^{\sharp} \le T$ then applying $(-)^{\flat}$ and using (a) gives $S \le T^{\flat}$. Conversely, suppose $S \le T^{\flat}$. Let $\psi \colon S \to T^{\flat}$ be the canonical function. Then the canonical function $\psi^{\sharp} \colon S^{\sharp} \to T$ exists: it is the unique function extending $\psi$ and satisfying $\psi^{\sharp}(y_n)=z_n$. Suppose $a < b$ are elements of $S^{\sharp}$. We must show $\psi^{\sharp}(a) < \psi^{\sharp}(b)$. If $a$ and $b$ belong to $S$ this is clear, since $\psi$ is order preserving. Suppose $a \in S$ and $b=y_n$. Then, by definition, $a \le \phi(\ol{b})$ for some $\ol{b} \in Q_{<x_n}$. Thus $\psi^{\sharp}(a) \le \psi^{\sharp}(\phi(\ol{b}))$ since $\psi^{\sharp}$ is order preserving on $S$, and $\phi(\ol{b})$ belongs to $S$. We have $\psi^{\sharp}(\phi(b)) = \theta(b)$ by definition of the canonical function. Since $\theta$ is strictly monotone, it follows that $\theta(\ol{b})<\theta(x_n)=z_n$. Thus $\psi^{\sharp}(a)<z_n=\psi^{\sharp}(b)$, as required. The case where $a=y_n$ is similar.

We finally verify (b). If $T \ge Q$ then $T^{\flat} \ge Q^{\flat}$ since $(-)^{\flat}$ is a monotone map, while $T \ge (T^{\flat})^{\sharp}$ is the co-unit of the adjunction. Now suppose $T^{\flat} \ge Q^{\flat}$ and $T \ge (T^{\flat})^{\sharp}$. Then
\begin{displaymath}
Q \le (T^{\flat})^{\sharp} \le T
\end{displaymath}
as required. Here we use the assumption that $T^{\flat}$ belongs to $\cE^{\flat}$ to ensure that $(T^{\flat})^{\sharp}$ is defined; since this belongs to $\cE$ by definition, we have the first inequality. The second inequality is our other assumption.
\end{proof}

\begin{remark}
In Lemma~\ref{lem:sharp}, $S^{\sharp}$ is the push-out of $S$ and $Q$ over $Q^{\flat}$. We will study push-outs of posets in detail in \S \ref{ss:poset-cat}.
\end{remark}

For $Q$ as above, regard $Q$ as a 1-extension of $Q^{\flat}$, with marked point $x_n$. Put
\begin{displaymath}
\hat{\theta}_P(Q) = \theta_P(Q^{\flat}) \cdot \theta_{Q^{\flat}}(Q).
\end{displaymath}
Note that the second factor on the right is $\nu(Q, x_n)$ when $x_n \notin Q^{\flat}$ by Lemma~\ref{lem:meas1-3}. Thus the following lemma (combined with a simple inductive argument) completes the proof of Proposition~\ref{prop:meas1}.

\begin{lemma} \label{lem:meas1-2}
We have $\hat{\theta}_P(Q) = \theta_P(Q)$ for any $n$-extension $Q$ of $P$.
\end{lemma}

\begin{proof}
By M\"obius inversion, it suffices to show
\begin{displaymath}
\omega_P(Q) = \sum_{T \ge Q} \hat{\theta}_P(T).
\end{displaymath}
We fix $Q$ in what follows, and prove this identity. We break the right sum up according to the value of $S=T^{\flat}$ to obtain
\begin{displaymath}
\sum_{T \ge Q} \hat{\theta}_P(T) = \sum_{\substack{S \in \cE_{n-1}(P), \\ S \ge Q^{\flat}}} \sum_{\substack{T \in \cE_n(P), \\ T^{\flat}=S, T \ge S^{\sharp}}} \big( \theta_P(S) \cdot \theta_S(T) \big).
\end{displaymath}
Here we have made use of the various statements in Lemma~\ref{lem:sharp}. Now, the $\theta_P(S)$ pulls out of the inner sum. Moreover, in the inner sum $T$ is varying over those elements of $\cE_1(S)$ with $T \ge S^{\sharp}$. By M\"obius inversion, we have
\begin{displaymath}
\sum_{\substack{T \in \cE_1(S), \\ T \ge S^{\sharp}}} \theta_S(T) = \omega_S(S^{\sharp}) = t^{\delta},
\end{displaymath}
where $\delta=\vert Q \vert - \vert Q^{\flat} \vert$ . We thus have
\begin{displaymath}
\sum_{T \ge Q} \hat{\theta}_P(T) = t^{\delta} \sum_{\substack{S \in \cE_{n-1}(P), \\ S \ge Q^{\flat}}} \theta_P(S) = t^{\delta} \cdot t^{\vert Q^{\flat} \vert - \vert P \vert} = \omega_P(Q),
\end{displaymath}
as required. In the second step, we used M\"obius inversion again.
\end{proof}

\subsection{The second measure} \label{ss:meas4}

Put $\alpha=\tau+\lambda_{\ast,1}$ and $\beta = \tau + \lambda_{1,\ast}$. We must show that
\begin{displaymath}
\nu = \tau+\lambda_{\ast,\ast}-\alpha-\beta
\end{displaymath}
is a measure. We note that Proposition~\ref{prop:lambda-calc}(d) shows that $\alpha(P,x)$ is~1 if $x$ begins a series pair, and is~0 otherwise; $\beta(P,x)$ is described by the transposed rule.

Let $P$ be a finite partially ordered set, and let $S$ be a subset of $P$. We are interested in ways of modifying the order on $P$ that do not affect the order on $P \setminus S$. Given a subset $F$ of $P \times P$, we let $\preceq_F$ be the transitive relation generated by $a \preceq_F b$ if $a \le b$ or $(a,b) \in F$; note that $\preceq_F$ is only a quasi-order in general. We want $\preceq_F$ to restrict to $\le$ on $P \setminus S$. An obvious necessary condition for this is that for any $(a,b) \in F$ we have $a \in S$ or $b \in S$. Also, there is no reason for $F$ to contain any pairs that are already comparable, since this will not affect $\preceq_F$. Motivated by these observations, we let $\tilde{\cF}_S(P)$ be the set of subsets $F \subset P^2$ such that for all $(a,b) \in F$ we have $a \not\le b$, and $\{a,b\} \cap S$ is non-empty. We let $\cF_S(P)$ be the subset of $\tilde{\cF}_S(P)$ consisting of those $F$ for which the restriction of $\preceq_F$ to $P \setminus S$ is $\le$.

Now, define
\begin{displaymath}
\omega_S(P) = \sum_{F \in \cF_S(P)} (-1)^{\vert F \vert}.
\end{displaymath}
If $S=\{x\}$, we write $\omega_x(P)$ in place of $\omega_S(P)$, and similarly for $\cF_S(P)$. The following result shows why this definition is relevant.

\begin{lemma} \label{lem:meas4-1}
For a marked poset $(P,x)$, we have $\omega_x(P) = \nu(P,x)$.
\end{lemma}

\begin{proof}
Let $\Sigma$ be the set of pairs $(L,R)$ where $L$ is a subset of $P^x \setminus P_{<x}$, $R$ is a subset of $P^x \setminus P_{>x}$, and
\begin{displaymath}
P_{<x} \cup L \le P_{>x} \cup R
\end{displaymath}
Given $F \in \cF_x(P)$, put
\begin{displaymath}
L = \{ a \in P^x \mid (a,x) \in F \}, \qquad
R = \{ a \in P^x \mid (x,a) \in F \}.
\end{displaymath}
Note that
\begin{equation} \label{eq:F}
F = (L \times \{x\}) \sqcup (\{x\} \times R).
\end{equation}
We claim that $(L,R)$ belongs to $\Sigma$. Certainly $L$ is a subset of $P^x \setminus P_{<x}$ and $R$ is a subset of $P^x \setminus P_{>x}$. Suppose $a \in L$ and $b \in P_{>x} \cup R$. Then we have $a \preceq_F x \preceq_F b$, and so $a \le b$ since $\preceq_F$ restricts to $\le$ on $P^x$. Thus $L \le P_{>x} \cup R$, and similarly $P_{<x} \cup L \le R$. This establishes the claim. We thus have a well-defined map
\begin{displaymath}
\cF_x(P) \to \Sigma, \qquad F \mapsto (L,R)
\end{displaymath}
which is injective by \eqref{eq:F}. We claim that it is a bijection. Thus let $(L,R) \in \Sigma$ be given, and define $F$ by \eqref{eq:F}. We show $F \in \cF_x(P)$. It is clear that $F \in \tilde{\cF}_x(P)$. Suppose $a,b \in P^x$ satisfy $a \preceq_F b$; we show $a \le b$. By definition, there is a chain $a=z_1 \preceq_F \cdots \preceq_F z_r=b$ such that for each $i$ we have $z_i \le z_{i+1}$ or $(z_i,z_{i+1}) \in F$. If there is some $1<i<r$ such that $z_i \in P^x$ then we have $a \le z_i$ and $z_i \le b$ by induction on the length of the chain, and so $a \le b$. We are thus reduced to the case of a chain $a \preceq_F x \preceq_F b$. Thus $a \in L \cup P_{<x}$ and $b \in R \cup P_{>x}$. By definition of $\Sigma$, we have $a \le b$, as required. Thus $\preceq_F$ restricts to $\le$ on $P^x$, and so $F \in \cF_x(P)$. This establishes the claim.

From the above discussion we find
\begin{displaymath}
\omega_x(P) = \sum_{(L,R) \in \Sigma} (-1)^{\vert L \vert + \vert R \vert}.
\end{displaymath}
We now analyze this sum. We begin by observing that $\Sigma$ admits a natural decomposition, as follows. Suppose $(L,R) \in \Sigma$. If $a,b \in (P_{<x} \cup L) \cap (P_{>x} \cup R)$ then $a\le b$ and $b\le a$, and so $a=b$. Thus $P_{<x} \cup L$ and $P_{>x} \cup R$ have at most one common element. Let $\Sigma^0$ be the subset of $\Sigma$ where these sets are disjoint, and for $y \in P^x$, let $\Sigma_y$ be the subset where $y$ is the unique common element. Thus $\Sigma$ is partitioned into $\Sigma^0$ and the $\Sigma_y$ with $y \in P^x$. We analyze the corresponding pieces of the sum separately.

Consider $(L,R) \in \Sigma^0$. Then $L$ is disjoint from both $P_{<x}$ and $P_{>x}$, and thus contained in $P_{\parallel x}$; similarly for $R$. Moreover, we have a strict inequality $P_{<x} \cup L < P_{>x} \cup R$. Thus $(L,R)$ is a cut for $(P,x)$. It is clear that $\Sigma^0$ contains every cut, and so we see that $\Sigma^0$ is exactly the set of cuts for $(P,x)$. The contribution of $\Sigma^0$ to $\omega_x(P)$ is exactly $\lambda_{\ast,\ast}(P,x)$ by Corollary~\ref{cor:euler-char} and Proposition~\ref{prop:mu-cuts}.

Now fix $y \in P^x$ and consider $\Sigma_y$. If $(L,R) \in \Sigma_y$ then $P_{<x} \cup L \le P_{>x} \cup R$. Since $y$ belongs to both sides, we must have
\begin{displaymath}
P_{<x} \subset P_{\le y}, \quad P_{>x} \subset P_{\ge y}, \quad L \subset P_{\le y}, \quad R \subset P_{\ge y}.
\end{displaymath}
If either of the first two containments fails then $\Sigma_y$ is empty. Thus suppose they hold. Let $L_0=\emptyset$ if $y<x$ and $L_0=\{y\}$ otherwise. Similarly, let $R_0=\emptyset$ if $y>x$ and $R_0=\{y\}$ otherwise. Note that $y$ belongs to both $P_{<x} \cup L_0$ and $P_{>x} \cup R_0$. Let $\Sigma'_y$ be the set of pairs $(L',R')$ with $L' \subset P^x_{<y} \setminus P_{<x}$ and $R' \subset P^x_{>y} \setminus P_{>x}$. Then we have a bijection
\begin{displaymath}
\Sigma'_y \to \Sigma_y, \qquad (L',R') \mapsto (L_0 \cup L', R_0 \cup R')
\end{displaymath}
We thus find
\begin{displaymath}
\sum_{(L,R) \in \Sigma_y} (-1)^{\vert L \vert+\vert R \vert} = (-1)^{\vert L_0 \vert + \vert R_0 \vert} \sum_{(L',R') \in \Sigma'_y} (-1)^{\vert L' \vert + \vert R' \vert}
\end{displaymath}
The sum on the right is equal to~1 if both $P^x_{<y} \setminus P_{<x}$ and $P^x_{>y} \setminus P_{>x}$ are empty, and vanishes otherwise. We thus see that the above quantity is $+1$ if $y$ is a twin of $x$, $-1$ if $(x,y)$ or $(y,x)$ is a series pair, and vanishes otherwise. We thus find
\begin{displaymath}
\sum_{y \in P^x} \sum_{(L,R) \in \Sigma_y} (-1)^{\vert L \vert+\vert R \vert} = \tau(P,x) - \alpha(P,x) - \beta(P,x),
\end{displaymath}
which completes the proof.
\end{proof}

We next analyze $\omega_S(P)$ when $\vert S \vert = 2$.

\begin{lemma} \label{lem:meas4-2}
Let $P$ be a finite poset, let $x \ne y$ be elements of $P$, and let $S=\{x,y\}$. Then
\begin{displaymath}
\omega_S(P) = \omega_x(P) \cdot \omega_y(P^x).
\end{displaymath}
\end{lemma}

\begin{proof}
Suppose $F \in \tilde{\cF}_S(P)$. By definition, every pair in $F$ meets $S$. Let $F_x$ be the pairs in $F$ that meet $\{x\}$, and let $F_y$ be the remaining pairs in $F$, so that $F = F_x \sqcup F_y$. One easily sees that $F \mapsto (F_x, F_y)$ defines a bijection $\tilde{\cF}_S(P) \to \tilde{\cF}_x(P) \times \tilde{\cF}_y(P^x)$.

For $\preceq_F$ to restrict to $\le$ on $P^{x,y}$, it is obviously necessary that $\preceq_{F_x}$ does. Suppose now that we have $E \in \tilde{\cF}_x(P)$ such that $\preceq_E $ restricts to $\le$ on $P^{x,y}$, but does not restrict to $\le$ on $P^x$. This means there is some $z \in P^{x,y}$ such that $y \preceq_{F_x} z$ but $y \not\le z$, or $z \preceq_{F_x} y$ but $z \not\le y$; without loss of generality, suppose the first. We analyze the $F \in \cF_S(P)$ with $F_x=E$. Given $F \in \tilde{\cF}_S(P)$, let $F^{\dag}$ be the result of ``toggling'' $(y,z)$ in $F$; that is, $F^{\dag} = F \cup \{(y,z)\}$ if $(y,z) \notin F$, and $F^{\dag} = F \setminus \{(y,z)\}$ if $(y,z) \in F$. If $F_x=E$ then $\preceq_F$ and $\preceq_{F^{\dag}}$ coincide, since the relation $y \preceq z$ is already generated by $F_x$, and $F_x$ is unchanged by $(-)^{\dag}$. In particular, $F \in \cF_S(P)$ if and only if $F^{\dag} \in \cF_S(P)$. Thus $(-)^{\dag}$ is involutive on the terms in $\omega_S(P)$ with $F_x=E$, and each pair occurs with opposite signs. Thus the total contribution of these terms is~0.

The above analysis shows that we need only consider those $F$ for which $F_x \in \cF_x(P)$. Precisely,
\begin{displaymath}
\sum_{F \in \cF_S(P)} (-1)^{\vert F \vert} = \sum_{F_x \in \cF_x(P)} \sum_{\substack{F_y \in \tilde{\cF}_y(P^x), \\ F_x \sqcup F_y \in \cF_S(P)}} (-1)^{\vert F_x \vert + \vert F_y \vert}.
\end{displaymath}
We claim that for $F_x \in \cF_x(P)$ and $F_y \in \tilde{\cF}_y(P^x)$, we have $F_x \sqcup F_y \in \cF_S(P)$ if and only if $F_y \in \cF_y(P^x)$. Granted this, the right side above becomes $\omega_x(P) \cdot \omega_y(P^x)$, which establishes the lemma.

We now prove the claim. Fix $F_x \in \cF_x(P)$ and $F_y \in \tilde{\cF}_y(P^x)$, and put $F=F_x \sqcup F_y$. Suppose $F$ belongs to $\cF_S(P)$. If $a,b \in P^{x,y}$ satisfy $a \preceq_{F_y} b$ then certainly they also satisfy $a \preceq_F b$; but this implies $a \le b$ since $F \in \cF_S(P)$. We thus see that $\preceq_{F_y}$ restricts to $\le$ on $P^{x,y}$, and so $F_y \in \cF_y(P^x)$, as required.

Now suppose that $F_y \in \cF_y(P^x)$. We show that $F \in \cF_S(P)$, meaning $\preceq_F$ restricts to $\le$ on $P^{x,y}$. Consider a relation $a \preceq_F b$ with $a,b \in P^{x,y}$; we show that $a \le b$. By definition, there is a chain $a=z_1 \preceq_F \cdots \preceq_F z_r=b$, where for each $i$ we have $z_i \le z_{i+1}$ or $(z_i, z_{i+1}) \in F$. We assume that consecutive elements in the chain are distinct. Suppose $z_i=x$ for some $i$. Then we must have $z_{i-1} \le z_i$ or $(z_{i-1}, z_i) \in F_x$, since $F_y$ does not contain any pairs with $x$, and so $z_{i-1} \preceq_{F_x} x$. Similarly, $x \preceq_{F_x} z_{i+1}$. Thus $z_{i-1} \preceq_{F_x} z_{i+1}$ since $\preceq_{F_x}$ restricts to $\le$ on $P^x$. Hence we can eliminate any $x$'s from our chain. It then follows that for each $i$ we have $z_i \le z_{i+1}$ or $(z_i, z_{i+1}) \in F_y$, and so $a \preceq_{F_y} b$, and so $a \le b$ as required.
\end{proof}

Of course, $x$ and $y$ play a symmetric role in Lemma~\ref{lem:meas4-2}, and so we also have
\begin{displaymath}
\omega_S(P) = \omega_y(P) \cdot \omega_x(P^y).
\end{displaymath}
Combined with Lemma~\ref{lem:meas4-1}, it follows that the defect of $\nu$ vanishes, and so $\nu$ is indeed a measure. In fact, one can show that $\nu(P \setminus S \to P) = \omega_S(P)$ holds in general; this is the global formula.

\begin{remark}
Recall that $\cE_n(P)$ is the set of $n$-extensions of $P$, equipped with the order $T \le S$ if the canonical function $\phi \colon T \to S$ is strictly monotone, $x<y$ implies $\phi(x)<\phi(y)$. Let $\ol{\cE}_n(P)$ be the same set, but equipped with the order $T \le S$ if the canonical function is (weakly) monotone, i.e., $x \le y$ implies $\phi(x) \le \phi(y)$. Let $\ol{\mu}$ be the M\"obius function for $\ol{\cE}_n(P)$. Let $S$ be a subset of $P$ of size $n$, and put $Q= P \setminus S$. One can show
\begin{displaymath}
\omega_S(P) = \sum_{T \in \ol{\cE}_n(Q), P \le T} \ol{\mu}(P, T).
\end{displaymath}
It is possible to give an alternative proof of Theorem~\ref{thm:some-meas}(b) using this perspective.
\end{remark}

\subsection{An observation}

Recall that a weak measure $\nu$ is a measure if $\nu(e)=1$ and the defect $\delta_{\nu}(P,x,y)$ vanishes for all $(P, x, y)$. In fact, we only need to check vanishing of the defect on certain $(P,x,y)$. This will simplify the proofs of the remaining cases of Theorem~\ref{thm:some-meas}. Recall that $\lessdot$ is the covering relation in a poset (\S \ref{ss:poset}).

\begin{proposition} \label{prop:cover-defect}
Let $\nu$ be a weak measure on $\fP$ valued in a ring $k$. Suppose $\nu(e)=1$ and $\delta_{\nu}(P,x,y)=0$ whenever $x \lessdot y$. Then $\nu$ is a measure.
\end{proposition}

\begin{proof}
Let $P$ be a poset and let $x$ and $y$ be distinct elements of $P$. If $x<y$ and $y$ is not a cover of $x$, then $x$ and $y$ are separated (Proposition~\ref{prop:poset-sep}), and so $\delta_{\nu}(P,x,y)=0$ (Proposition~\ref{prop:sep}). If $x<y$ and $y$ is a cover of $x$ then $\delta_{\nu}(P,x,y)=0$ by assumption. If $x>y$ then $\delta_{\nu}(P,x,y) = - \delta_{\nu}(P,y,x)=0$. Finally, suppose $x$ and $y$ are incomparable. Consider the pre-amalgamation $(P,x,y)$. One amalgamation is $P$. In any other amalgamation $Q$, the elements $x$ and $y$ are comparable, and so $\delta_{\nu}(Q,x,y)=0$ by the cases considered above. Proposition~\ref{prop:defect-sum} thus shows $\delta_{\nu}(P,x,y)=0$. We have thus shown that $\delta_{\nu}$ vanishes identically, and so $\nu$ is a measure.
\end{proof}

\subsection{The third measure} \label{ss:meas2}

Let $\alpha=\tau+\lambda_{\ast,1}$, as above. We must show that
\begin{displaymath}
\nu = \lambda_{0,0}+\lambda_{0,1}-\lambda_{\ast,0}-\alpha
\end{displaymath}
is a measure. We recall the values of the individual terms given by Proposition~\ref{prop:lambda-calc}:
\begin{itemize}
\item $\lambda_{0,0}(P,x)$ is~1 if $x$ if isolated, and~0 otherwise.
\item $\lambda_{0,1}(P,x)$ is~0 if $x$ is not minimal. If $x$ is minimal, it equals $\epsilon-n$, where $\epsilon$ is~1 if $x$ has a unique cover and~0 otherwise, and $n$ is the number of upper twins of $x$.
\item $\lambda_{\ast,0}(P,x)$ is~1 if $x$ is the greatest element, and~0 otherwise.
\item $\alpha(P,x)$ is~1 if $x$ begins a series pair, and is~0 otherwise.
\end{itemize}
Let $\hat{P} = P \sqcup \{\bone\}$, where $x<\bone$ for all $x \in P$. Note that the $\hat{P}$ construction commutes with deleting points in $P$, i.e., if $x \in P$ then $(\hat{P})^x$ is obtained from $P^x$ by adjoining the maximal point $\bone$. We begin with the following useful observation.

\begin{lemma} \label{lem:meas2-1}
For a marked poset $(P,x)$, we have
\begin{displaymath}
\lambda_{0,0}(P,x)+\lambda_{0,1}(P,x) = \lambda_{0,1}(\hat{P}, x), \qquad
\lambda_{\ast,0}(P,x) + \alpha(P,x) = \alpha(\hat{P}, x)
\end{displaymath}
and so
\begin{displaymath}
\nu(P,x) = \lambda_{0,1}(\hat{P}, x) - \alpha(\hat{P}, x).
\end{displaymath}
\end{lemma}

\begin{proof}
We first prove the formula for $\lambda_{0,1}(\hat{P}, x)$. If $x$ is not minimal in $P$ then all three quantities vanish. Thus assume $x$ is minimal. Let $a=\lambda_{0,0}(P,x)$, and write $\lambda_{0,1}(P,x) = \epsilon - n$ and $\lambda_{0,1}(\hat{P}, x) = \epsilon' - n'$. Clearly, the twins of $x$ in $P$ are the same as those in $\hat{P}$, and so $n=n'$. We must therefore show $a+\epsilon = \epsilon'$. If $x$ is isolated in $P$ then $a=\epsilon'=1$ and $\epsilon=0$, as required. Assume $x$ is not isolated, so $a=0$. Since $x$ is minimal, it follows that $P_{>x}$ is non-empty. Thus the covers of $x$ in $P$ are the same as those in $\hat{P}$, and so $\epsilon=\epsilon'$. This establishes the formula.

We now prove the formula for $\alpha(\hat{P}, x)$. Put $a=\lambda_{\ast,0}(P,x)$, $b=\alpha(P,x)$, and $c=\alpha(\hat{P}, x)$. We consider three cases. (1) Suppose $x<y$ is a series pair in $P$. Then $x$ is not maximal in $P$, and $x<y$ remains a series pair in $\hat{P}$. Thus $a=0$ and $b=c=1$. (2) Suppose $x$ is greatest in $P$. Then $x$ does not begin a series pair in $P$, while $x<\bone$ is a series pair in $\hat{P}$. Thus $a=c=1$ and $b=0$. (3) Finally, suppose $x$ does not begin a series pair in $P$ and is not greatest in $P$. Then $x$ does not begin a series pair in $\hat{P}$ either; indeed, the only possible series pair is $x<\bone$, but this implies $x$ is greatest in $P$. We thus have $a=b=c=0$.
\end{proof}

The above lemma allows us to give a convenient formula for $\nu$.

\begin{lemma} \label{lem:meas2-2}
Let $(P,x)$ be a marked poset. Define quantities $\delta$, $\epsilon$, and $n$ as follows:
\begin{itemize}
\item $\delta$ is~1 if $x$ has a unique cover in $\hat{P}$, and~0 otherwise.
\item $\epsilon$ is~1 if $x$ begins a series pair in $\hat{P}$, and~0 otherwise.
\item $n$ is the number of upper twins of $x$ (in $P$ or $\hat{P}$).
\end{itemize}
Then:
\begin{displaymath}
\nu(P,x) = \begin{cases}
-\epsilon & \text{if $x$ is not minimal in $P$} \\
\delta-\epsilon-n & \text{if $x$ is minimal in $P$.}
\end{cases}
\end{displaymath}
\end{lemma}

\begin{proof}
This follows from Lemma~\ref{lem:meas2-1} and the formulas for $\lambda_{0,1}$ and $\alpha$.
\end{proof}

We now carry out two auxiliary computations of $\nu$.

\begin{lemma} \label{lem:meas2-3}
Let $x \lessdot y$ be elements of $P$, and suppose $P_{<y}=\{x\}$. Then $\nu(P,x)=0$.
\end{lemma}

\begin{proof}
The assumption implies $x$ is minimal. Let $\delta$, $\epsilon$, and $n$ be as in Lemma~\ref{lem:meas2-2}. Clearly, $x$ cannot have an upper twin since $x$ is the unique point below $y$, and so $n=0$. If $x$ has a unique upper cover then that cover is $y$, and $x<y$ is a series pair; thus $\epsilon=\delta$. The result follows.
\end{proof}

\begin{lemma} \label{lem:meas2-4}
Let $x \lessdot y \lessdot z$ be elements of $\hat{P}$, and assume $y<z$ is a series pair. Then $\nu(P,x) = \nu(P^y, x)$. (Note: $z=\bone$ is allowed.)
\end{lemma}

\begin{proof}
Let $\delta$, $\epsilon$, and $n$ be the invariants for $(P,x)$ in Lemma~\ref{lem:meas2-2}, and let $\delta'$, $\epsilon'$, and $n'$ be those for $(P^y, x)$. We claim that they agree, that is, $\delta=\delta'$, $\epsilon=\epsilon'$, and $n=n'$. Since $x$ is minimal in $P$ if and only if it is minimal in $P^y$, we will therefore have $\nu(P,x) = \nu(P^y, x)$ by Lemma~\ref{lem:meas2-2}.

We first look at $\delta$. Clearly, $y$ is a cover of $x$ in $\hat{P}$, and $z$ is a cover of $x$ in $\hat{P}^y$. If $a \notin \{x,y,z\}$, then $a$ is a cover of $x$ in $\hat{P}$ if and only if it is so in $\hat{P}^y$. This shows that $\delta=\delta'$.

We now look at $\epsilon$. Clearly, if $\delta=\delta'=0$ then $\epsilon=\epsilon'=0$ as well. Thus suppose $\delta=\delta'=1$, meaning $y$ is the unique cover of $x$ in $\hat{P}$, and $z$ is the unique cover of $x$ in $\hat{P}^y$. Since $\hat{P}_{<y} = \hat{P}^y_{<z}$, we see that $x$ is the unique point covered by $y$ in $\hat{P}$ if and only if $x$ is the unique point covered by $z$ in $\hat{P}^y$. Thus $x<y$ is a series pair in $\hat{P}$ if and only if $x<z$ is a series pair in $\hat{P}^y$. Hence $\epsilon=\epsilon'$.

Finally, we look at $n$. Let $a \notin \{x,y,z\}$. Suppose $a$ is an upper twin of $x$ in $\hat{P}$, meaning $\hat{P}_{>x} = \hat{P}_{>a}$. It is clear that $a$ is then an upper twin of $x$ in $\hat{P}^y$. Now, conversely, suppose that $a$ is an upper twin of $x$ in $\hat{P}^y$, meaning $\hat{P}^y_{>x} = \hat{P}^y_{>a}$. Then $a<z$. Since $y<z$ is a series pair, we thus have $a \le y$, and since $a \ne y$, we have $a<y$. Thus $y \in \hat{P}_{>a}$, and so $a$ and $x$ are upper twins in $\hat{P}$. We thus see that the upper twins of $x$ in $\hat{P}$ and $\hat{P}^y$ agree (note that $x$ is not an upper twin of $x$, $y$, or $z$ in either poset), and so $n=n'$.
\end{proof}

Finally, we reach the main result.

\begin{lemma}
$\nu$ is a measure.
\end{lemma}

\begin{proof}
By Proposition~\ref{prop:cover-defect}, it suffices to show that the defect $\delta_{\nu}(P, x, y)$ vanishes whenever $y$ is a cover of $x$ in a poset $P$. Thus fix a poset $P$ and a cover $x \lessdot y$, and put $\delta=\delta_{\nu}(P,x,y)$. We show that $\delta=0$. Note that since $y$ is not minimal in $P$, we have $\nu(P,y) = -\alpha(\hat{P}, y)$. We distinguish two cases.

\textit{Case 1: $x$ is the unique element below $y$.} By Lemma~\ref{lem:meas2-3}, we have $\nu(P,x)=0$.  If $y$ does not begin a series pair in $\hat{P}$ then $\nu(P,y)=0$, and so $\delta=0$. If $y$ does begin a series pair then Lemma~\ref{lem:meas2-4} gives $\nu(P^y,x) = \nu(P,x) = 0$, and so $\delta=0$.

\textit{Case 2: $x$ is not the unique element below $y$.} In this case, $y$ is not minimal in $P^x$. Removing $x$ does not change whether or not $y$ begins a series pair in $\hat{P}$, and so $\nu(P^x,y)=\nu(P,y)$. If $y$ does not begin a series pair then both of these quantities vanish and $\delta=0$. If $y$ does begin a series pair then Lemma~\ref{lem:meas2-4} gives $\nu(P,x)=\nu(P^y,x)$, and once again we find $\delta=0$.
\end{proof}

\subsection{The fourth measure}

As above, let $\alpha = \tau+\lambda_{\ast,1}$. We must show that
\begin{displaymath}
\nu=\lambda_{0,\ast}-\lambda_{0,0}-\lambda_{1,0}-\alpha
\end{displaymath}
is a measure. We once again begin by giving a more convenient formula for $\nu$. For a poset $P$, let $\check{P} = P \sqcup \{\bzero\}$, where $\bzero<x$ for all $x \in P$.

\begin{lemma} \label{lem:meas3-1}
Let $(P, x)$ be a marked poset. If $x$ is not maximal then
\begin{displaymath}
\nu(P,x) = \begin{cases}
+1 & \text{if $x$ is least and does not begin a series pair} \\
-1 & \text{if $x$ is not least and begins a series pair} \\
0 & \text{otherwise.} \end{cases}
\end{displaymath}
If $x$ is maximal and $\vert P \vert>1$ then $\nu(P,x) = n-\epsilon$, where $n$ is the number of lower twins to $x$, and $\epsilon$ is~1 if $x$ covers exactly one element of $\check{P}$, and~0 otherwise. If $P=\{x\}$ then $\nu(P,x)=0$.
\end{lemma}

\begin{proof}
By the transpose of Lemma~\ref{lem:meas2-1}, we have
\begin{displaymath}
\lambda_{0,0}(P,x) + \lambda_{1,0}(P,x) = \lambda_{1,0}(\check{P}, x).
\end{displaymath}
Suppose $x$ is not maximal. Then $\lambda_{1,0}(\check{P}, x)=0$, and so
\begin{displaymath}
\nu(P,x) = \lambda_{0,\ast}(P,x) - \alpha(P,x).
\end{displaymath}
This is exactly equal to the stated quantity. Now suppose $x$ is maximal and $\vert P \vert>1$. This implies $x$ is not least, and also that $x$ does not begin a series pair, and so $\lambda_{0,\ast}(P,x) = \alpha(P,x)=0$. Thus $\nu(P,x) = -\lambda_{1,0}(\check{P}, x)$, which is equal to the stated quantity.
\end{proof}

Now, fix a poset $P$ and a covering relation $x \lessdot y$, and let $\delta=\delta_{\nu}(P,x,y)$ be the defect. We show that $\delta=0$. This will show that $\nu$ is a measure by Proposition~\ref{prop:cover-defect}.

\begin{lemma}
If $x<y$ is a series pair then $\delta=0$.
\end{lemma}

\begin{proof}
Since $x<y$ is a series pair, the marked posets $(P^x, y)$ and $(P^y, x)$ are isomorphic. We consider two cases separately.

\textit{Case~1: $y$ is not maximal.} Letting $a$ be~1 if $x$ is least and~0 otherwise, we have $\nu(P,x)=a-1$ by Lemma~\ref{lem:meas3-1}. Letting $b$ be~1 if $y$ begins a series pair and~0 otherwise, we have $\nu(P,y) = -b$ by Lemma~\ref{lem:meas3-1}. Now, $x$ is least in $P^y$ if and only if it is so in $P$, and begins a series pair if and only if $y$ begins a series pair in $P$ (using the isomorphism of $(P^y,x)$ with $(P^x,y)$). We thus find $\nu(P^x,y)=a-b$ by Lemma~\ref{lem:meas3-1}. Therefore
\begin{displaymath}
\delta=(a-1)(a-b)- (-b)(a-b) = (a+b-1)(a-b) = 0
\end{displaymath}
as required.

\textit{Case~2: $y$ is maximal.} Since $x<y$ is a series pair, we have $\nu(P,y)=-1$ by Lemma~\ref{lem:meas3-1} (with $n=0$ and $\epsilon=1$). If $x$ is not least then $\nu(P,x)=-1$ by Lemma~\ref{lem:meas3-1}, and so $\delta=0$. If $x$ is least then $P=\{x,y\}$, and we find $\nu(P^x,y)=0$, from which $\delta=0$.
\end{proof}

\begin{lemma}
If $x<y$ is not a series pair and $x$ is not least then $\delta=0$.
\end{lemma}

\begin{proof}
We have $\nu(P,x)=0$ by Lemma~\ref{lem:meas3-1}. We show that $\nu(P,y)=0$ or $\nu(P^y,x)=0$.

\textit{Case~1: $y$ is not maximal.} Suppose $\nu(P,y)$ and $\nu(P^y,x)$ are both non-zero. Since $\nu(P,y)$ is non-zero, it follows from Lemma~\ref{lem:meas3-1} that $y$ begins a series pair, say $y<z$. Since $\nu(P^y,x)$ is non-zero and $x$ is not least, it follows from Lemma~\ref{lem:meas3-1} that $x$ begins a series pair in $P^y$. However, this implies that $x<y$ is a series pair, which is a contradiction. (Since $x \lessdot y$ and $y \lessdot z$, and $y$ is the unique element covered by $z$, it follows that $y$ is the only element strictly between $x$ and $z$. Thus $z$ covers $x$ in $P^y$, and so $x<z$ must be the series pair that $x$ begins in $P^y$. It now follows easily that $x<y$ is a series pair in $P$.)

\textit{Case~2: $y$ is maximal.} We have $\nu(P,y) = n-\epsilon$, with notation as in Lemma~\ref{lem:meas3-1}. We assume in what follows that $\nu(P^y,x)$ is non-zero, and show $n=\epsilon$, i.e., $\nu(P,y)=0$.

\textit{Case~2a: $x$ is maximal in $P^y$.} Since $x<y$ is not a series pair, it follows that $y$ covers an element of $P$ besides $x$, and so $\epsilon=0$. There can be no lower twin $z$ to $y$, for such a point would satisfy $x<z$, and then $x$ would not be maximal in $P^y$. Thus $n=0$ as well, as required.

\textit{Case~2b: $x$ is not maximal in $P^y$.} Since $x$ is neither maximal nor least in $P^y$ and $\nu(P^y, x)$ is non-zero, it follows from Lemma~\ref{lem:meas3-1} that $x$ begins a series pair $x<z$ in $P^y$. Suppose $\epsilon=1$. Then $x$ is the unique element covered by $y$, and so $P_{<y} = P_{\le x}$. Since $x<y$ is not a series pair, $x$ must be covered by some element other than $y$, which is necessarily $z$ (as it is the only cover of $x$ in $P^y$). Since $x<z$ is a series pair in $P^y$, we have $P^y_{<z} = P^y_{\le x}$, and so $z$ is a lower twin of $y$ (note $y \nless z$ since $y$ and $z$ cover $x$). It is the only one: indeed, if $a$ is a lower twin of $y$ then $x<a$, and so $z \le a$ (since $x<z$ is a series pair in $P^y$), and so $a=z$ (as otherwise $z$ would belong to $P_{<a}$). Hence $n=1$, as required.

Now suppose $\epsilon=0$, meaning $y$ covers some element besides $x$. Observe that $z$ is not a lower twin of $y$: indeed, either $y \in P_{<z}$ or $x$ is the greatest element of $P_{<z}$, and in either case $P_{<z} \ne P_{<y}$. Suppose $a$ is a lower twin of $y$. Since $x<y$, we have $x<a$. Again, since $x<z$ is a series pair in $P^y$, this implies $z \le a$, and so $z<a$. But $z \nless y$, since $x \lessdot y$, a contradiction. Thus $y$ has no lower twins, meaning $n=0$, as required.
\end{proof}

\begin{lemma}
If $x<y$ is not a series pair and $x$ is least then $\delta=0$.
\end{lemma}

\begin{proof}
We claim that $\nu(P,y) = \nu(P^x, y)$. First suppose $y$ is not maximal in $P$, and therefore not maximal in $P^x$ either. Of course, $y$ begins a series in $P$ if and only if it does so in $P^x$; moreover, $y$ is not least in $P^x$, for if it were then $x<y$ would be a series pair in $P$. Thus the claim follows. Now suppose $y$ is maximal in $P$, and therefore maximal in $P^x$ as well. Write $\nu(P,y)=n-\epsilon$ and $\nu(P^x,y)=n'-\epsilon'$ as in Lemma~\ref{lem:meas3-1}; note $\vert P \vert>2$ for otherwise $x<y$ would be a series pair. Since $x$ is least we have $P_{<y}=\{x\}$, and so $\epsilon=\epsilon'=1$; note that $\bzero$ is the unique element covered by $y$ in $\check{P}^x$. For an element $w \ne x$ of $P$, we have $P_{<w}=\{x\}$ if and only if $P^x_{<w}=\emptyset$ (since $x$ is least), which shows that $w$ is a lower twin of $y$ in $P$ if and only if it is so in $P^x$. Hence $n=n'$, as required.

Since $x$ is least in $P^y$ and not maximal (as $\vert P \vert>2$), Lemma~\ref{lem:meas3-1} shows $\nu(P^y,x)=1-a$, where $a=1$ if $x$ begins a series pair in $P^y$, and~0 otherwise. Let $b$ be the common value of $\nu(P,y) = \nu(P^x, y)$. Lemma~\ref{lem:meas3-1} gives $\nu(P,x)=1$. We thus have
\begin{displaymath}
\delta = 1 \cdot b - b \cdot (1-a) = ab
\end{displaymath}
In the remainder of the proof, we assume $a=1$ and show $b=0$.

Since $a=1$, it follows that $x$ begins a series pair in $P^y$, say $x<z$. Since $x$ is least, it follows that $z$ is least in $P^{x,y}$. We cannot have $z<y$ since $y$ covers $x$. We also cannot have $y<z$, as then $x<y$ would be a series pair. Thus $y \parallel z$.

Suppose $y$ is not maximal. Since $z$ is minimal in $P^{x,y}$, we have $z<P_{>y}$. Since $z \parallel y$, it follows that $y$ cannot begin a series pair. Thus $b=0$ by Lemma~\ref{lem:meas3-1}, as required.

Now suppose $y$ is maximal. Write $\nu(P,y)=n-\epsilon$ as above. We have already remarked that $\epsilon=1$ and $P_{<y}=\{x\}$. Since $z$ is least in $P^{x,y}$ and incomparable to $y$, we have $P_{<z}=\{x\}$. Thus $z$ is a lower twin of $y$. It is the only one since it is least in $P^{x,y}$. Thus $n=1$ and $b=0$, as required.
\end{proof}

\section{The monoid of marked posets} \label{s:monoid}

In \S \ref{s:monoid}, we define the monoid $\cP$ of marked posets and show that it acts on the space of measures. Using this action and the measures constructed in \S \ref{s:some-meas}, we construct the remaining measures.

\subsection{Composition} \label{ss:poset-comp}

Let $Q=(Q,y)$ be a marked poset. Given a poset $P$, we define a new poset, denoted either $Q \circ P$ or $\Pi_Q(P)$, by substituting $P$ into $y$; we call this operation \defn{composition}. Precisely, the set underlying $\Pi_Q(P)$ is $Q^y \sqcup P$. The order restricts to the given orders on $Q^y$ and $P$. For $a \in Q^y$ and $b \in P$, we declare $a<b$ if and only if $a<y$, and $b<a$ if and only if $y<a$. One readily verifies that this is a well-defined poset. If $\alpha \colon P \to P'$ is an embedding of posets, then there is a natural embedding $\Pi_Q(\alpha) \colon \Pi_Q(P) \to \Pi_Q(P')$, given by the identity on $Q^y$ and $\alpha$ on $P$. We thus see that $\Pi_Q \colon \fP \to \fP$ is a functor, where $\fP$ is regarded as a category whose morphisms are embeddings.

If $P=(P,x)$ is a marked poset, then $\Pi_Q(P)$ is naturally marked by $x$. If $T$ is a third poset, then we have a natural isomorphism $(Q \circ P) \circ T = Q \circ (P \circ T)$; in other words, $\Pi_Q \circ \Pi_P = \Pi_{Q \circ P}$. Moreover, the one-point marked poset is the identity for composition. The operation $\circ$ thus defines a monoidal structure on the category $\fP_{\ast}$ of marked posets. Letting $\cP$ be the set of isomorphism classes of $\fP_{\ast}$, we therefore find that $\cP$ is a monoid under composition.

\subsection{Measures} \label{ss:meas-monoid}

Fix a marked poset $Q=(Q,y)$. We now investigate how composition interacts with measures. The following is the key observation.

\begin{proposition}
Let $\alpha \colon S \to T$ and $\beta \colon S \to S'$ be poset embeddings, and let $(T'_i, \alpha'_i, \beta'_i)$ for $1 \le i \le r$ be the various amalgamations (up to isomorphism). Then the $\Pi_Q(T'_i)$, for $1 \le i \le r$, are exactly the amalgamations (up to isomorphism) of $\Pi_Q(T)$ and $\Pi_Q(S')$ over $\Pi_Q(S)$.
\end{proposition}

\begin{proof}
It is clear that each $\Pi_Q(T_i')$ is an amalgamation, and that they are pairwise non-isomorphic. We now show that they account for all amalgamations. Thus let $X$ be an amalgamation of $\Pi_Q(T)$ and $\Pi_Q(S')$ over $\Pi_Q(S)$. For notational simplicity, we identify $\Pi_Q(S)$, $\Pi_Q(T)$, $\Pi_Q(S')$ with subsets of $X$. Since $\Pi_Q(S)$ contains $Q^y$, we can write $X=Q^y \sqcup X'$. The posets $T$ and $S'$ embedd in $X'$, and $X'=T \cup S'$. Thus $X'$ is an amalgamation of $T$ and $S'$, and is therefore isomorphic to one of the $T'_i$. To complete the proof, it suffices to show that $X=\Pi_Q(X')$ as posets. Let $a \in Q^y$ and $b \in X'$. Then $b$ belongs to $T$ or $S'$; without loss of generality, suppose the former. In $\Pi_Q(T)$, we have $a<b$ if and only if $a<y$. Since $\Pi_Q(T)$ is an embedded poset in $X$, the same holds in $X$. Similarly, $b<a$ if and only if $y<a$. This shows that the order on $X$ is exactly the one on $\Pi_Q(X')$, as required.
\end{proof}

\begin{corollary}
Let $\nu$ be a measure on $\fP$. Then $\alpha \mapsto \nu(\Pi_Q(\alpha))$ defines a measure $\Pi_Q^*(\nu)$ on $\fP$.
\end{corollary}

The corollary implies that the monoid $\cP$ acts on $\Theta(\fP)$ by ring homomorphisms (on the left), and on the space of measures (on the right). We now compute how $\Pi_Q$ affects the matrix associated to a measure.

\begin{proposition} \label{prop:comp-meas}
Let $\nu$ be a measure on $\fP$, let $\theta=\Pi_Q^*(\nu)$, let $\nu_{p,q}=\nu(j_{p,q})$, and analogously define $\theta_{p,q}$. Let
\begin{displaymath}
r = \min(\vert \Max(Q_{<y}) \vert, 2), \qquad
s = \min(\vert \Min(Q_{>y}) \vert, 2).
\end{displaymath}
Let $S$ (resp.\ $T$) be the set of elements $z \in Q_{\parallel y}$ such that $z<Q_{>y}$ (resp.\ $Q_{<y}<z$). Let $\epsilon$ be~0 if $S$ is empty, and~1 otherwise, and analogously define $\delta$ for $T$. Then
\begin{displaymath}
\begin{pmatrix}
\theta_{0,0} & \theta_{0,1} & \theta_{0,2} \\
\theta_{1,0} & \theta_{1,1} & \theta_{1,2} \\
\theta_{2,0} & \theta_{2,1} & \theta_{2,2}
\end{pmatrix}
=
\begin{pmatrix}
\nu(Q,y) & \nu_{r,1} - \delta \nu_{r,2} & (1-\delta) \nu_{r,2} \\
\nu_{1,s}  - \epsilon \nu_{2,s} & \nu_{1,1} & \nu_{1,2} \\
(1-\epsilon) \nu_{2,s} & \nu_{2,1} & \nu_{2,2}
\end{pmatrix}
\end{displaymath}
\end{proposition}

Before giving the proof, we make two comments on the shape of the formula. First, the most striking feature is that the bottom right $2 \times 2$ block is unchanged by $\Pi_Q^*$. This gives an easy way to see that certain measures, such as $\CC_1$ and $\CC_8$, cannot be in the same $\cP$ orbit. Second, apart from the $(0,0)$ entry, the dependence on $Q$ is only through the invariants $(r, s, \delta, \epsilon)$.

\begin{proof}
We have
\begin{displaymath}
\theta_{p,q} = \theta(J_{p,q}, \ast) = \nu(Q \circ (J_{p,q}, \ast)).
\end{displaymath}
Since $Q \circ J_{0,0} \cong Q$, the formula for $\theta_{0,0}$ follows.

We now examine $\theta_{1,0}$. Let $P=Q \circ J_{1,0}$. We identify the set underlying $P$ with $Q^y \sqcup \{a,x\}$, where $a<x$, $a$ and $x$ compare to points in $Q^y$ as $y$ does, and $x$ is the marked point. Note that $x$ does not have a twin in $P$; indeed, if $z \in Q^y$ were a twin then we would have $z>a$, but then $z>y$ in $Q$, and so $z>x$ in $P$, a contradiction. We have
\begin{displaymath}
P_{<x}=\{a\} \cup Q_{<y}, \qquad
P_{>x} = Q_{>y}, \qquad
P_{\parallel x} = Q_{\parallel y}.
\end{displaymath}
Let $(L,R)$ be a cut for $(P,x)$. We must have $R=\emptyset$, since no element of $P_{\parallel x}$ is $>a$, and we must have $L \subset S$. In fact, for any $L \subset S$, the pair $(L, \emptyset)$ is a cut. We have $q(\emptyset)= \vert \Min(Q_{>y}) \vert$. The set $P_{<x}$ has $a$ as its greatest element, and so $p(\emptyset)=1$. If $L \subset S$ is non-empty then a maximal element of $L$ is incomparable to $a$, and so $p(L) \ge 2$. We thus have
\begin{displaymath}
\theta_{1,0} = \nu(P,x) =\sum_{L \subset S} (-1)^{\vert L \vert} \nu_{p(L), q(\emptyset)}
= \nu_{1,s} - \epsilon \nu_{2,s}.
\end{displaymath}
The term $\nu_{1,s}$ is the contribution from $L=\emptyset$. For all $L \ne \emptyset$, we have $\nu_{p(L), q(\emptyset)} = \nu_{2,s}$. The sum of $(-1)^{\vert L \vert}$ over non-empty subsets of $S$ is $-\epsilon$. The formula for $\theta_{0,1}$ is similar.

We now examine $\theta_{2,0}$. Again, put $P=Q \circ J_{2,0}$, and identify $P$ with $Q^y \sqcup \{a_1,a_2,x\}$, where $x$ is the marked point. The order on $P$ is as follows: $a_1<x$, $a_2<x$, and $a_1$, $a_2$, and $x$ compare to points in $Q^y$ as $y$ does. As in the previous case, $x$ does not have a twin, and the cuts are exactly those pairs $(L, \emptyset)$ with $L \subset S$. However, $\Max(L \cup P_{<x})$ now always contains both $a_1$ and $a_2$, and so we have $p(L) \ge 2$ for any cut. We thus find
\begin{displaymath}
\theta_{2,0} = \sum_{L \subset S} (-1)^{\vert L \vert} \nu_{p(L), q(\emptyset)} = (1-\epsilon) \nu_{2,s},
\end{displaymath}
since $\nu_{p(L), q(\emptyset)} = \nu_{2,s}$ for all terms in the sum. The formula for $\theta_{0,2}$ is similar.

We now examine the remaining cases. Let $p,q \ge 1$, and put $P = Q \circ J_{p,q}$. We identify $P$ with the set $Q^y \sqcup \{a_1, \ldots, a_p, b_1, \ldots, b_q, x\}$. Again, $x$ does not have a twin. Every element of $P_{\parallel x} = Q_{\parallel y}$ is incomparable to both $a_1 \in P_{<x}$ and $b_1 \in P_{>x}$. It follows that $(\emptyset, \emptyset)$ is the only cut. We have $p(\emptyset) = \vert \Max(P_{<x}) \vert = p$, and similarly $q(\emptyset)=q$. Thus $\theta_{p,q} = \nu_{p,q}$.
\end{proof}

We now specialize the above computation to a few important cases. For a $3 \times 3$ matrix $N=(n_{p,q})_{0 \le p,q \le 2}$, put
\begin{displaymath}
\bE(N) = \begin{pmatrix}
n_{0,0}-n_{1,0}-n_{0,1}-1 & n_{0,1}-n_{0,2} & 0 \\
n_{1,0}-n_{2,0} & n_{1,1} & n_{1,2} \\
0 & n_{2,1} & n_{2,2} \end{pmatrix},
\end{displaymath}
\begin{displaymath}
\bF(N) = \begin{pmatrix}
n_{1,0} & n_{1,1} & n_{1,2} \\
n_{1,0} & n_{1,1} & n_{1,2} \\
n_{2,0} & n_{2,1} & n_{2,2} \end{pmatrix},
\qquad
\bG(N) = \begin{pmatrix}
n_{2,0} & n_{2,1} & n_{2,2} \\
n_{1,0} & n_{1,1} & n_{1,2} \\
n_{2,0} & n_{2,1} & n_{2,2} \end{pmatrix}.
\end{displaymath}

\begin{proposition} \label{prop:EFG}
Let $\nu$ be a measure for $\fP$, and let $N=(\nu_{p,q})_{0 \le p,q \le 2}$ be the associated matrix. Let $Q$ be the discrete 2-element poset with one point marked.
\begin{enumerate}
\item The matrix for $\Pi_Q^*(\nu)$ is $\bE(N)$.
\item The matrix for $\Pi_{J_{1,0}}^*(\nu)$ is $\bF(N)$.
\item The matrix for $\Pi_{J_{2,0}}^*(\nu)$ is $\bG(N)$.
\end{enumerate}
\end{proposition}

\begin{proof}
This follows from Proposition~\ref{prop:comp-meas} and some simple calculations. We note that
\begin{displaymath}
[Q,\ast] = j_{0,0}-j_{1,0}-j_{0,1}-1
\end{displaymath}
holds in $\Theta(\fP)$, as one sees by computing twins and cuts; this accounts for the $(0,0)$ entry of $\bE(N)$.
\end{proof}

\begin{corollary} \label{cor:EFG}
Suppose $N$ is a matrix that comes from a measure. Then the matrices $\bE(N)$, $\bF(N)$, and $\bG(N)$ also come from measures.
\end{corollary}

\subsection{Existence of the remaining measures} \label{ss:remaining}

We now verify that the remaining matrices in Table~\ref{t:matrix} come from measures, thereby completing the proof of Theorem~\ref{mainthm}.

\begin{theorem} \label{thm:remaining}
Every matrix in Table~\ref{t:matrix} come from a measure.
\end{theorem}

\begin{proof}
We first suppose that $k$ has characteristic~0. Let $\bA^{3 \times 3}$ denote the affine space of $3 \times 3$ matrices, and let $X \subset \bA^{3 \times 3}$ be the set of matrices that come from measures. This is a Zariski closed subset; indeed, if $N \in \bA^{3 \times 3}$ and $\nu=\nu(N)$ is as in \S \ref{ss:quadthm}, then $\nu$ is a measure if and only if $\delta_{\nu}(P,x,y)=0$ for all $(P,x,y)$, and this is a quadratic polynomial equation in the entries of $N$, as we saw in \S \ref{ss:some-eq}. From Theorem~\ref{thm:some-meas}, we know that the matrices $\BB_1(t)$, $\CC_3$, $\CC_4$, and $\CC_8$ belong to $X$ for any $t \in k$. From Corollary~\ref{cor:EFG}, we know that $X$ is stable under the operations $\bE$, $\bF$, and $\bG$. We also know that $X$ is stable under transpose. We deduce the theorem from these pieces of information.

We have
\begin{displaymath}
\bF(\CC_8)=\CC_7, \quad
\bF(\CC_4^{\top}) = \CC_5^{\top}, \quad
\bG(\CC_4^{\top}) = \CC_6^{\top}, \quad
\bE(\CC_6) = \CC_1, \quad
\bF(\CC^{\top}_1) = \CC^{\top}_2.
\end{displaymath}
We thus see that $X$ contains all of the $\CC_i$'s. We have
\begin{displaymath}
\bE(\BB_1(t)) = \begin{pmatrix}
-t-1 & 0 & 0 \\
0 & t & t \\
0 & t & t \end{pmatrix} = \AA(t, -t-1),
\end{displaymath}
from which one finds
\begin{displaymath}
\bE^m(\BB_1(t)) = \AA(t, -t-m)
\end{displaymath}
for $m \ge 1$. We thus see that $X$ contains $\AA(t, -t-m)$ for all $t \in k$ and all positive integers $m$. The points $(t, -t-m)$, with $t \in k$ and $m \ge 1$, are Zariski dense in $\bA^2$ since $k$ has characteristic~0. As $X$ is Zariski closed, it follows that $X$ contains $\AA(s,t)$ for all $s,t \in k$. We have $\bF(\AA(s,t)) = \BB_2(t)$, and so this also belongs to $X$ for all $t$. Similarly, for a positive integer $m$, we have
\begin{displaymath}
\bE^m(\CC_8) = \BB_5(m-1), \quad
\bE^m(\CC_4) = \BB_4(m-1), \quad
\bE^m(\CC_3) = \BB_3(1-m),
\end{displaymath}
and so $X$ contains $\BB_3(t)$, $\BB_4(t)$, $\BB_5(t)$ for all $t$ (using a similar Zariski density argument). This completes the proof in characteristic~0.

We now handle the case where $k$ has positive characteristic. We first construct a measure with matrix $\AA(s,t)$, for arbitrary $s,t \in k$. By the characteristic~0 case, there is a measure valued in the rational function field $\bQ(S,T)$ with matrix $\AA(S,T)$. This means that there is a ring homomorphism $\phi \colon \Theta(\fP) \to \bQ(S,T)$ such that $\phi(j_{p,q})$ is the $(p,q)$ entry of $\AA(S,T)$ for $p,q \in \{0,1,2\}$; in particular, $\phi(j_{p,q})$ is equal to 0, $S$, or $T$ for such $p$ and $q$. By \eqref{eq:meas}, it follows that $\phi(j_{p,q})$ is equal to 0, $S$, or $T$ for all $p,q \in \bN$. Since the $j_{p,q}$ generate $\Theta(\fP)$ as a ring (Corollary~\ref{cor:linear}), we see that $\phi$ actually takes values in $\bZ[S,T]$. Composing $\phi$ with the ring homomorphism $\bZ[S,T] \to k$ given by $S \mapsto s$ and $T \mapsto t$, we obtain a $k$-valued measure with matrix $\AA(s,t)$. The other cases follow from the same argument.
\end{proof}

\subsection{The Zariski closure of the action} \label{ss:monoid-closure}

We close \S \ref{s:monoid} with a brief discussion illustrating that the action of $\cP$ on measures has rather rich structure. This material will not be used elsewhere, and we omit the details of the calculations. We assume $\operatorname{char}(k)=0$.

Given $Q \in \cP$ and a measure $\nu$, Proposition~\ref{prop:comp-meas} shows that the matrix entries of $\Pi^*_Q(\nu)$ are affine linear expressions in the matrix entries of $\nu$. In other words, there is an affine linear transformation $\sigma(Q)$ of $\bA^{3 \times 3}$ that describes the action on measures. Letting $\Aff(3 \times 3)$ denote the monoid of affine linear transformations of $\bA^{3 \times 3}$, we thus have a monoid anti-homomorphism\footnote{The action is by pull-back, and thus contravariant.}\textsuperscript{,}\footnote{There is a subtlety in the definition of $\sigma(Q)$. The matrices corresponding to measures do not span $\bA^{3 \times 3}$, and so the action of $\Pi^*_Q$ on measures does not determine a unique affine linear transformation. To resolve this, one can consider weak measures $\nu$ satisfying $\nu(e)=1$ and $\nu(j_{p,q})=\nu(j_{\ol{p},\ol{q}})$. The space of such weak measures is $\bA^{3 \times 3}$, and the arguments from \S \ref{ss:meas-monoid} show that $\cP$ acts on this space. The transformation $\sigma(Q)$ is uniquely determined by the action of $\Pi_Q^*$ on these weak measures.}
\begin{displaymath}
\sigma \colon \cP \to \Aff(3 \times 3).
\end{displaymath}
Thus $\cP$ admits a natural affine representation.

Let $r,s \in \{0,1,2\}$ and $\delta, \epsilon \in \{0,1\}$ be the invariants of $Q$ appearing in Proposition~\ref{prop:comp-meas}. It turns out that only the following~30 possibilities for $(r,s,\delta,\epsilon)$ actually occur:
\begin{itemize}
\item $(r,s,1,1)$ with any $r$ and $s$.
\item $(r,s,0,1)$ with $r>0$.
\item $(r,s,1,0)$ with $s>0$.
\item $(r,s,0,0)$ with and $r$ and $s$.
\end{itemize}
Let $\cP(r,s,\delta,\epsilon)$ be the subset of $\cP$ consisting of those elements with the given invariants. These sets partition $\cP$ into 30 pieces. It turns out that closure of $\sigma(\cP)$ is the disjoint union of the closures of the $\sigma(\cP(r,s,\delta,\epsilon))$. Moreover, if at least one of $\delta$ or $\epsilon$ vanishes then $\sigma(Q)$ is determined from $(r,s,\delta,\epsilon)$, and does not vary continuously; hence $\sigma(\cP(r,s,\delta,\epsilon))$ (and its closure) is a single point. The Zariski closure of $\sigma(\cP(r,s,1,1))$ is isomorphic to an affine space $\bA^d$; the dimension is as follows:
\begin{itemize}
\item If $r>0$ and $s>0$ then $d=3$.
\item If exactly one of $r$ or $s$ equal to~0 then $d=4$
\item If $r=s=0$ then $d=5$.
\end{itemize}
Thus the closure of $\sigma(\cP)$ is a disjoint union of 21 points, four copies of $\bA^3$, four copies of $\bA^4$, and a single $\bA^5$.

\section{Support} \label{s:support}

In \S \ref{s:support}, we determine the supports of the measures on $\fP$. We also show that nilpotent endomorphisms in the associated $\uPerm$ categories have trace zero.

\subsection{Fra\"iss\'e subclasses} \label{ss:subclass}

We now recall the classification of the Fra\"iss\'e subclasses of $\fP$ due to Schmerl. Let $0 \le n \le \infty$. Write $[n]$ for the set $\{1, \ldots, n\}$, with the convention that $[n]$ is empty if $n=0$ and the set of all positive integers if $n=\infty$. Let $\fS_n$ be the group of all permutations of $[n]$, and let $H=\Aut(\bQ, <)$ be the group of all order-preserving self-bijections of the rational numbers. We define three families of Fra\"iss\'e subclasses of $\fP$.

\textit{The $\fA$ family: bounded antichains.} Let $\fA_n$ consist of all antichains with at most $n$ elements. This is clearly a Fra\"iss\'e class. The Fra\"iss\'e limit is the antichain $A_n=[n]$, and its automorphism group is the symmetric group $\fS_n$.

\textit{The $\fB$ family: bounded disjoint unions of chains.} Let $\fB_n$ consist of those posets $X$ for which there is a decomposition $X = X_1 \sqcup \cdots \sqcup X_m$, with $m \le n$, such that each $X_i$ is totally ordered and $a \parallel b$ whenever $a \in X_i$, $b \in X_j$, and $i \ne j$. Let $B_n = [n] \times \bQ$, equipped with the order $(a,p)<(b,q)$ if $a=b$ and $p<q$. It is not difficult to see that $B_n$ is a homogeneous partial order with age $\fB_n$. This implies that $\fB_n$ is a Fra\"iss\'e class and $B_n$ is its limit. The automorphism group of $B_n$ is the wreath product $H \wr \fS_n$.

\textit{The $\fC$ family: ordinal sums of bounded antichains.} Let $\fC_n$ consist of those posets $X$ for which there is a decomposition $X = X_1 \sqcup \cdots \sqcup X_m$, for some $m$, such that each $X_i$ is an antichain with at most $n$ elements, and $a<b$ for all $a \in X_i$ and all $b \in X_j$ when $i<j$. Let $C_n = [n] \times \bQ$, equipped with the order $(a,p)<(b,q)$ if $p<q$. Again, it is not difficult to see that $C_n$ is homogeneous with age $\fC_n$, which implies that $\fC_n$ is a Fra\"iss\'e class with limit $C_n$. The automorphism group of $C_n$ is the wreath product $\fS_n \wr H$.

We note that $\fA_0=\fB_0=\fC_0$ and $\fB_1=\fC_1$. We are now ready to state the Schmerl's theorem\footnote{We note that Schmerl omits $\fA_0$ from his list.} \cite{Schmerl}:

\begin{theorem} \label{thm:schmerl}
The Fra\"iss\'e subclasses of $\fP$ are exactly:
\begin{itemize}
\item $\fP$ itself,
\item $\fA_n$ for $0 \le n \le \infty$,
\item $\fB_n$ for $1 \le n \le \infty$,
\item $\fC_n$ for $2 \le n \le \infty$.
\end{itemize}
The classes listed above are pairwise distinct.
\end{theorem}

\subsection{Regular measures} \label{ss:regular}

Let $\nu$ be a measure on $\fP$ valued in a field $k$. Recall that $\nu$ is \defn{regular} if $\nu(P)=\nu(\emptyset \to P)$ is non-zero for every poset $P$. Equivalently, this means that $\nu$ has support $\fP$.

\begin{proposition} \label{prop:noreg}
No measure for $\fP$ is regular.
\end{proposition}

\begin{proof}
Let $E$ be the poset depicted by the following Hasse diagram:
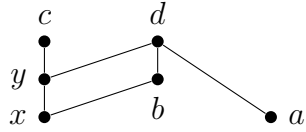
\begin{center}
\begin{tikzpicture}[
    x=1.5cm,
    y=1cm,
    dot/.style={
        circle,
        fill,
        inner sep=1.6pt
    },
]

\node[dot,label=left:$x$] (x) at (-1,0.5) {};
\node[dot,label=right:$a$] (a) at (1,0.5) {};
\node[dot,label=left:$y$] (y) at (-1,1) {};
\node[dot,label=below:$b$] (b) at (0,1) {};
\node[dot,label=above:$c$] (c) at (-1,1.5) {};
\node[dot,label=above:$d$] (d) at (0,1.5) {};

\draw
    (x) -- (y)
    (x) -- (b)
    (y) -- (c)
    (y) -- (d)
    (a) -- (d)
    (b) -- (d);
\end{tikzpicture}
\end{center}
This is the poset $E$ used in \S \ref{ss:jstab} with $n=m=0$. By Lemma~\ref{lem:quad-2}, $[E,y]=0$ in $\Theta(\fP)$, and so $[E]=[E,y] \cdot [E^y]$ vanishes as well. Thus any measure $\nu$ satisfies $\nu(E)=0$, and is therefore not regular.
\end{proof}

Proposition~\ref{prop:noreg} has a useful consequence. Although a measure on $\fP$ is never regular, its restriction to its support is regular, and the proposition combined with Schmerl's theorem implies that the support is one of the much simpler classes $\fA_n$, $\fB_n$, or $\fC_n$. We use this observation to control traces in the associated $\uPerm$ tensor categories.

Before getting into the details, we recall a bit of background. Measures were introduced in \cite{repst} for the purpose of constructing tensor categories: given an oligomorphic group $G$ and a measure $\nu$ on $G$, we defined a tensor category $\uPerm(G, \nu)$. In \cite[\S 4]{arboreal}, we translated this construction to the language of Fra\"iss\'e classes: given a Fra\"iss\'e class $\fF$ and a measure $\nu$ on $\fF$, we defined a tensor category $\uPerm(\fF, \nu)$. These $\uPerm$ tensor categories are essentially never abelian, and an important problem is to determine when they embed into a pre-Tannakian category. A necessary condition is that every nilpotent endomorphism in the $\uPerm$ category have trace~0; as far as we currently know, this condition could be sufficient. Using the above proposition, we can verify this important condition in characteristic~0.

\begin{corollary} \label{cor:nilp}
Let $\nu$ be a measure on $\fP$ valued in a field $k$ of characteristic~0. Then any nilpotent endomorphism in the category $\uPerm(\fP, \nu)$ has trace~0.
\end{corollary}

\begin{proof}
Recall (\S \ref{ss:support}) that $\fP_{\nu}$ is the support of $\nu$, and $\ol{\nu}$ is the regular measure induced on $\fP_{\nu}$. By \cite[\S 4.6]{arboreal}, there is a natural tensor functor
\begin{displaymath}
\Phi \colon \uPerm(\fP, \nu) \to \uPerm(\fP_{\nu}, \ol{\nu}).
\end{displaymath}
Since $\Phi$ preserves trace and maps nilpotent endomorphisms to nilpotent endomorphisms, it is enough to show that nilpotent endomorphisms in $\uPerm(\fP_{\nu}, \ol{\nu})$ have trace~0. By Theorem~\ref{thm:schmerl} and Proposition~\ref{prop:noreg}, $\fP_{\nu}$ is one of the classes $\fA_n$, $\fB_n$, or $\fC_n$. In \cite{repst}, we proved that nilpotent endomorphisms have trace~0 for all measures on $\fA_{\infty}$ and $\fB_1$ in characteristic~0. The other cases can be handled by the same methods.
\end{proof}

\subsection{Supports} \label{ss:supp-P}

\begin{table}
\caption{Supports of measures. See \S \ref{ss:subclass} and \S \ref{ss:supp-P} for notation.} \label{t:support}
\begin{tabular}{l|cccccc}
\toprule
Measure &
$\AA(s,t)$ &
$\BB_1(t)$ &
$\BB_2(t)$ &
$\BB_3(t)$ &
$\BB_4(t)$ &
$\BB_5(t)$
\\
Support &
$\fA_t$ &
$\fC_{-t}$ &
$\fA_0$ &
$\fA_t$ &
$\fB_{-t}$ &
$\fB_{-t}$ \\
\bottomrule
\end{tabular}
\vskip 8pt
\begin{tabular}{l|cccccccc}
\toprule
Measure &
$\CC_1$ &
$\CC_2$ &
$\CC_3$ &
$\CC_4$ &
$\CC_5$ &
$\CC_6$ &
$\CC_7$ &
$\CC_8$
\\
Support &
$\fA_0$ &
$\fB_1$ &
$\fA_0$ &
$\fA_0$ &
$\fA_0$ &
$\fA_1$ &
$\fA_0$ &
$\fA_1$ \\
\bottomrule
\end{tabular}
\end{table}

We now determine the supports of the measures on $\fP$. If $\nu$ is a measure, then the support of $\nu^{\top}$ is the transpose of the support of $\nu$. In fact, since every Fra\"iss\'e subclass of $\fP$ is stable under transpose (by Theorem~\ref{thm:schmerl}), it follows that $\nu$ and $\nu^{\top}$ have the same support. It thus suffices to determine the support of one member of each orbits under transpose.

We introduce a convenient notational device. Let $t$ be an element of our coefficient field $k$. If $k$ has characteristic~0, we define $\fA_t$ to be $\fA_{\infty}$ if $t$ does not belong to $\bN$, and $\fA_n$ if $t=n$ does belong to $\bN$. If $k$ has positive characteristic, we define $\fA_t$ to be $\fA_{\infty}$ if $t$ does not belong to the prime subfield, and $\fA_n$ if it does and $n \in \bN$ is a minimal representative of $t$. We similarly define $\fB_t$ and $\fC_t$.

\begin{theorem} \label{thm:support}
The supports of measures on $\fP$ are as listed in Table~\ref{t:support}.
\end{theorem}

\begin{proof}
Let $\nu$ be a measure on $\fP$ valued in a field $k$, and let $N$ be the associated matrix, i.e., $N_{p,q}=\nu(j_{p,q})$ for $0 \le p,q \le 2$. By Theorem~\ref{thm:schmerl} and Proposition~\ref{prop:noreg}, the support of $\nu$ is one of the classes $\fA_n$, $\fB_n$, or $\fC_n$. We use a few small test posets to determine the family.

\textit{Observation 1.} We have $\nu(J_{0,0})=N_{0,0}$, where $J_{0,0}$ is the one-element poset. If $N_{0,0}=0$ then $\nu$ vanishes on all posets and its support is $\fA_0$.

\textit{Observation 2.} Let $A_2$ be a two-element antichain, and let $x \in A_2$. Then $[A_2,x] = -1+j_{0,0}-j_{1,0}-j_{0,1}$ in $\Theta(\fP)$, and so
\begin{displaymath}
\nu(A_2) = N_{0,0} \cdot (-1+N_{0,0}-N_{1,0}-N_{0,1}).
\end{displaymath}
If $\nu(A_2)=0$ then the support of $\nu$ is $\fA_0$, $\fA_1$, or $\fB_1$.

\textit{Observation 3.} We have
\begin{displaymath}
\nu(J_{0,1}) = N_{0,0} \cdot N_{0,1} = N_{0,0} \cdot N_{1,0}.
\end{displaymath}
The two expressions come from factoring $\emptyset \to J_{0,1}$ in two ways; this is just the $h_1=0$ equation from Table~\ref{t:eqs}. If $P$ is a poset that is not an antichain then there is an embedding $J_{0,1} \to P$, and so $\nu(J_{0,1})=0$ implies $\nu(P)=0$. We thus see that $\nu(J_{0,1})=0$ if and only if the support of $\nu$ is $\fA_n$ for some $0 \le n \le \infty$.

\textit{Observation 4.} We have
\begin{displaymath}
\nu(J_{0,2}) = N_{0,2} \cdot \nu(A_2),
\end{displaymath}
where $A_2$ is a two-element antichain. The poset $J_{0,2}$ belongs to each $\fC_n$ (with $n \ge 2)$ and does not belong to any $\fA_n$ or $\fB_n$. Thus $\nu(J_{0,2})=0$ if and only if the support of $\nu$ is some $\fA_n$ or $\fB_n$.

We can deduce quite a bit from the above observations:
\begin{itemize}
\item If $N$ has type $\BB_2$ or some $\CC_i$ then the support is as stated in Table~\ref{t:support}.
\item If $N$ has type $\AA$ or $\BB_3$ then the support is some $\fA_n$.
\item If $N$ is $\BB_4(t)$ or $\BB_5(t)$ then the support is some $\fA_n$ or $\fB_n$. Moreover, if $t \ne 0$ then $\nu(J_{0,1}) \ne 0$, and so the support is some $\fB_n$.
\item If $N=\BB_1(t)$ and $t \notin \{0,-1\}$ then $\nu(J_{0,2}) \ne 0$, and so the support of $\nu$ is some $\fC_n$ (since we know the support is not $\fP$ by Proposition~\ref{prop:noreg}).
\end{itemize}
To complete the proof, we must compute a little bit more to determine the precise index in some cases.

Let $A_r$ be the $r$-element antichain, and let $\ast$ be an arbitrary point. A simple computation with twins and cuts shows
\begin{displaymath}
[A_{r+1}, \ast] = -r + j_{0,0} + \sum_{p=1}^r (-1)^p \binom{r}{p} j_{p,0} + \sum_{q=1}^r (-1)^q \binom{r}{q} j_{0,q}.
\end{displaymath}
Thus if $\nu$ is a measure with matrix $N$ then
\begin{displaymath}
\nu(A_{r+1}, \ast) = N_{0,0} - r - r(N_{1,0}+N_{0,1}) + (r-1)(N_{2,0}+N_{0,2}).
\end{displaymath}
This gives
\begin{displaymath}
\nu(A_{r+1}, \ast) = \begin{cases}
t-r & \text{if $N$ is $\AA(s,t)$ or $\BB_3(t)$} \\
t+r & \text{if $N$ is $\BB_4(t)$ or $\BB_5(t)$} \\
-t-t & \text{if $N$ is $\BB_1(t)$.} \end{cases}
\end{displaymath}
The result now follows easily. We spell out the details in one case. If $N=\AA(s,t)$ then $\nu(A_r) = t(t-1) \cdots (t-r+1)$. We therefore see that $\nu(A_r)=0$ if and only if $t$ belongs to the subset $\{0,\ldots,r-1\}$ of $k$. Since we know the support of $\nu$ is some $\fA_n$, this implies that the support is $\fA_t$.
\end{proof}

Comparing Table~\ref{t:support} with Schmerl's classification, we find:

\begin{corollary} \label{cor:support}
Every proper Fra\"iss\'e subclass of $\fP$ occurs as the support of a measure valued in a field of characteristic~0.
\end{corollary}

\section{The oligomorphic group} \label{s:oligo}

In \S \ref{s:oligo}, we recall how measures on Fra\"iss\'e classes relate to measures for oligomorphic groups in general. We then examine this correspondence in the case of $\fP$. The main result is that measures on $\fP$ naturally correspond to measures on the associated oligomorphic group in characteristic~0.

\subsection{Background}

Let $\fF$ be a Fra\"iss\'e class, let $\Omega$ be its Fra\"iss\'e limit, and let $G$ be the automorphism group of $\Omega$. Thus $(G, \Omega)$ is an oligomorphic group.

For a finite subset $A$ of $\Omega$, let $G(A)$ be the subgroup of $G$ consisting of elements that fix each element of $A$. These groups form a neighborhood basis of the identity for a topology on $G$ \cite[\S 2.2]{repst}. We say that an action of $G$ is \defn{smooth} if every stabilizer is open, and \defn{finitary} if it has finitely many orbits; we use the term \defn{$G$-set} to mean ``set equipped with a finitary smooth action.'' We let $\bS(G)$ denote the category of $G$-sets. Note that the transitive $G$-sets are exactly the quotients of orbits on powers of $\Omega$. See \cite[\S 2.3]{repst} for details.

A \defn{stabilizer class} is a collection $\sE$ of open subgroups of $G$ satisfying certain conditions; see \cite[\S 2.6]{repst}. Given a stabilizer class $\sE$, we say that a $G$-set is \defn{$\sE$-smooth} if the stabilizer of any point belongs to $\sE$. We let $\bS(G, \sE)$ be the category of such sets. The collection $\sE(\Omega)$ of open subgroups of the form $G(A)$ is a stabilizer class, and the only example we will require. The transitive $\sE(\Omega)$-smooth $G$-sets are exactly the orbits on powers of $\Omega$.

A \defn{measure} for $G$ valued in a ring $k$ is a rule $\nu$ that assigns to each map $f \colon Y \to X$ of transitive objects in $\bS(G)$ a value $\nu(f)$ in $k$ such that certain axioms hold; see \cite[Definition~3.12]{repst}. There is a universal measure valued in a ring $\Theta(G)$. There is a variant: a measure for $G$ \defn{relative} to a stabilizer class $\sE$ is defined in the same way, but restricting to maps of objects in $\bS(G, \sE)$. There is again a universal ring $\Theta(G, \sE)$.

We now have three notions of measures: those for $\fF$, those for $G$, and those for $G$ relative to $\sE(\Omega)$. There is one straightforward comparison: measures for $\fF$ correspond bijectively to measures for $G$ relative to $\sE(\Omega)$, i.e., there is a natural ring isomorphism $\Theta(\fF) = \Theta(G, \sE(\Omega))$ \cite[Theorem~6.9]{repst}. In fact, this correspondence motivated the relative notion. In general, comparing measures for $G$ to their relative counterparts can be difficult. However, there is one useful result. We say that a stabilizer class $\sE$ is \defn{large} if every open subgroup of $G$ contains some member of $\sE$ with finite index. In this case, measures for $G$ and measures for $G$ relative to $\sE$ correspond in characteristic~0, that is, we have a natural ring isomorphism $\Theta(G) \otimes \bQ = \Theta(G, \sE) \otimes \bQ$ \cite[Proposition~2.10]{distal}.

The following well-known result gives a useful criterion for $\sE(\Omega)$ to be large.

\begin{proposition} \label{prop:large}
Suppose $G$ has no fixed points on $\Omega$. Then the following are equivalent:
\begin{enumerate}
\item For all finite subsets $A,B \subset \Omega$, the subgroup of $G$ generated by $G(A)$ and $G(B)$ is equal to $G(A \cap B)$.
\item For every open subgroup $U$ of $G$, there is a unique finite subset $A$ of $\Omega$ such that $G(A) \subset U \subset G\{A\}$. Here $G\{A\}$ is the subgroup of $G$ consisting of elements $g$ such that $gA=A$.
\item The stabilizer class $\sE(\Omega)$ is large, and every finite subset $A$ of $\Omega$ is algebraically closed, i.e., every $G(A)$-orbit on $\Omega \setminus A$ is infinite.
\end{enumerate}
\end{proposition}

\begin{proof}
See \cite[Lemma~1.6]{BJJ}.
\end{proof}

\subsection{Partial orders} \label{ss:oligo-poset}

We now return to partial orders. Let $\Omega$ be the Fra\"iss\'e limit of the class $\fP$; this is the universal homogeneous partial order. Let $G$ be its automorphism group. We note that $G$ acts transitively on $\Omega$ by homogeneity: all one element posets are isomorphic. Thus $G$ has no fixed points on $\Omega$.

\begin{proposition} \label{prop:poset-large}
The three conditions of Proposition~\ref{prop:large} hold.
\end{proposition}

\begin{proof}
Condition~(a) is proved in \cite[Lemma~7.7]{Jech}.
\end{proof}

\begin{remark}
In the nominal sets literature, the conditions of Proposition~\ref{prop:large} are closely related to the notion of least supports. In \cite[Theorem~6]{BKL}, it is stated that $(G, \Omega)$ admits least supports.
\end{remark}

As a consequence, we obtain a description of measures for $G$ in characteristic~0:

\begin{proposition} \label{prop:oligo-meas}
Measures for $G$ valued in a field $k$ of characteristic~0 are naturally in bijection with measures for $\fP$ valued in $k$.
\end{proposition}

\begin{proof}
This follows from Proposition~\ref{prop:poset-large}, the aforementioned \cite[Proposition~2.10]{distal}, and the isomorphism $\Theta(\fP) = \Theta(G, \sE(\Omega))$.
\end{proof}

We finish \S \ref{s:oligo} by translating Corollary~\ref{cor:nilp} to the group-theoretic setting. Recall that for a measure $\nu$ on $G$ we have a tensor category $\uPerm(G, \nu)$ by \cite[\S 8]{repst}. The basic objects of this category correspond to $G$-sets; we write $[X]$ for the object corresponding to $X \in \bS(G)$.

\begin{proposition}
Let $\nu$ be a measure for $G$ valued in a field of characteristic~0. Then nilpotent endomorphisms in $\uPerm(G, \nu)$ have trace~0.
\end{proposition}

\begin{proof}
Let $\nu'$ be the measure for $\fP$ corresponding to $\nu$. By \cite[\S 4.4]{arboreal}, we have an equivalence of tensor categories
\begin{displaymath}
\uPerm(\fP, \nu') = \uPerm(G, \sE(\Omega), \nu).
\end{displaymath}
Here $\uPerm(G, \sE(\Omega), \nu)$ is the full subcategory of $\uPerm(G, \nu)$ spanned by objects $[Y]$ where $Y$ is $\sE(\Omega)$-smooth. It follows from Corollary~\ref{cor:nilp} that nilpotent endomorphisms in these categories have trace~0.

Suppose $X$ is a transitive $G$-set. Write $X=G/U$ for an open subgroup $U$. By Proposition~\ref{prop:large}(b), there is a finite subset $A$ of $\Omega$ such that $G(A) \subset U \subset G\{A\}$. Put $Y=G/G(A)$ and $\Gamma=U/G(A)$. Since $G\{A\}$ is the normalizer of $G(A)$ \cite[Corollary~2.3]{distal}, we have $\Gamma \subset \Aut_G(Y)$ and $X=Y/\Gamma$. By the definition of the category $\uPerm(G, \nu)$, we find that $[X]$ is isomorphic to the invariant subobject $[Y]^{\Gamma}$ of $[Y]$. From this, we see that $\uPerm(G, \nu)$ and $\uPerm(G, \sE(\Omega), \nu)$ have the same Karoubi envelope, and so the result follows.
\end{proof}

\section{Knop-like measures} \label{s:knop}

Knop \cite{Knop} introduced the notion of \defn{degree function} on a (finitely powered) regular category $\cE$, and showed how they can be used to produce tensor categories. We \cite{regcat} showed that such a regular category $\cE$ has an associated pro-oligomorphic group $G$, and degree functions on $\cE$ give measures for $G$; this distinguishes a special subclass of \defn{Knop-like} measures on $G$. The category $\fP^+$ of finite posets and monotone maps is coregular. We determine the degree functions on it (or, technically, its opposite category), and determine which measures on $\fP$ are Knop-like.

\subsection{Background} \label{ss:knop}

Let $\cE$ be a category. A morphism $f \colon Y \to X$ in $\cE$ is a \defn{regular epimorphism} if it is the co-equalizer of a diagram $Z \rightrightarrows Y$; if $\cE$ has finite limits and $f$ is regular then it is the co-equalizer of its kernel pair, i.e., the diagram $Y \times_X Y \rightrightarrows Y$. We say that $\cE$ is \defn{regular} if it has finite limits, the kernel pair of any morphism admits a co-equalizer, and any base change of a regular epimorphism is regular. Any morphism $f$ in a regular category factors uniquely (up to isomorphism) as $i \circ s$ where $i$ is a monomorphism and $s$ is a regular epimorphism. We say that $\cE$ is \defn{finitely powered} if every object has only finitely many subobjects. See \cite[\S 2.1]{regcat} for additional details.

Suppose now that $\cE$ is finitely powered and regular. As usual, $k$ denotes our coefficient field (or even ring). The following definition is due to Knop \cite{Knop}:

\begin{definition} \label{defn:knop}
A \defn{degree function} on $\cE$ valued in $k$ is a rule $\kappa$ that assigns to each regular epimorphism $f \colon Y \to X$ a value $\kappa(f)$ in $k$ such that the following conditions hold:
\begin{enumerate}
\item $\kappa(f)=1$ if $f$ is an isomorphism.
\item $\kappa(g \circ f) = \kappa(g) \cdot \kappa(f)$, when defined.
\item Let $f \colon Y \to X$ be a regular epimorphism, let $X' \to X$ be an arbitrary morphism, and let $f' \colon Y' \to X'$ be the base change of $f$. Then $\kappa(f')=\kappa(f)$.
\end{enumerate}
\end{definition}

\begin{remark}
The rule $\kappa^{\rm triv}$ defined by $\kappa^{\rm triv}(\alpha)=1$ for any regular epimorphism $\alpha$ is a degree function. It is called the \defn{trivial degree function}.
\end{remark}

Suppose now, for simplicity, that the final object of $\cE$ has no proper subobjects. Let $\cE^s$ denote the category with the same objects of $\cE$ and where the morphisms are regular epimorphisms. In \cite[\S 4]{regcat}, we construct a pro-oligomorphic group $G$ and a stabilizer class $\sE$ such that $\cE^s$ is equivalent to the category of transitive $\sE$-smooth $G$-sets. Write $\fB(X)$ for the $G$-set corresponding to the object $X$ of $\cE$. If $\kappa$ is a degree function for $\cE$, we show that there is a measure $\nu$ for $G$ relative to $\sE$ given as follows: if $f \colon Y \to X$ is a regular epimorphism in $\cE$ then
\begin{displaymath}
\nu(f \colon \fB(Y) \to \fB(X)) = \sum_{Z \subset Y, f(Z)=X} \mu(Z,Y) \cdot \kappa(f \vert_Z).
\end{displaymath}
Here the sum is taken over subobjects $Z$ of $Y$ whose image under $f$ is $X$, and $\mu$ is the M\"obius function for the poset of subobjects of $Y$. It is not difficult to see that $\kappa \mapsto \nu$ is injective. Thus a certain subclass of measures for $G$ relative to $\sE$ are distinguished as coming from degree functions. We call these measures \defn{Knop-like}.

\begin{remark}
It is possible for the same group $G$ to come from multiple regular categories. In this case, the notion of Knop-like measure can depend on the choice of regular category. In fact, this happens in the case of posets, as we explain in Remark~\ref{rmk:multi-knop}.
\end{remark}

\subsection{The category of posets and monotone maps} \label{ss:poset-cat}

We say that a function $\alpha \colon P \to Q$ of posets is \defn{monotone} if $x \le y$ implies $\alpha(x) \le \alpha(y)$. Let $\fP^+$ be the category of finite posets with monotone maps. We note that monomorphisms in $\fP^+$ are exactly injective monotone maps, and epimorphisms are exactly surjective monotone maps. We require the following result.

\begin{proposition} \label{prop:coreg}
The category $\fP^+$ is coregular, i.e., the opposite category is regular. The regular monomorphisms in $\fP^+$ are exactly the embeddings of posets.
\end{proposition}

\begin{proof}
We break the proof into a number of pieces.

\textit{(a) Colimits.} We show that $\fP^+$ has all finite colimits. It is clear that $\emptyset$ is the initial object. It thus suffices to show that $\fP^+$ has push-outs. Consider monotone maps $\alpha \colon P \to Q$ and $\beta \colon P \to P'$. Let $R$ be the push-out in the category of finite sets, and let $\tilde{\beta} \colon Q \to R$ and $\tilde{\alpha} \colon P' \to R$ be the given functions. Let $\preceq$ be the transitive relation on $R$ generated by $\tilde{\beta}(b) \preceq \tilde{\beta}(b')$ whenever $b \le b'$ in $Q$, and $\tilde{\alpha}(c) \preceq \tilde{\alpha}(c')$ whenever $c \le c'$ in $P'$. The relation $\preceq$ is a quasi-order on $R$, that is, it is reflexive and transitive. There is an equivalence relation $\sim$ on $R$ given by $d \sim d'$ if $d \preceq d'$ and $d' \preceq d$. Define $Q'$ to be the quotient of $R$ by $\sim$; this is a partially ordered set under $\preceq$. Let $\beta' \colon Q \to Q'$ and $\alpha' \colon P' \to Q'$ be the natural maps. One easily sees that $\alpha'$ and $\beta'$ are monotone, and $\beta' \circ \alpha = \alpha' \circ \beta$.

We now show that $(Q', \alpha', \beta')$ is universal. Let $X$ be a partially ordered set and let $\sigma \colon Q \to X$ and $\tau \colon P' \to X$ be monotone maps such that $\sigma \circ \alpha = \tau \circ \beta$. Since $R$ is the push-out in the category of sets, there exists a unique map $\tilde{\phi} \colon R \to X$ such that $\tilde{\phi} \circ \tilde{\beta} = \sigma$ and $\tilde{\phi} \circ \tilde{\alpha} = \tau$. It is clear that $\tilde{\phi}$ is monotone (applying this term in the obvious sense for maps of quasi-orders). Thus equivalent elements of $R$ have the same image in $X$ (since $X$ is a partial order), and so $\tilde{\phi}$ factors uniquely through a map $\phi \colon Q' \to X$, which is necessarily monotone. This shows that $Q'$ is the push-out in $\fP^+$.

\textit{(b) Push-outs of embeddings.} Consider a push-out diagram in $\fP^+$
\begin{displaymath}
\xymatrix{
Q \ar[r]^{\beta'} & Q' \\
P \ar[r]^{\beta} \ar[u]^{\alpha} & P' \ar[u]_{\alpha'} }
\end{displaymath}
where $\alpha$ is an embedding. We show that $\alpha'$ is an embedding. Since any embedding can be factored into a sequence of one-point extensions, it suffices to treat the case where $\alpha$ is such a map. We thus assume $Q=P \sqcup \{x\}$. With notation as in (a), we have $R = P' \sqcup \{\tilde{y}\}$. The quasi-order $\preceq$ on $R$ is generated by the order $\le$ on $P'$ together with the relations $b \preceq \tilde{y}$ whenever $b=\beta(a)$ and $a \in P_{<x}$, and the similar relations $\tilde{y} \preceq b$ whenever $b=\beta(a)$ and $a \in P_{>x}$.

We claim that $\preceq$ restricts to the given order $\le$ on $P'$. Thus suppose $b \preceq b'$. We then have a chain $b=x_1 \preceq \cdots \preceq x_n=b'$, where $x_i \preceq x_i'$ is one of the generating relations. If any $x_i$ with $1<i<n$ belongs to $P'$ then we have $x_1 \le x_i$ and $x_i \le x_n$ by induction on the length of the chain, and so $b \le b'$ as required. It thus suffices to consider the case where $n=3$ and $x_2=\tilde{y}$; that is, we have $b=\beta(a)$ and $b'=\beta(a')$ and $a \le x \le a'$. But then $a \le a'$ and so $b \le b'$ since $\beta$ is monotone. This establishes the claim, from which it easily follows that $\alpha' \colon P' \to Q'$ is an embedding.

\textit{(c) Equalizers.} One easily sees that the equalizer of a diagram $P \rightrightarrows Q$ in $\fP^+$ is simply the equalizer in sets endowed with the induced order from $P$. In particular, all equalizers exist.

\textit{(d) Regular monomorphisms.} We claim that a monotone map $\alpha \colon P \to Q$ of posets is a regular monomorphism if and only if $\alpha$ is an embedding. Combined with (b), this shows that an arbitrary push-out of a regular monomorphism remains a regular monomorphism, which will complete the proof.

First suppose that $\alpha$ is a regular monomorphism; it is thus the equalizer of the diagram $Q \rightrightarrows Q \amalg_P Q$. By (c), it follows that $P$ is simply the usual equalizer with the induced order. Thus $\alpha$ is an embedding.

Now suppose $\alpha$ is an embedding. Let $P'$ be a copy of $Q$, let $R$ be the push-out of $P \to Q$ and $P \to P'$ in the category of sets, and let $\preceq$ be the quasi-order on $R$, as in (a). We saw in (b) that $\preceq$ restricts to $\le$ on $Q$ and $P'$. Suppose now $b \in Q$ and $c \in P'$ satisfy $b \preceq c$. We claim there exists $a \in P$ with $b \le a$ and $a \le c$. By definition, there is then a chain $b=x_1 \preceq \cdots \preceq x_n=c$, where each $x_i \preceq x_{i+1}$ is one of the generating relations. The claim is clear if $n=2$, for then $b \preceq c$ is one of the defining relations, and so either $b$ or $c$ belongs to $P$. Suppose $n>2$. If $x_2$ belongs to $Q$ then, by induction on the length of the chain, there is $a \in P$ such that $x_2 \le a$ and $a \le c$, and so $b \le a$ as well. If $x_2$ belongs to $P'$ the situation is similar. This establishes the claim.

Suppose now $b \sim c$, meaning $b \preceq c$ and $c \preceq b$. There are then elements $a,a' \in P$ such that $b \le a \le c \le a' \le b$. But this implies $a \le a'$ and $a' \le a$, and so $a=a'$, which in turn implies $b=c=a$. This shows that the equivalence relation on $R$ is discrete, and so $Q'=R$ is the push-out in $\fP^+$. It thus follows that the equalizer of $Q \rightrightarrows Q'$ is $P$, which shows that $\alpha$ is a regular monomorphism.
\end{proof}

We now analyze push-outs of one-point extensions in more detail.

\begin{proposition} \label{prop:push-one}
Consider a push-out square in $\fP^+$
\begin{displaymath}
\xymatrix{
Q \ar[r]^{\beta'} & Q' \\
P \ar[r]^{\beta} \ar[u]^{\alpha} & P' \ar[u]_{\alpha'} }
\end{displaymath}
where $\alpha$ is a one-point extension; note that $\alpha'$ is an embedding by Proposition~\ref{prop:coreg}. Identify $P$ and $P'$ with subsets of $Q$ and $Q'$, put $Q = P \sqcup \{x\}$, and let $y=\beta'(x)$. Let $E$ be the set $\beta(P_{<x}) \cap \beta(P_{>x})$.
\begin{enumerate}
\item Suppose $E$ is empty. Then $y \notin P'$, and so $\alpha'$ is a one-point extension. For $a \in P'$ we have $a<y$ if and only if there is some element $b \in P_{<x}$ such that $a \le \beta(b)$. There is an analogous condition for $y<a$.
\item Suppose $E$ is non-empty. Then $E=\{y\}$, and $\alpha'$ is an isomorphism of posets.
\end{enumerate}
\end{proposition}

\begin{proof}
Let $R = P' \sqcup \{\tilde{y}\}$. Let $\preceq$ be the reflexive and transitive relation on $R$ generated by the relation $\le$ on $P'$ together with the relations $\beta(a) \preceq \tilde{y}$ and $\tilde{y} \preceq \beta(b)$ for all $a \in P_{<x}$ and $b \in P_{>x}$. Let $\sim$ be the equivalence relation corresponding to $\preceq$. It follows from the general construction of push-outs in the proof of Proposition~\ref{prop:coreg} that $Q'$ is $R/\sim$ as a poset. Under this identification, $\alpha'$ and $\beta'$ are the natural maps, and $y$ is the class of $\tilde{y}$. We note that $\preceq$ restricts to the given partial order on $P'$ by the proof of Proposition~\ref{prop:coreg}. In particular, distinct elements of $P'$ are inequivalent.

Let $a \in P'$. We claim that $a \preceq \tilde{y}$ if and only if $a \le \beta(b)$ for some $b \in P_{<x}$. One direction is clear: if $a \le \beta(b)$ for such $b$ then $a \preceq \tilde{y}$, by definition of $\preceq$. Conversely, suppose $a \preceq \tilde{y}$. Then there exists a chain $a=z_1 \preceq \cdots \preceq z_r = \tilde{y}$, where each $z_i \preceq z_{i+1}$ is a generating relation. If $r=2$ the result is clear. Suppose $r>2$. We can assume that no $z_i$ with $1<i<r$ is equal to $\tilde{y}$. Thus $z_2$ belongs to $P'$. By induction on the length of the chain, we have $z_2 \le \beta(b)$ for some $b \in P_{<x}$. Since $a=z_1$ and $z_2$ belong to $P'$ and the order $\preceq$ restricts to the given order on $P'$, we have $a \le z_2 \le \beta(b)$, as required. This establishes the claim. Of course, there is an analogous statement for $\tilde{y} \preceq a$.

Let $a \in P'$ be equivalent to $\tilde{y}$. Then $a \preceq \tilde{y}$, and so $a \le \beta(b)$ for some $b \in P_{<x}$; similarly, $\tilde{y} \preceq a$, and so $\beta(c) \le a$ for some $c \in P_{>x}$. Since $b<x<c$ and $\beta$ is monotone, we have $\beta(b) \le \beta(c)$. But we also have $\beta(c) \le a \le \beta(b)$. Thus $a=\beta(b)=\beta(c)$ belongs to $E$. The reasoning is reversible, and so we see that $E$ is exactly the set of elements of $P'$ that are equivalent to $\tilde{y}$. Since distinct elements of $P'$ are equivalent, it follows that $E$ is either empty or a singleton. The result now follows.
\end{proof}

\begin{remark} \label{rmk:push-one}
Suppose we are in the setting of Proposition~\ref{prop:push-one}. If $x$ is maximal in $Q$ then $P_{>x}$ is empty. Thus we are in case (a), and $y$ is maximal in $Q'$. Similarly if $x$ is minimal or isolated.
\end{remark}

\subsection{Degree functions on posets}

We have seen (Proposition~\ref{prop:coreg}) that $\fP^+$ is a coregular category, and that $\fP$ is the wide subcategory\footnote{A wide subcategory is one that contains all objects.} whose morphisms are regular monomorphisms. The opposite category to $\fP^+$ is finitely powered: indeed, any epimorphism $P \to Q$ in $\fP^+$ satisfies $\vert Q \vert \le \vert P \vert$, and so a subobject $Q$ of $P$ in  $(\fP^+)^{\rm op}$ satisfies $\vert Q \vert \le \vert P \vert$. We can therefore speak of degree functions on $(\fP^+)^{\rm op}$. For simplicity, we will simply refer to these as degree functions on $\fP^+$.

To be clear, a degree function on $\fP^+$ is a rule $\kappa$ that assigns to each embedding $\alpha$ of finite posets a quantity $\kappa(\alpha)$ in $k$, subject to the conditions (a), (b), and the dual of (c) from Definition~\ref{defn:knop}. The final axiom states that whenever we have a push-out square
\begin{displaymath}
\xymatrix{
Q \ar[r]^{\beta'} & Q' \\
P \ar[u]^{\alpha} \ar[r]^{\beta} & P' \ar[u]_{\alpha'} }
\end{displaymath}
where $\alpha$ is an embedding, we must have $\kappa(\alpha) = \kappa(\alpha')$. We note that in the presence of axiom (b), it suffices to prove axiom (c) when $\alpha$ is a one-point extension.

In the remainder of \S \ref{s:knop}, we study degree functions on $\fP^+$.

\subsection{The max and min degree functions}

Let $\alpha \colon P \to Q$ be an embedding of finite posets, and identify $P$ with its image in $Q$. Put
\begin{displaymath}
\kappa^{\min}(\alpha) = \begin{cases}
1 & \text{if $\Min(Q) \subset P$} \\
0 & \text{otherwise.}
\end{cases}
\end{displaymath}
Note that $\Min(Q) \subset P$ is equivalent to $\Min(P) = \Min(Q)$. We define $\kappa^{\max}$ to be the transposed rule: precisely, $\kappa^{\max}(\alpha)$ is~1 if $\Max(Q) \subset P$, and 0 otherwise.

\begin{proposition}
$\kappa^{\min}$ and $\kappa^{\max}$ are degree functions.
\end{proposition}

\begin{proof}
It suffices to handle $\kappa^{\min}$. We verify the axioms of Definition~\ref{defn:knop}. Axiom (a) is clear.

(b) Now let $\alpha \colon P \to Q$ and $\beta \colon Q \to R$ be embeddings of posets. Suppose $\kappa^{\min}(\alpha) = \kappa^{\min}(\beta) = 1$. Then $\Min(P) = \Min(Q) = \Min(R)$, and so $\kappa^{\min}(\beta \circ \alpha) = 1$. Now suppose $\kappa^{\min}(\alpha) = 0$. Then there is an element $y \in \Min(Q)$ that does not belong to $P$. If $z$ is a minimal element of $R$ below $y$ then $z$ does not belong to $P$, and so $\kappa^{\min}(\beta \circ \alpha)=0$. If $\kappa^{\min}(\beta) = 0$ then there is an element of $\Min(R)$ that does not belong to $Q$, and so it certainly does not belong to $P$ either; thus $\kappa^{\min}(\beta \circ \alpha)=0$. This verifies the axiom.

(c) Consider a push-out square
\begin{displaymath}
\xymatrix{
Q \ar[r]^{\beta'} & Q' \\
P \ar[u]^{\alpha} \ar[r]^{\beta} & P' \ar[u]_{\alpha'} }
\end{displaymath}
where $\alpha$ is a one-point extension. Identify $P$ and $P'$ with subsets of $Q$ and $Q'$. Write $Q = P \sqcup \{x\}$ and $Q' = P' \cup \{y\}$, where $y=\beta'(x)$. We must show that $\kappa^{\min}(\alpha) = \kappa^{\min}(\alpha')$.

Suppose $\kappa^{\min}(\alpha)=1$. Then $x$ is not minimal in $Q$, and so there is an element $a \in P$ with $a<x$. We have $\beta(a) \le y$, and so either $y$ is not minimal or it belongs to $P'$. In either case, we find $\Min(Q') \subset P'$, and so $\kappa^{\min}(\alpha')=1$.

Now suppose that $\kappa^{\min}(\alpha)=0$. Then $x$ is a minimal element of $Q$, and so $y$ is a minimal element of $Q'$ that does not belong to $P'$ (Remark~\ref{rmk:push-one}). Thus $\kappa^{\min}(\alpha')=0$, which completes the proof.
\end{proof}

\subsection{A family of degree functions}

Let $P$ be a finite poset. Write $\Iso(P)$ for the set of isolated points in $P$. Define the \defn{proper boundary} of $P$ by
\begin{displaymath}
\partial_*(P) = (\Max(P) \cup \Min(P)) \setminus \Iso(P).
\end{displaymath}
Fix $t \in k$. Let $\alpha \colon P \to Q$ be an embedding of finite posets, and identify $P$ with its image in $Q$. Put
\begin{displaymath}
\kappa_t(\alpha) = \begin{cases}
t^{\vert \Iso(Q) \setminus P \vert} & \text{if $\partial_*(Q) \subset P$} \\
0 & \text{otherwise.} \end{cases}
\end{displaymath}
Here we use the convention that $t^0=1$; thus $\kappa_0(\alpha)=1$ if $\partial_*(Q) \subset P$ and $\Iso(Q) \subset P$. Before getting to the main result, we require an observation about the proper boundary.

\begin{lemma} \label{lem:prop-bd}
Let $P \to Q$ be an embedding of posets. Then $\partial_*(Q) \subset P$ if and only if $\partial_*(Q) = \partial_*(P)$.
\end{lemma}

\begin{proof}
If $\partial_*(Q) = \partial_*(P)$ then clearly $\partial_*(Q) \subset P$.

Now suppose $\partial_*(Q) \subset P$. Let $x \in \partial_*(Q)$, and without loss of generality suppose $x \in \Max(Q)$. Since $x \in P$ by assumption, it also belongs to $\Max(P)$. Since $x \notin \Iso(Q)$, there is an element $y \in \Min(Q)$ such that $y<x$. This $y$ also belongs to $\partial_*(Q)$, and thus to $P$, and so $x \notin \Iso(P)$. Thus $x \in \partial_*(P)$. We have thus shown $\partial_*(Q) \subset \partial_*(P)$.

Now suppose $x \in \partial_*(P)$, and without loss of generality, suppose $x \in \Max(P)$. Since $x$ is not isolated in $P$, it is also not isolated in $Q$. If $x$ were not maximal in $Q$ then we would have $x<y$ for some $y \in \Max(Q)$; but then this $y$ would belong to $\partial_*(Q) \subset P$, and so $x$ would not be maximal in $P$. Thus $x$ is maximal in $Q$, and therefore belongs to $\partial_*(Q)$. Hence $\partial_*(P) \subset \partial_*(Q)$.
\end{proof}

We are now ready to prove the main result.

\begin{proposition}
$\kappa_t$ is a degree function.
\end{proposition}

\begin{proof}
We verify the axioms of Definition~\ref{defn:knop}. Axiom (a) is clear.

(b) Now let $\alpha \colon P \to Q$ and $\beta \colon Q \to R$ be embeddings of posets. First suppose $\partial_*(R) \subset Q$ and $\partial_*(Q) \subset P$. Thus, by Lemma~\ref{lem:prop-bd}, we have $\partial_*(P) = \partial_*(Q) = \partial_*(R)$. Hence $\kappa_t(\beta \circ \alpha) = t^{\vert \Iso(R) \setminus P \vert}$. We must therefore show
\begin{displaymath}
\vert \Iso(R) \setminus P \vert = \vert \Iso(R) \setminus Q \vert + \vert \Iso(Q) \setminus P \vert.
\end{displaymath}
In fact, we claim that
\begin{displaymath}
\Iso(R) \setminus P = (\Iso(R) \setminus Q) \sqcup (\Iso(Q) \setminus P).
\end{displaymath}
The two sets on the right side are clearly disjoint. We also clearly have $\Iso(R) \setminus Q \subset \Iso(R) \setminus P$. Now suppose $x \in \Iso(Q) \setminus P$. If $x$ were not isolated in $R$, then $x$ would be comparable to some $y \in \partial_*(R)$; but since $\partial_*(R)=\partial_*(Q)$, this would show that $x$ is not isolated in $Q$, a contradiction. Thus $x \in \Iso(R) \setminus P$. We have thus shown that the right side above is contained in the left. Finally, consider $x \in \Iso(R) \setminus P$. If $x \notin Q$ then $x \in \Iso(R) \setminus Q$; if $x \in Q$ then it is clearly isolated in $Q$, and so $x \in \Iso(Q) \setminus P$. This verifies the claim.

Now suppose $\partial_*(R) \not\subset Q$ or $\partial_*(Q) \not\subset P$. Thus $\kappa_t(\beta)=0$ or $\kappa_t(\alpha)=0$. We must show that $\kappa_t(\beta \circ \alpha) = 0$. To do this, we show $\partial_*(R) \not\subset P$. If $\partial_*(R) \not\subset Q$ then this is clear. If $\partial_*(R) \subset Q$ but $\partial_*(Q) \not\subset P$ then $\partial_*(R) = \partial_*(Q)$ by Lemma~\ref{lem:prop-bd}, and so $\partial_*(R) \not\subset P$, as required.

(c) Consider a push-out square
\begin{displaymath}
\xymatrix{
Q \ar[r]^{\beta'} & Q' \\
P \ar[u]^{\alpha} \ar[r]^{\beta} & P' \ar[u]_{\alpha'} }
\end{displaymath}
where $\alpha$ is a one-point extension. Identify $P$ and $P'$ with subsets of $Q$ and $Q'$. Write $Q = P \sqcup \{x\}$ and $Q' = P' \cup \{y\}$, where $y=\beta'(x)$. We must show $\kappa_t(\alpha) = \kappa_t(\alpha')$. We consider two cases separately.

\textit{Case 1.} Suppose that $\partial_*(Q) \subset P$. This means $x \notin \partial_*(Q)$, so either $x \in \Iso(Q)$ or $x \notin \Min(Q) \cup \Max(Q)$. We claim that $\partial_*(Q') \subset P'$. If $x \in \Iso(Q)$ then $y \in \Iso(Q')$ (Remark~\ref{rmk:push-one}), and so $y \notin \partial_*(Q')$, and so $\partial_*(Q') \subset P'$. If $x \notin \Min(Q) \cup \Max(Q)$ then there are $a,b \in P$ such that $a<x<b$, and so $\beta(a) \le y \le \beta(b)$. Thus either $y \in P'$ or $y \notin \Min(Q') \cup \Max(Q')$, and so $\partial_*(Q') \subset P'$. This verifies the claim.

Suppose that $x$ is isolated. Then $y$ is also isolated and does not belong to $P'$ (Remark~\ref{rmk:push-one}). We thus find $\kappa_t(\alpha) = \kappa_t(\alpha') = t$, as required.

Suppose that $x$ is not isolated. Then there is some $a \in P$ such that $a<x$ or $x<a$. Thus $\beta(a) \le y$ or $y \le \beta(a)$, and so either $y \in P'$ or $y$ is not isolated. We thus find $\kappa_t(\alpha)=\kappa_t(\alpha')=1$, as required.

\textit{Case 2.} Suppose that $\partial_*(Q) \not\subset P$, meaning $x \in \partial_*(Q)$. Without loss of generality, suppose $x$ is maximal in $Q$. Then $y$ is maximal in $Q'$ and does not belong to $P'$ (Remark~\ref{rmk:push-one}). Since $x$ is not isolated in $Q$, there is $a \in P$ such that $a<x$, and so $\beta(a) < y$; this inequality is strict since $\beta(a) \in P'$ and $y \notin P'$. Thus $y$ is not isolated, and therefore belongs to $\partial_*(Q')$; hence $\partial_*(Q') \not\subset P'$. We therefore find $\kappa_t(\alpha) = \kappa_t(\alpha') = 0$, as required.
\end{proof}

\subsection{Classification of degree functions}

We now arrive at the main result of \S \ref{s:knop}. In what follows, we consider degree functions valued in the field $k$.

\begin{theorem}
The degree functions on $\fP^+$ are exactly $\kappa^{\rm triv}$, $\kappa^{\min}$, $\kappa^{\max}$, and $\kappa_t$, for $t \in k$.
\end{theorem}

\begin{proof}
Let $\kappa$ be a degree function. Put
\begin{displaymath}
a = \kappa(\{1\} \to \{1<2\}), \qquad
b = \kappa(\{2\} \to \{1<2\}), \qquad
c = \kappa(\emptyset \to \{1\}).
\end{displaymath}
Let $(P,x)$ be a marked poset. We claim
\begin{displaymath}
\kappa(P^x \to P) = \begin{cases}
1 & \text{if $P_{<x} \ne \emptyset$ and $P_{>x} \ne \emptyset$} \\
a & \text{if $P_{<x} \ne \emptyset$ and $P_{>x} = \emptyset$} \\
b & \text{if $P_{<x} = \emptyset$ and $P_{>x} \ne \emptyset$} \\
c & \text{if $P_{<x} = \emptyset$ and $P_{>x} = \emptyset$.}
\end{cases}
\end{displaymath}
Since any degree function is determined by its values on one-point extensions, this will show that $\kappa$ is completely determined by the triple $(a,b,c)$.

We now prove the claim. Consider the push-out
\begin{displaymath}
\xymatrix{
P \ar[r] & Q \\
P^x \ar[u] \ar[r] & \{y\} \ar[u] }
\end{displaymath}
where $y \notin P$. Note that $\kappa(P^x \to P)$ is equal to $\kappa(\{y\} \to Q)$. We compute $Q$ using Proposition~\ref{prop:push-one}. Start with $\tilde{Q} = \{x,y\}$. Let $\preceq$ be the quasi-order on $\tilde{Q}$ generated by $x \le y$ if $P_{>x}$ is non-empty and $y \le x$ if $P_{<x}$ is non-empty. Then $Q$ is the quotient of $\tilde{Q}$ by the associated equivalence relation. If $P_{>x}$ and $P_{<x}$ are both non-empty then $x$ and $y$ are equivalent, and the map $\{y\} \to Q$ is an isomorphism; this gives $\kappa(P^x \to P)=1$. If $P_{<x}$ is non-empty and $P_{>x}$ is empty, then $Q=\{y<x\}$, and the map $\{y\} \to Q$ is isomorphic to $\{1\} \to \{1<2\}$, and so $\kappa(P^x \to P)=a$. The transposed case is similar. Finally, if $P_{<x}$ and $P_{>x}$ are both empty then $Q=\{x,y\}$ is a discrete poset. The map $\{y\} \to Q$ is isomorphic to the push-out of $\emptyset \to \{x\}$ along $\emptyset \to \{y\}$, and so $\kappa(P^x \to P)=c$. This establishes the claim.

We now obtain some relations on the values $(a,b,c)$. Consider the inclusion $\alpha \colon \{1\} \to \{1,2,3\}$, where $1<2<3$. Factoring $\alpha$ through $\{1,2\}$, we find $\kappa(\alpha)=a^2$. Factoring $\alpha$ through $\{1,3\}$, we obtain $\kappa(\alpha)=a$. Thus $a=a^2$. Similarly, $b=b^2$. Thus $a,b \in \{0,1\}$.

Now consider the poset $P=\{x,y,z\}$ with $x<y$ and $x<z$, and consider the embedding $\alpha \colon \{y\} \to P$. Factoring through $\{x,y\}$, we obtain $\kappa(\alpha)=ab$, while factoring through $\{y,z\}$, we obtain $\kappa(\alpha)=bc$. Thus $ac=bc$. Taking the transpose, we find $ab=ac$ as well.

We thus have $a,b \in \{0,1\}$ and $ab=ac$ and $ab=bc$. These equations imply that $(a,b,c)$ is one of the following triples:
\begin{displaymath}
(1,1,1), \quad (1,0,0), \quad (0,1,0), \quad (0,0,t)
\end{displaymath}
where $t$ is arbitrary. These triples are realized by the degree functions $\kappa^{\rm triv}$, $\kappa^{\min}$, $\kappa^{\max}$, and $\kappa_t$ respectively. Thus $\kappa$ must be one of these degree functions, as required.
\end{proof}

\subsection{The associated measures} \label{ss:knoplike}

As we saw in \S \ref{ss:knop}, degree functions on a regular category give rise to measures on an associated oligomorphic group. In the case of $\fP^+$, we can phrase the construction in terms of measures on $\fP$. The following is the precise statement. By a \defn{quotient} of a poset $Q$ we mean a surjective monotone function $\beta \colon Q \to R$, though we often refer to $R$ itself as the quotient. The collection $\Quot(Q)$ of all quotients of $Q$ (taken up to isomorphism) forms a poset, with $Q$ as the maximal element.

\begin{proposition} \label{prop:degree-meas}
Let $\kappa$ be a degree function on $\fP^+$. Given an embedding of posets $\alpha \colon P \to Q$, define
\begin{displaymath}
\nu(\alpha) = \sum_{\substack{\beta \colon Q \twoheadrightarrow R, \\ \beta \circ \alpha \colon P \hookrightarrow R}} \mu(R, Q) \cdot \kappa(\beta \circ \alpha).
\end{displaymath}
Here we sum over elements $R \in \Quot(Q)$ such that $P \to Q \to R$ is an embedding of posets, and $\mu$ denotes the M\"obius function for $\Quot(Q)$. Then $\nu$ is a measure for $\fP$.
\end{proposition}

\begin{proof}
One can prove this directly by translating the arguments in \cite[Proposition~6.9]{regcat} from the language of oligomorphic groups to Fra\"iss\'e classes. Alternatively, one can directly apply \cite[Proposition~6.9]{regcat}, in the following manner. As we saw in \S \ref{ss:knop}, there is an oligomorphic group $G$ and stabilizer class $\sE$ such that the category of regular epimorphisms in $(\fP^+)^{\rm op}$ is equivalent to the category of transitive $\sE$-smooth $G$-sets. This $G$ is nothing other than the automorphism group of the universal homogeneous poset $\Omega$ (see \S \ref{ss:oligo-poset}), and $\sE=\sE(\Omega)$. Thus \cite[Proposition~6.9]{regcat} shows that a degree function on $\fP^+$ gives a measure on $G$ relative to $\sE$, which, in turn, is the same thing as a measure on $\fP$. The formula for the measure given in \cite[Construction~6.7]{regcat} is exactly our formula for $\nu$, after performing the necessary translations (passing from $(\fP^+)^{\rm op}$ to $\fP^+$, and converting from the group language to the Fra\"iss\'e class language).
\end{proof}

Fix a degree function $\kappa$ on $\fP^+$ and let $\nu$ be the measure associated to it by Proposition~\ref{prop:degree-meas}. We now determine the matrix $N$ of $\nu$. Put
\begin{displaymath}
\kappa_{p,q} = \kappa(J_{p,q} \setminus \{\ast\} \to J_{p,q})
\end{displaymath}
and let $K$ be the $3 \times 3$ matrix $(\kappa_{p,q})_{0 \le p,q \le 2}$.

\begin{proposition}
With the above notation, we have
\begin{displaymath}
N = K - A, \qquad
A = \begin{pmatrix}
0 & 1 & 0 \\
1 & 2 & 1 \\
0 & 1 & 0 \end{pmatrix}.
\end{displaymath}
\end{proposition}

\begin{proof}
We begin with a general observation. Let $\alpha \colon P \to Q$ be a one-point extension of posets, write $Q = P \sqcup \{x\}$, and assume that $Q_{\parallel x} = \emptyset$. Consider a proper quotient $\beta \colon Q \to R$ such that the composition $P \to Q \to R$ is an embedding. In general, a proper quotient in $\fP^+$ can be bijective; this happens if more pairs are comparable in the target than in the source. This cannot happen here, since $x$ is already comparable to all other elements of $Q$ and $\beta \vert_P$ is an embedding. Thus $\vert R \vert < \vert Q \vert$. It follows that the natural map $P \to R$ is an isomorphism, and so $\kappa(\beta \circ \alpha) = 1$. Moreover, in $\Quot(Q)$, the only element above $R$ is $Q$, and so $\mu(R,Q)=-1$. Applying the formula in Proposition~\ref{prop:degree-meas}, we find
\begin{displaymath}
\nu(\alpha) = \kappa(\alpha) - n,
\end{displaymath}
where $n$ is the number of retractions of $Q$ onto $P$, where a \defn{retraction} is a monotone map $Q \to P$ that is the identity on $P$. The result now follows upon observing that $A_{p,q}$ is the number of retractions of $J_{p,q}$ onto $J_{p,q} \setminus \{\ast\}$.
\end{proof}

Some simple calculations now give
\begin{displaymath}
\kappa^{\rm triv} \rightsquigarrow \CC_8, \qquad
\kappa^{\min} \rightsquigarrow \CC_7, \qquad
\kappa^{\max} \rightsquigarrow \CC_7^{\top}, \qquad
\kappa_t \rightsquigarrow \BB_5(t),
\end{displaymath}
where the right side indicates the matrix for the measure associated to the degree function on the left side. Thus the Knop-like measures are exactly those corresponding to these matrices.

\begin{remark}
The material in \S \ref{s:knop} gives an independent proof of Theorem~\ref{thm:some-meas}(b), i.e., that there is a measure with matrix $\CC_8$.
\end{remark}

\begin{remark}
Suppose $\kappa$ is a degree function with matrix $K$. As we have seen above, $K-A$ is the matrix of the associated measure. Curiously, each matrix $K$ is itself the matrix of a measure: $\kappa^{\rm triv}$ gives $\BB_1(1)$, $\kappa^{\min}$ gives $\BB_2(1)^{\top}$, and $\kappa_t$ gives $\AA(1,t)$. We do not have a conceptual explanation for this.
\end{remark}

\begin{remark} \label{rmk:multi-knop}
Let $\fQ^+$ be the following category: objects are finite posets that contain a least element, and morphisms are monotone maps that preserve the least element. One can show that $\fQ^+$ is coregular, and that regular monomorphisms are poset embeddings preserving the least element. Moreover, the wide subcategory $\fQ$ consisting of regular monomorphisms is equivalent to $\fP$, via the functor $\fP \to \fQ$ sending $P$ to $P \sqcup \{\bzero\}$, where $\bzero$ is declared to be least. Thus a degree function on $\fQ^+$ induces a measure on $\fP$, just as in Proposition~\ref{prop:degree-meas}. It turns out that $\fQ^+$ admits only two degree functions ($\kappa^{\rm triv}$ and $\kappa^{\max}$), and the Knop-like measures are just $\CC_7$ and $\CC_8$.
\end{remark}


\begin{thebibliography}{Sno2}

\bibitem[BJJ]{BJJ} R\'emi Barritault, Colin Jahel, Matthieu Joseph. On dissociated infinite permutation groups. \arxiv{2504.14057v1}

\bibitem[BKL]{BKL} Miko\l{}aj Boja\'nczyk, Bartek Klin, S\l{}awomir Lasota. Automata with group actions. In \textit{Proceedings of the 26th Annual IEEE Symposium on Logic in Computer Science (LICS)}, pp.~355--364. IEEE Computer Society, 2011. \DOI{10.1109/LICS.2011.48}

\bibitem[Cam]{Cameron} Peter J.\ Cameron. Oligomorphic permutation groups. London Mathematical Society Lecture Note Series, vol. 152, Cambridge University Press, Cambridge, 1990. \DOI{10.1017/CBO9780511549809}

\bibitem[HNS]{arboreal} Nate Harman, Ilia Nekrasov, Andrew Snowden. Arboreal tensor categories. \textit{Selecta Math. (N.S.)} \textbf{31} (2025), article~58. \DOI{10.1007/s00029-025-01058-1} \arxiv{2308.06660}

\bibitem[HSS]{line} Nate Harman, Andrew Snowden, Noah Snyder. The Delannoy category. \textit{Duke Math.\ J.} \textbf{173} (2024), no.~16, pp.~3219--3291. \DOI{10.1215/00127094-2024-0012} \arxiv{2211.15392}

\bibitem[HS]{repst} Nate Harman, Andrew Snowden. Oligomorphic groups and tensor categories. \arxiv{2204.04526}

\bibitem[Jec]{Jech} Thomas J.~Jech. The axiom of choice. North--Holland Publishing Co., Amsterdam--London; Amercan Elsevier Publishing Co., Inc., New York, 1973. Studies in Logic and the Foundations of Mathematics, Vol.~75.

\bibitem[Kno]{Knop} Friedrich Knop. Tensor envelopes of regular categories. \textit{Adv.\ Math.} \textbf{214} (2007), pp.~571--617. \DOI{10.1016/j.aim.2007.03.001} \arxiv{math/0610552}

\bibitem[Mac]{Macpherson} Dugald Macpherson. A survey of homogeneous structures. \textit{Discrete Math.} \textbf{311} (2011), no.~15, pp.~1599--1634. \DOI{10.1016/j.disc.2011.01.024}

\bibitem[Nek]{Nekrasov} Ilia Nekrasov. Tensorial measures on $\omega$-categorical structures. Ph.~D.\ dissertation, University of Michigan, 2024. \DOI{10.7302/23874}

\bibitem[NS]{distal} Ilia Nekrasov, Andrew Snowden. Upper bounds for measures on distal classes. \arxiv{2407.19131}

\bibitem[Sch]{Schmerl} James H.~Schmerl. Countable homogeneous partially ordered sets. \textit{Algebra Universalis} \textbf{9} (1979), pp.~317--321. \DOI{10.1007/BF02488043}

\bibitem[Sno1]{regcat} Andrew Snowden. Regular categories, oligomorphic monoids, and tensor categories. \arxiv{2403.16267}

\bibitem[Sno2]{homoperm} Andrew Snowden. The thirty-seven measures on permutations. \textit{Exp.\ Math.}, to appear. \arxiv{2404.08775}

\bibitem[Sno3]{webpage} Andrew Snowden. Poset calculator. \\ {\tiny\url{https://public.websites.umich.edu/~asnowden/notes/marked_poset_calculator.html}}

\end{thebibliography}
\end{document}